\documentclass{amsart}
\usepackage{dsfont}
\usepackage[usenames,dvipsnames,svgnames,table]{xcolor}
\usepackage[utf8]{inputenc}
\usepackage[T1]{fontenc}
\usepackage{ tipa }
\usepackage{helvet}
\usepackage{color}
\usepackage{mathrsfs}  
\usepackage{stmaryrd}  
\usepackage{amsmath,amsxtra,amsthm,amssymb,xr,enumerate,fullpage}
\usepackage[all]{xy}
\usepackage[bbgreekl]{mathbbol}
\usepackage{bbm}

\DeclareMathSymbol{\invques}{\mathord}{operators}{`>}
\DeclareUnicodeCharacter{00BF}{\tmquestiondown}
\DeclareRobustCommand{\tmquestiondown}{%
  \ifmmode\invques\else\textquestiondown\fi
}
\usepackage{verbatim}
\numberwithin{equation}{section}

\makeatletter
\newcommand{\mylabel}[2]{#2\def\@currentlabel{#2}\label{#1}}
\makeatother

\newtheorem{theorem}{Theorem}[section]
\newtheorem{lemma}[theorem]{Lemma}
\newtheorem{conj}[theorem]{Conjecture}
\newtheorem{proposition}[theorem]{Proposition}
\newtheorem{corollary}[theorem]{Corollary}
\newtheorem{defn}[theorem]{Definition}

\newtheorem{remark}[theorem]{Remark}

\newcommand{\cM}{\mathcal{M}}
\newcommand{\dd}{\operatorname{d}}

\newcommand{\Gal}{\operatorname{Gal}}

\newcommand{\sign}{\operatorname{sign}}

\newcommand{\NN}{\mathbb{N}}
\newcommand{\QQ}{\mathbb{Q}}
\newcommand{\Qp}{\mathbb{Q}_p}
\newcommand{\Zp}{\mathbb{Z}_p}
\newcommand{\ZZ}{\mathbb{Z}}
\renewcommand{\AA}{\mathbb{A}}
\newcommand{\FF}{\mathbb{F}}
\newcommand{\FFF}{\mathcal{F}}

\newcommand{\ord}{\mathrm{ord}}

\newcommand{\fp}{\mathfrak{p}}
\newcommand{\fq}{\mathfrak{q}}

\newcommand{\cL}{\mathcal{L}}

\newcommand{\cO}{\mathcal{O}}

\newcommand{\GL}{\mathrm{GL}}

\newcommand{\bR}{\mathbf{R}}

\newcommand{\image}{\mathrm{Im}}
\newcommand{\cyc}{\textup{cyc}}

\newcommand{\uIk}{u(\mathbb{I}_{\mathbf{k}\otimes\chi})}
\newcommand{\Ik}{\mathbb{I}_{\mathbf{k}\otimes\chi}}

\newcommand{\ff}{\mathfrak{f}}

\newcommand{\Hom}{\mathrm{Hom}}

\newcommand{\Char}{\mathrm{char}}
\newcommand{\Ind}{\mathrm{Ind}}
\newcommand{\ac}{\textup{ac}}
\newcommand{\LL}{\Lambda}
\newcommand{\TT}{\mathbb{T}}

\newcommand{\f}{\textup{\bf f}}

\newcommand{\lra}{\longrightarrow}
\newcommand{\ra}{\lra}
\newcommand{\res}{\textup{res}}

\newcommand{\Bf}{\mathbf{f}}

\newcommand{\fP}{\mathfrak{P}}
\newcommand{\BF}{\textup{BF}}
\newcommand{\cP}{\mathcal{P}}
\newcommand{\cF}{\mathcal{F}}

\newcommand{\p}{\mathfrak{p}}
\newcommand{\q}{\mathfrak{q}}

\newcommand{\m}{\mathfrak{m}}

\newcommand{\calF}{\mathcal{F}}

\newcommand{\LSU}{\mathcal{L}^{\rm SU}}
\newcommand{\LN}{\mathcal{L}^{\mathrm{N}}}

\newcommand{\cE}{\mathcal{E}}
\newcommand{\cR}{\mathcal{R}}

\newcommand{\fa}{\mathfrak{a}}

\newcommand{\cA}{\mathcal{A}}

\usepackage{hyperref}

\begin{document}

\title{I\lowercase{wasawa theory for}  $\mathrm{GL}_2\times \mathrm{Res}_{K/\mathbb{Q}}\mathrm{GL}_1$}

\begin{abstract}
Let $K$ be an imaginary quadratic field where the prime $p$ splits. Our goal in this article is to prove results towards the Iwasawa main conjectures for $p$-nearly-ordinary families associated to $\mathrm{GL}_2\times \mathrm{Res}_{K/\mathbb{Q}}\mathrm{GL}_1$ with a minimal set of assumptions. The main technical input is an improvement on the locally restricted Euler system machinery that allows the treatment of residually reducible cases, which we apply with the Beilinson--Flach Euler system.

\medskip

\noindent\textsc{R\'esum\'e.}
Soit $K$ un corps imaginaire quadratique où le nombre premier $p$ est déployé. Le but de cet article est de démontrer des résultats envers les conjectures principales d'Iwasawa pour des familles presque-$p$-ordinaires associées à $\mathrm{GL}_2\times \mathrm{Res}_{K/\mathbb{Q}}\mathrm{GL}_1$  sous un ensemble minimal d'hypothèses. La contribution technique principale est une amélioration sur la machinerie de systèmes d'Euler localement restrients qui permet le traitement des cas résiduellement réductibles, que nous appliquons avec le système d'Euler de Beilinson--Flach.
\end{abstract}

\author{K\^az\i m B\"uy\"ukboduk}
\address{K\^az\i m B\"uy\"ukboduk\newline UCD School of Mathematics and Statistics\\ University College Dublin\\ Ireland}
\email{kazim.buyukboduk@ucd.ie}

\author{Antonio Lei}
\address{Antonio Lei\newline
D\'epartement de Math\'ematiques et de Statistique\\
Universit\'e Laval, Pavillion Alexandre-Vachon\\
1045 Avenue de la M\'edecine\\
Qu\'ebec, QC\\
Canada G1V 0A6}
\email{antonio.lei@mat.ulaval.ca}

\dedicatory{Cet article est dédié à Bernadette Perrin-Riou dont les travaux nous ont grandement inspirés.}

\thanks{The first named author acknowledges partial support from European Union's Horizon 2020 research and innovation programme under the Marie  Skłodowska-Curie Grant Agreement No. 745691. The second named author's research is supported by the NSERC Discovery Grants Program  RGPIN-2020-04259 and RGPAS-2020-00096.}
\subjclass[2020]{11R23 (primary); 11F11, 11R20 (secondary)}
\keywords{Iwasawa theory, Rankin--Selberg products, locally restricted Euler systems, Beilinson--Flach elements}

\maketitle
\tableofcontents
\section{Introduction}
Let $p\ge 7$ be a prime number. Let $f\in S_{k_f+2}(\Gamma_1(N_f),\epsilon_f)$ be a cuspidal eigenform  which is not of CM type, where $N_f$ is coprime to $p$. We assume that $f$ admits a $p$-ordinary stabilization $f^{\alpha}$ with $U_p$-eigenvalue $\alpha_f$. We fix an imaginary quadratic field $K$ with discriminant $D_K$ which is coprime to $N_f$ and where $p$ splits with $(p)=\p\p^c$; here $c$ denotes any lift of a generator of $\Gal(K/\QQ)$ to $G_\QQ$. We also fix a ray class character $\chi$ of $K$ with conductor $\ff$. We note that there will be no new result in the present article that concerns the ``Eisentein case'' $\chi= \chi^c$.

Our goal in this paper is to prove  divisibility results on the main conjectures for $p$-nearly-ordinary families of automorphic motives on $\GL_2\times {\rm Res}_{K/\QQ}\GL_1$. In more explicit (but still very rough) terms, we will prove divisibilities in the Iwasawa Main Conjectures for families that interpolate Rankin--Selberg convolutions $f^{\alpha} \times \theta_{\chi\psi}$, where $\psi$ is a Hecke character of $K$ with $p$-power conductor and $\theta_{\chi\psi}$ is the theta-series associated $\chi\psi$.  These results should be thought of as evidence towards variational versions of Bloch--Kato conjectures for the relevant class of motives.

There are many earlier works in this direction: 
\begin{itemize}
\item The work of Bertolini--Darmon \cite{BertoliniDarmon2005} addresses the ``definite'' anticyclotomic Iwasawa main conjectures for the base change $f_{/K}$ of a modular form $f\in S_{2}(\Gamma_0(N_f))$ to $K$. The nearly-ordinary family in question interpolates the self-dual twists of the Rankin--Selberg motives associated to $f\times \theta_{\psi}$ as $\psi$ varies among Hecke characters of $K$ with $p$-power conductor.  
\item Chida and Hiseh in~\cite{chidahsiehanticyclomainconjformodformscomposito} generalized the work of  Bertolini--Darmon to study the  ``definite'' anticyclotomic Iwasawa main conjectures for the base change to $K$ of a modular form $f\in S_{k_f+2}(\Gamma_0(N_f))$ of arbitrary weight.
\item Howard in \cite{howardcompositio1} studied Perrin-Riou's Heegner point (``indefinite'' anticyclotomic) Main Conjectures for $f\in S_{2}(\Gamma_0(N_f))$; and in \cite{howard2007}, he initiated the study of ``indefinite'' anticyclotomic Main Conjectures which allow variation in Hida families. The nearly-ordinary family in question interpolates the self-dual twists of the Rankin--Selberg motives associated to $f^{\alpha}\times \theta_{\psi}$ as $\psi$ varies among Hecke characters of $K$ with $p$-power conductor and $f^{\alpha}$ in a Hida family. Works of Fouquet~\cite{Fouquet2013} and the first named author~\cite{KbbBigHeegner} relied on the technology developed in \cite{howard2007} to obtain results towards ``indefinite'' anticyclotomic Main Conjectures in the latter setting.
\item Skinner--Urban~\cite{skinnerurbanmainconj}, combined with the work of Kato \cite{kato04}, proved (under certain hypotheses) the Iwasawa Main Conjecture for the base change $f_{/K}$ of a modular form $f$ of any weight, along the $\ZZ_p^2$-tower $\Gamma_K$ of $K$. Their results also allowed variation in $f$. The nearly-ordinary family in question interpolates the Rankin--Selberg motives associated to $f^{\alpha}\times \theta_{\psi}$ as $\psi$ varies among Hecke characters of $K$ with $p$-power conductor and $f^{\alpha}$ in a Hida family. In the work of Skinner--Urban, the global root number for $f_{/K}$ is assumed to be $+1$ (this property is constant in Hida families).
\end{itemize}  
All the results above concern the case when $\chi=\chi^c$ is abelian over $\QQ$. In the more general setup, there has also been some progress, which is somewhat weaker than the case when $\chi=\chi^c$:
\begin{itemize}
    \item Under the hypothesis that $\chi$ is $p$-distinguished (which amounts to the requirement that the difference $\chi(\p)-\chi(\p^c)$ does not belong to the maximal ideal of the coefficient ring),  the present authors~\cite{BLForum} (general $k_f$) obtained results for the nearly-ordinary family that consists of Rankin--Selberg motives associated to $f\times \theta_{\chi\psi}$ as $\psi$ varies among Hecke characters of $K$ with $p$-power conductor (Castella~\cite{castellaJLMS} has proved similar results in the case where $k_f=0$). In \cite{BLForum}, we were able to descend to the anticyclotomic tower to treat definite and indefinite cases simultaneously.
    \item In some cases, the work of Wan~\cite{xinwanwanrankinselberg} complements the results of \cite{BLForum} to obtain a proof of Iwasawa Main Conjectures (up to $\mu$-invariants).    
    \item We note that in \cite{CastellaHsiehGHC}, Castella and Hsieh  have also obtained (non-variational) results towards Bloch--Kato Conjectures for certain Rankin--Selberg motives $f\times \theta_{\chi\psi}$ at the central critical point.
\end{itemize}
All these results will be recalled in the main body of the article and their comparisons to the new results we obtain in the present article will be made clear.

The main goal of the present article is to remove the $p$-distinguished hypothesis, which is present in all previous works which concern the ``cuspidal'' scenario $\chi\neq \chi^c$ (c.f. Theorems~\ref{intro_thm_BK_type_result}, \ref{intro_thm_ordinary_Deltai_MC_via_horizontal_ES}, \ref{intro_thm_ordinary_definite_ac} and \ref{intro_thm_ordinary_Delta1_Selmer_torsion_via_horizontal_ES_indefinite} for a quick overview of our results; see also the paragraph following Theorem~\ref{intro_thm_ordinary_Deltai_MC_via_horizontal_ES} for a discussion concerning the term $\mathscr{C}(u)$ and Remark~\ref{remark_why_nonEis} for the reason why we assume that $\chi\neq \chi^c$). Before we present our results in their precise form and explain our strategy, we first recall the role $p$-distinguished hypothesis on $\chi$ played in earlier works.

When the prime $p$ splits in $K/\QQ$ and the Hecke character $\chi$ is $p$-distinguished, one may construct a full-fledged Euler system over $K$ associated to the Rankin--Selberg product $f_{/K}\otimes\chi$. This relies on a ``patching'' argument introduced in \cite{LLZ2} when $f$ has weight $2$ and  in \cite{BLForum} when $f$ is of higher weight. The $p$-distinguished hypothesis is crucial when identifying the irreducible component of the eigencurve that contains the $p$-ordinary stabilization of $\theta_\chi$ (which necessarily has CM by $K$ thanks to the $p$-distinguished hypothesis, by a result of Bella\"iche and Dimitrov).

In the absence of the $p$-distinguished hypothesis, we are forced to work directly with the motives over $\QQ$  associated to Rankin--Selberg convolutions of the form $f\times \theta_{\chi\psi}$ (where the Hecke character $\psi$ will vary). This approach will require an extension of the locally restricted Euler system machinery, which we develop in Appendix~\ref{appendix_sec_ES_main} and it constitutes the technical backbone of our strategy. The difficulty stems from the fact that in the absence of  $p$-distinguished hypothesis on $\chi$, the residual Galois representation $\overline{\rho}_f\otimes{\rm Ind}_{K/\QQ}\overline{\chi}$ may well be reducible. Therefore, locally restricted Euler system machinery developed in \cite{kbbesrankr, kbbstick, kbbCMabvar, KBOchiai_1}, which ultimately relies on the theory of Kolyvagin systems~\cite{mr02} (where residual irreducibility is also essential), cannot be put to use in the present setting. 

Before we present our main results, we introduce some notation.
\subsection{Setup}
Let us fix an algebraic closure $\overline{\QQ}$ of $\QQ$ and fix embeddings $\iota_\infty: \overline{\QQ} \hookrightarrow \mathbb{C}$ and $\iota_p: \overline{\QQ} \hookrightarrow \mathbb{C}_p$ as well as an isomorphism $j:\mathbb{C}\stackrel{\sim}{\longrightarrow} \mathbb{C}_p$ in a way that the diagram
$$\xymatrix@R=.1cm{&\mathbb{C}\ar[dd]^{j}\\
\overline{\QQ}\ar[ur]^{\iota_\infty}\ar[rd]_{\iota_p} &\\
&\mathbb{C}_p
}$$
commutes. 

\begin{defn}[Abelian extensions of $K$]
\label{define_K_Gamma}
\item[i)] We let $K(p^n)/K$ denote the ray class extension of conductor $p^n$ and set $K({p^\infty}):=\cup_n K({p^n})$. We let $K_\infty/K$ denote the maximal pro-$p$ subextension in $K({p^\infty})/K$. We put $H_{p^\infty}:=\Gal(K(p^\infty)/K)$ and $\Gamma_K:=\Gal(K_\infty/K)$; note that $\Gamma_K\cong \ZZ_p^2$. 
\item[ii)] We let $K^\cyc \subset K_\infty$ denote the cyclotomic 
$\ZZ_p$-extension of $K$. We set $\Gamma_\cyc:=\Gal(K^\cyc/K)$.
\item[iii)] We let  $D_n/K$ denote the ring class extension of $K$ of conductor $p^n$ and let  $D_\infty$ denote the compositum of all $D_n$. Then $D_\infty/K$ contains a unique $\ZZ_p$-extension $K^\ac/K$, the anticyclotomic $\ZZ_p$-extension of $K$. We put $\Gamma_\ac:=\Gal(K^\ac/K)$.
\item[iv)] We let ${\rm Gr}(\Gamma_K)$  denote the set of torsion-free quotients $\Gamma$ of $\Gamma_K$. For $\Gamma\in{\rm Gr}(\Gamma_K)$, let $K_\Gamma$ denote the corresponding $\Zp$-power extension of $K$ with Galois group $\Gamma$. 
\item[v)] For any profinite abelian group $G$, we put $\LL(G)=\ZZ_p[[G]]:=\varprojlim_U \ZZ_p[G/U]$ where the projective limit is over all subgroups of $G$ with finite order. For any ring $R$ that contains $\ZZ_p$, we also set $\LL_{R}(G):=\LL(G)\otimes_{\ZZ_p}R$.
\item[vi)] We fix throughout a ray class character $\chi$ modulo $\ff$ over $K$, where $\ff$ is coprime to $p$.
\end{defn}
\begin{defn}
For each integer $m$, let us write $\QQ(m)$ to denote the maximal $p$-extension in $\QQ(\mu_m)/\QQ$.
\end{defn}
\begin{defn}[Eigenforms and $p$-stabilizations]
\item[i)]Let $f=\sum a_n(f)q^n$ be a cuspidal non-CM eigenform of weight $k_f+2$, level $N_f$ and nebentype $\epsilon_f$. We assume that $f$ is $p$-ordinary in the sense we make precise below. We also assume that the level $N_f$ of $f$ is coprime to $D_K\mathbf{N}\ff$.
\item[ii)] We let $L\subset \mathbb{C}_p$ denote a finite extension of $\QQ_p$ that contains the values of $j\circ \chi$ and also admits the Hecke field of $f$ as a subfield. We let $\cO$ denote its ring of integers, $\frak{m}_\cO$ its maximal ideal and $v_p$ the normalized valuation on $L$. Throughout, $\varpi$ is a fixed uniformizer of $L$. We assume that the Hecke polynomial $X^2-a_p(f)X+p^{k_f+1}\epsilon_f(p)$ splits in $L$ and has a root $\alpha_f$  such that $v_p(\alpha_f)=0$. We let $f^{\alpha}$ denote the (ordinary) $p$-stabilization of $f$ for which we have $U_pf^{\alpha}=\alpha_f f^{\alpha}$. We note that the choice of the field $L$ is rather unessential (so long as it verifies the properties we recorded above, which is always possible to ensure  so long as $a_p(f)$ is not divisible by all the primes of the Hecke field of $f$ lying above $p$) and we will assume  without notifying readers that it is enlarged as necessary to serve our purposes.
\end{defn}

\begin{defn}[The weight space and Hida's universal Hecke algebra]
\item[i)] We set $\LL_{\rm wt}:=\LL(\ZZ_p^\times)$ and put $[\cdot]: \ZZ_p^\times \hookrightarrow \LL(\ZZ_p^\times)^\times$ for the natural injection. We define the universal weight character $\bbkappa$ on setting $\bbkappa: G_\QQ\stackrel{\chi_\cyc}{\lra}\ZZ_p^\times\hookrightarrow \LL_{\rm wt}^\times$ (where $\chi_\cyc$ is the $p$-adic cyclotomic character). We say that a ring homomorphism $\kappa: \LL_{\rm wt} \lra \cO$ is an arithmetic specialization of weight $k+2 \in \ZZ$ to mean that the composition 
$$G_\QQ\stackrel{\bbkappa}{\lra} \LL_{\rm wt}^\times\stackrel{\kappa}{\lra}\cO$$
agrees with $\chi_\cyc^{k}$ on an open subgroup of $G_\QQ$. 
\item[ii)] For any integer $k$, we have set $\langle k \rangle:\LL(\ZZ_p^\times)\rightarrow \ZZ_p$ for the group homomorphism induced from the map $[x]\mapsto x^k$ on group-like elements $x$ of $\LL(\ZZ_p^\times)$.
Let us define $\LL_{\rm wt}^{(f)}\cong\LL(1+p\ZZ_p)$ as the component of $\LL_{\rm wt}$ that corresponds to the weight $k_f+2$, in the sense that $\LL_{\rm wt}^{(f)}$ is the smallest quotient of $\LL_{\rm wt}$ such that  $\langle k_f \rangle$ factors through $\LL_{\rm wt}^{(f)}$.
\item[iii)] We let $\f=\sum_{n=1}^{\infty} \mathbb{a}_{n}(\f)q^n \in \LL_\f[[q]]$ denote the branch of the primitive Hida family of tame conductor $N_f$, which admits $f^{\alpha}$ as a weight $k_f+2$ specialization. Here, $\LL_\f$ is the branch (i.e., the irreducible component) of the Hida's universal ordinary Hecke algebra determined by $f^{\alpha}$. It is finite flat over $\LL_{\rm wt}^{(f)}$ and the universal weight character $\bbkappa$ restricts to a character (also denoted by $\bbkappa$)
$$\bbkappa:  G_\QQ\lra \LL_{\rm wt}^\times \twoheadrightarrow \LL_{\rm wt}^{(f),\times}\lra \LL_\f^\times\,.$$
\item[iv)] For any $\Gamma\in {\rm Gr}(\Gamma_K)$, we set $\LL_{\f}(\Gamma):=\LL_\f\,\widehat{\otimes}\,\LL(\Gamma)$. If $R$ is a ring that contains $\Zp$, we write $\LL_{\f,R}(\Gamma)$ for the tensor product $\LL_\f(\Gamma)\otimes_{\Zp}R$.
\end{defn}

\begin{defn}[Galois representations attached to (families of) eigenforms]
\label{defn_twistedGalRep_Gamma}
\item[i)] Let $W_f\cong L^2$ denote Deligne's cohomological representation attached to $f$ and fix a $G_\QQ$-stable lattice $R_f\cong \cO^2$ inside $W_f$.
\item[ii)] When $\overline{\rho}_f$ is absolutely irreducible, there exists a free $\LL_\f$-module $R_\f^*$ of rank two, which is equipped with a continuous action of $G_{\QQ}$ unramified outside primes dividing $pN_f$ and which interpolates Deligne's representations $R_{\f(\kappa)}^*$ associated to arithmetic specializations $\f(\kappa)$ of the Hida family. 
\item[iii)] We have $\bigwedge^2 R_\f^*\stackrel{\sim}{\lra} \LL_\f(\bbkappa+1)\otimes\epsilon_f$ for the determinant of the $G_\QQ$-representation $R_\f^*$. 
\item[iv)] If we assume in addition that $f$ is $p$-distinguished, then the $G_{\QQ_p}$-representation $R_\f^*$ admits a $G_{\QQ_p}$-stable direct summand $F^+R_\f^*$ of rank one. 
\end{defn}

\begin{defn}[Self-dual twists]
\label{def_self_dual_twist_Hida}
\item[i)] The universal weight character $\bbkappa$ admits a ``square root''; namely, there exists a character $\bbkappa^{\frac{1}{2}}:  G_\QQ\lra \LL_\f^\times$ with the property that $\bbkappa=(\bbkappa^{\frac{1}{2}})^2$.  When $\epsilon_f=\mathds{1}$, we set $R_\f^\dagger:=R_{\f}^*(-\bbkappa^{\frac{1}{2}})$ and call it the self-dual twist of $R_{\f}^*$. 
\item[ii)] Notice that $R_\f^\dagger$ specializes in weight $k_f+2$ to $R_{f}^{\dagger}:=R_f^*(-k_f/2)$. We call $R_f^\dagger$ the self-dual twist of $R_f^*$.
\end{defn}
\begin{remark}
Instead of requiring that  $\epsilon_f=\mathds{1}$, we could equally work under the weaker hypothesis that $\epsilon_f$ admits a square-root, in the sense that there exists a Dirichlet character $\eta_f$ with $\epsilon_f=\eta_f^2$. In that scenario, we would  put $R_\f^\dagger:=R_{\f}^*(-\bbkappa^{\frac{1}{2}}\eta_f^{-1})$ so that $R_\f^\dagger$ is self-dual in the sense that there exists an isomorphism
$$R_\f^\dagger\stackrel{\sim}{\lra}\Hom_{\LL_\f}(R_\f^{\dagger},\LL_\f)(1)\,.$$
\end{remark}
We next define the Galois representations which will arise as the $p$-adic (\'etale) realizations of the Rankin--Selberg motives associated to $f \times \theta_{\chi\psi}$.

\begin{defn}[Galois representations attached to Rankin--Selberg convolutions $f_{/K}\otimes\chi$ and families]
\label{def_intro_Gal_rep_RS}
\item[i)] We define the $G_K$-representation $T_{f,\chi\psi}:=R_f^*\otimes\chi\psi$ for any Hecke character $\psi$ of $K$.  We similarly put $T_{\f,\chi\psi}:=R_\f^*\otimes\chi\psi$.
\item[ii)] When $\chi=\chi^c$ and $\epsilon_f=\mathds{1}$, we define $T_{f,\chi\xi}^{\dagger}:=R_f^\dagger\otimes\chi\xi$ for any character $\xi$ of $\Gamma_\ac$. We similarly put $T_{\f,\chi\xi}^\dagger:=R_\f^\dagger\otimes\chi\xi$.  

\item[iii)] For  $\Gamma\in\mathrm{Gr}(\Gamma_K)$, we write $\LL(\Gamma)^\sharp$ for the free $\LL(\Gamma)$-module of rank one equipped with the tautological action of $G_K$. For each such $\Gamma$ other than $\Gamma_\ac$, we define 
$$\TT_{f,\chi}^{(\Gamma)}:=T_{f,\chi}\otimes\LL(\Gamma)^\sharp$$
and we put $\TT_{f,\chi}^{(\Gamma_\ac)}:=T_{f,\chi}^\dagger\otimes\LL(\Gamma_\ac)^\sharp\,$. We also set $\TT_{f,\chi}^{\dagger}:=T_{f,\chi}^{\dagger}\otimes \LL(\Gamma_K)^{\sharp}$.
\item[iv)] When $\Gamma=\Gamma_K,\Gamma_\cyc$ or $\Gamma_\ac$, we shall write $\TT_{f,\chi}$, $\TT_{f,\chi}^{\cyc}$ or $\TT_{f,\chi}^{\ac}$ in place of $\TT_{f,\chi}^{(\Gamma_K)}$, $\TT_{f,\chi}^{(\Gamma_\cyc)}$ or $\TT_{f,\chi}^{(\Gamma_\ac)}$, respectively.
\item[v)] We similarly define the $G_K$-representations $\TT_{\f,\chi}^{(\Gamma)}$, $\TT_{\f,\chi}$, $\TT_{\f,\chi}^{\cyc}$,  $\TT_{\f,\chi}^{\ac}$ and $\TT_{\f,\chi}^{\dagger}$  attached to $\f$ and $\chi$.
\end{defn}

\begin{remark}
\label{remark_why_central_critical_twists}
\item[i)] We explain why we prefer to work with $T_{f,\chi\xi}^{\dagger}$ instead of $T_{f,\chi\xi}$ in the anticyclotomic setting. When $\chi^c=\chi^{-1}$ and $\varepsilon_f=\mathds{1}$, the representations $T_{f,\chi\xi}^\dagger$ are conjugate self-dual in the sense that there are $G_K$-isomorphisms 
$T_{f,\chi\xi}^{\dagger}\stackrel{\sim}{\lra} {\rm Hom}_{\ZZ_p}\left(T_{f,\chi^c\xi^c}^{\dagger},\ZZ_p(1) \right).$ 
Those Hecke characters $\psi$ such that the $G_K$-representation $T_{f,\chi\psi}^\dagger$ is conjugate self-dual are precisely the anticyclotomic Hecke characters.
\item[ii)] It follows from the discussion above that we have a natural conjugate self-duality isomorphism
$$\TT_{f,\chi}^{\ac}\stackrel{\sim}{\lra}\Hom_{\LL(\Gamma_\ac)}\left(\TT_{f,\chi^c}^{\ac,\iota},\LL(\Gamma_\ac)\right)(1)\,$$
where $\TT_{f,\chi^c}^{\ac,\iota}:=\TT_{f,\chi^c}^{\ac}\otimes_{\LL(\Gamma_\ac)}\LL(\Gamma_\ac)^{\iota}$ and $\LL(\Gamma_\ac)^\iota$ is the free $\LL(\Gamma_\ac)$-module of rank one on which $G_K$ acts via the composition $G_K\lra \Gamma_\ac\stackrel{\gamma\mapsto\gamma^{-1}}{\lra}\Gamma_\ac$.
\end{remark}
\begin{defn}
For any Hecke character $\psi$ of $K$, we define the $G_\QQ$-representation $X_{f,\chi\psi}^\circ:=R_f^{*}\otimes {\rm Ind}_{K/\QQ}\chi\psi$ and $X_{f,\chi\psi}=X_{f,\chi\psi}^\circ \otimes_{\ZZ_p}\QQ_p$. We similarly put $X_{\f,\chi\psi}^\circ:=R_\f^{*}\otimes {\rm Ind}_{K/\QQ}\chi\psi$. Finally,  for any $\Gamma\in {\rm Gr}(\Gamma_K)$, let us put $X_{\f,\chi,\Gamma}^{\circ}:={\rm Ind}_{K/\QQ}\, R_{\f}^*(\chi)\widehat{\otimes}_{\ZZ_p}\LL(\Gamma)^\sharp$.
\end{defn}

\subsection{Main results}
\label{subsec_intro_results}
Throughout \S\ref{subsec_intro_results}, we assume that 
\begin{enumerate}
   \item[\mylabel{item_RIa}{{\bf ($\hbox{RI}_a$)}}] $\chi\not\equiv \chi^c \mod \frak{m}_\cO^{a+1}$ for some integer $a\in \ZZ_{\geq 0}$.
\end{enumerate}
This hypothesis will be in effect in all our main results. Notice that  the existence of $a\in \ZZ_{\geq 0}$ is equivalent to the assumption that $\chi\neq \chi^c$. We introduce this hypothesis to highlight the dependence of our results on the integer $a$.

Consider the following condition.
\begin{enumerate}
   \item[\mylabel{item_tau}{{\bf ($\tau_{\f}$)}}] $p\geq 7$ and 
   that ${\rm SL}_2(\FF_p)\subset  \overline{\rho}_\f(G_{\QQ(\mu_{p^\infty})})$.
\end{enumerate}

\begin{remark}
 Thanks to \cite[Theorem C.2.4]{BLInertOrdMC}, when the condition \ref{item_tau} holds,  there  exists an element $\tau\in G_{\QQ(\mu_{p^\infty})}$ such that $X_{g,\chi\psi}^\circ/(\tau-1)X_{g,\chi\psi}^\circ$ is a free $\cO$-module of rank one, for any crystalline specialization $g$ of the Hida family $\f$ and Hecke character $\psi$.
\end{remark}

The discrete Bloch--Kato Selmer groups, denoted by $H^1_{\mathcal{F}^*_{{\rm BK}^{(i)}}}(K,A_{f,\chi\psi}^*(1))$ ($i=1,2$), that make an appearance in Theorems~\ref{intro_thm_Split_BK_Sigma12} are introduced in Definition~\ref{defn_aux_Selmer_groups}(ii) below. 
We also refer the reader to Definition~\ref{defn_criticality_1_2} for the definitions of $\Sigma^{(i)}_{\rm crit}$ ($i=1,2$).

\begin{theorem}[Theorem~\ref{thm_Split_BK_Sigma12}]
\label{intro_thm_Split_BK_Sigma12} 
We assume the validity of the condition \ref{item_tau}. 
\item[i)] Suppose that $\psi \in \Sigma^{(1)}_{\rm crit}$ is a $p$-crystalline character with infinity type $(\ell_1,\ell_2)$. We  have  
\begin{align*}
    {\ord_p}\left(H^1_{\mathcal{F}^*_{{\rm BK}^{(1)}}}(K,A_{f,\chi\psi}^*(1))\right)\leq \left[H^1(K_\p,F^-R_f^*(\chi\psi)):\cO\cdot \res_\p^{(1)}\left(\BF^{f\otimes\chi\psi}_{\alpha_f,\alpha_{\chi\psi}}\right)\right]+t\,.
\end{align*}
for some constant $t$ that depends only on $a$ and the residual representation $\overline{X}_{f,\chi\psi}$.

Moreover, one may take $t=0$ if $a=0$.

\item[ii)] Suppose that $\psi \in \Sigma^{(2)}_{\rm crit}$ is $p$-crystalline. We have  
\begin{align*}
{\ord_p}\left(H^1_{\mathcal{F}^*_{{\rm BK}^{(2)}}}(K,A_{f,\chi\psi}^*(1))\right)\leq \left[H^1(K_{\p^c},F^+R_f^*(\chi\psi)):\cO\cdot \res_{\p^c}^{(2)}\left(\BF^{f\otimes\chi\psi}_{\alpha_f,\alpha_{\chi\psi}}\right)\right]+t\,.
\end{align*}
for some constant $t$ that depends only on $a$ and the residual representation $\overline{X}_{f,\chi\psi}$. 

Once again, one may take $t=0$ if $a=0$.
\end{theorem}
 The proof of Theorem~\ref{intro_thm_Split_BK_Sigma12} relies on the locally restricted Euler system machinery, which we develop in Appendix~\ref{appendix_sec_ES_main} in the setting where the residual representation may not be irreducible. The general statement we prove in this vein is Theorem~\ref{thm_appendix_ES_main} and it may be of independent interest.   
 
 \begin{remark}
 \label{remark_why_nonEis}
 When $\chi=\chi^c$ (in other words, when \ref{item_RIa} fails for any $a$), Theorem~\ref{intro_thm_Split_BK_Sigma12} is still valid for all  $\psi \in \Sigma^{(i)}_{\rm crit}$ ($i=1,2$) such that $\psi\not\equiv \psi^c \mod \frak{m}_{\cO}^{b+1}$ for some natural number $b$ (with the error term $t$ determined by $b$). For Iwasawa theoretic applications (c.f. Theorem~\ref{intro_thm_ordinary_Deltai_MC_via_horizontal_ES}), we would like to vary $\psi$ over a large set of Hecke characters, in a manner to ensure that the error term $t$ can be chosen uniformly as $\psi$ varies. Since $t$ depends only on $b$ (and it potentially grows with $b$), one needs to fix a natural number $b$ and work with the collection of characters $\psi$ with $\psi\not\equiv \psi^c \mod \frak{m}_{\cO}^{b+1}$.
 
 However, the $p$-adic Galois characters associated to $\psi$ in this collection could not be dense in ${\rm Sp}(\LL(\Gamma_K))(\cO)$ (e.g., they cannot approximate the trivial character) and as a result, our patching argument would break down. This is the sole reason why we assume that $\chi\neq \chi^c$ throughout this work.
 \end{remark}

The conclusions of Theorem~\ref{intro_thm_Split_BK_Sigma12} and the relation of Beilinson--Flach elements to $L$-values combined together give the following result.
\begin{theorem}[Theorem~\ref{thm_BK_type_result}]
 \label{intro_thm_BK_type_result}
Suppose $f \in S_{k_f+2}(\Gamma_1(N_f),\epsilon_f)$ is a cuspidal eigen-newform and $\chi$ is a ray class character of $K$. Let $\psi \in \Sigma_{\rm crit}^{(i)}$ be an algebraic Hecke character of $p$-power conductor, which is critical for the eigenform $f$. Suppose either that $\chi\neq \chi^c$ or that $\psi\neq \psi^c$. Assume in addition that the hypothesis \ref{item_tau} holds true.  If $L(f_K\otimes\chi^{-1}\psi^{-1},0)\neq 0$, then the Selmer group $H^1_{\mathcal{F}^*_{{\rm BK}^{(i)}}}(K,A_{f,\chi\psi}^*(1))$ has finite cardinality.
\end{theorem}

In the statement of Theorem~\ref{intro_thm_Split_BK_Sigma12}, one may vary $\psi$ among critical Hecke characters and $f^{\alpha}$ in the Hida family $\f$ to prove results towards Iwasawa Main Conjectures. This is achieved using the patching argument \cite[Proposition A.0.1]{BLInertOrdMC}. The extended Selmer groups, denoted by $\widetilde{H}^2_{\rm f}(-)$ and that make an appearance in Theorems~\ref{intro_thm_ordinary_Deltai_MC_via_horizontal_ES}--\ref{intro_thm_ordinary_Delta1_Selmer_torsion_via_horizontal_ES_indefinite}, are introduced in Definitions~\ref{defn_Selmer_complex_split_ord} and~\ref{def_Nekovar_for_families} below. The comparison of the Iwasawa theoretic extended Selmer group and their classical counterparts is given in Lemma~\ref{lemma_compare_Iwasawa_Nek_Iwasawa_Greenberg_K_infty} (see also Lemma~\ref{lemma_NEK_global_duality_compare_BK} for a comparison for Selmer groups over $K$). The reason why we prefer to state our Iwasawa theoretic results in terms of Nekov\'a\v{r}'s Selmer complexes is because we would like appeal (in \S\ref{subsec_descent_to_Gamma} and onward) to the descent formalism developed in \cite{nekovar06}.

\begin{theorem}[Theorem~\ref{thm_ordinary_Deltai_MC_via_horizontal_ES}]
\label{intro_thm_ordinary_Deltai_MC_via_horizontal_ES}
Suppose $f \in S_{k_f+2}(\Gamma_1(N_f),\epsilon_f)$ is a cuspidal eigen-newform and $\chi$ is a ray class character of $K$ verifying \ref{item_RIa}  with $a=0$. We let $\f$ denote the unique branch of the Hida family that admits $f$ as a specialization in weight $k_f+2$. Suppose that \ref{item_tau} holds and $\overline{\rho}_{\f}$ is $p$-distinguished. 
\item[i)] Let $L_p(f_{/K}\otimes\chi,\Sigma^{(i)}_{\rm crit})$ $(i=1,2)$ denote Hida's pair of $p$-adic $L$-functions given as in Definition~\ref{defn_Hidas_padic_Lfunctions} below. For some $a_1\in \ZZ$ which only depends on $\overline{\rho}_\f$ and $\chi$ we have
$$\Char_{\LL_{\cO}(\Gamma_K)}\left(\widetilde{H}^2_{\rm f}(G_{K,\Sigma},\TT_{f,\chi};\Delta^{(1)})\right)\,\Big{|}\,\,  {\varpi}^{a_1}\LL_{\cO}(\Gamma_K) L_p(f_{/K}\otimes\chi,\Sigma^{(1)}_{\rm crit}){\big \vert}_{\Gamma_K}
\,.$$
Moreover, for some $a_2\in \ZZ$ which only depends on $\overline{\rho}_\f$ and $\chi$ we have
$$\Char_{\LL_{\cO}(\Gamma_K)}\left(\widetilde{H}^2_{\rm f}(G_{K,\Sigma},\TT_{f,\chi};\Delta^{(2)})\right) {\otimes_\cO\cO_\Phi}\,\Big{|}\,\,  {\varpi}^{a_2}\LL_{\cO{_\Phi}}(\Gamma_K)H_\chi L_p(f_{/K}\otimes\chi,\Sigma^{(2)}_{\rm crit}){\big \vert}_{\Gamma_K}
\,.$$
Here, $H_\chi \in {\LL_{\cO{_\Phi}}(\Gamma_\ac)}
$ is a generator of the congruence module associated the CM form $\theta_\chi$, where $\Phi=L\widehat{\QQ_p^{\rm ur}}$ and $\cO_\Phi\subset \Phi$ is the ring of integers.

\item[ii)] Suppose in addition that $\LL_\f$ is regular. Then,
$$\Char_{\LL_\f(\Gamma_K)}\left(\widetilde{H}^2_{\rm f}(G_{K,\Sigma},\TT_{\f,\chi};\Delta^{(1)})\right)\,\Big{|}\,\,  {\varpi}^{b_1}\LL_\f(\Gamma_K)H_{\f}L_p(\f_{/K}\otimes\chi,{\mathbb \Sigma}^{(1)}){\big \vert}_{\Gamma_K}
$$
for some $b_1\in \ZZ$, and
$$\Char_{{\LL_{\f,\cO_\Phi}}(\Gamma_K)}\left(\widetilde{H}^2_{\rm f}(G_{K,\Sigma},\TT_{\f,\chi};\Delta^{(2)})\right)\,\Big{|}\,\,  {\varpi}^{b_2}{\LL_{\f,\cO_\Phi}}(\Gamma_K)H_\chi L_p(\f_{/K}\otimes\chi,{\mathbb \Sigma}^{(2)}){\big \vert}_{\Gamma_K}
$$
for some $b_2\in \ZZ$. Here, $H_\f \in \LL_{\f}$ is a generator of Hida's congruence ideal and $L_p(\f_{/K}\otimes\chi,{\mathbb \Sigma}^{(i)})$, $(i=1,2)$ are Hida's pair of $p$-adic $L$-functions given as in Theorem~\ref{thm_Hida_3var_ord_split} below.
\end{theorem}

 We actually prove a more general version of Theorem~\ref{intro_thm_ordinary_Deltai_MC_via_horizontal_ES} in the main body of the present article (c.f. Theorem~\ref{thm_ordinary_Deltai_MC_via_horizontal_ES}), where we do not assume $a=0$. In order to study the general case where $a>0$, an error term $\mathscr{C}(u)$ (c.f. \S\ref{subsubsec_local_properties_CMgg} for its definition) appears on the right-hand side of the four divisibilities in the statement of the theorem. This error term may be dropped if an extra hypothesis concerning the CM branches of Hida families (labelled \ref{item_Ind}  in the main text) holds. We note that the main motivation of this article was to establish this general version, which required us to develop, as in Appendix~\ref{appendix_sec_ES_main}, a locally restricted Euler system machinery in the residually reducible scenario.


We expect that the divisibilities in Theorem~\ref{intro_thm_ordinary_Deltai_MC_via_horizontal_ES} can be upgraded in the near future to equalities, as suitable extensions of the results in \cite{xinwanwanrankinselberg} become available. 

In \S\ref{subsec_split_indefinite_definite_ord}, we explain how to apply the general descent formalism Nekov\'a\v{r} has developed in \cite{nekovar06} combined with our results here to obtain the following   $\LL_{\cO}(\Gamma_\ac)$-adic and $\LL_{\f}(\Gamma_\ac)$-adic Birch and Swinnerton-Dyer type formulae. We assume until the end of \S\ref{subsec_intro_results} that $\epsilon_f=\mathds{1}$ and $\chi^c=\chi^{-1}$. 

Suppose $R$ is a complete local Krull domain and $F\in R[[\gamma_\cyc-1]]$ (resp., $J\subset R[[\gamma_\cyc-1]]$ is an ideal). The definition of the \emph{mock leading term} $\partial_\cyc^*F \in R$ (resp., $\partial_\cyc^*J \in R$, which is defined only up to units of $R$) that make an appearance below is given in Definition~\ref{def_mock_leading_term} (resp., in  Definition~\ref{def_mock_leading_term_2}).

\begin{theorem}[Theorem~\ref{thm_big_BSD_step1}]
\label{intro_thm_big_BSD_step1} Suppose $N_f^-$ is square-free, $\chi^2\neq \mathds{1}$ and also that $\overline{\rho}_{\f}$ is $p$-distinguished. Assume that the hypothesis \ref{item_tau}  holds.
\item[i)]  We have the disivibility
$$\partial_{\cyc}^{*}\Char_{\LL_\cO(\Gamma_K)}\left(\widetilde{H}^2_{\rm f}(G_{K,\Sigma},\TT_{f,\chi}^\dagger;\Delta^{(1)})\right)={\rm Reg}_{\TT_{f,\chi}^\ac}\cdot \Char_{\LL_\cO(\Gamma_\ac)}\left(\widetilde{H}^2_{\rm f}(G_{K,\Sigma},\TT_{f,\chi}^\ac;\Delta^{(1)})_{{\rm tor}}\right)\,.$$
\item[ii)] Assume in addition that $\LL_\f$ is regular. Then,
\begin{align*}\partial_{\cyc}^{*}\Char_{\LL_\f(\Gamma_K)}\left(\widetilde{H}^2_{\rm f}(G_{K,\Sigma},\TT_{\f,\chi}^\dagger;\Delta^{(1)})\right)={\rm Reg}_{\TT_{\f,\chi}^\ac}\cdot \Char_{\LL_\f(\Gamma_\ac)}\left( \widetilde{H}^2_{\rm f}(G_{K,\Sigma},\TT_{\f,\chi}^\ac;\Delta^{(1)})_{{\rm tor}}\right)\,.
    \end{align*}
\end{theorem}

Combining this algebraic variants of $\LL$-adic BSD formulae with Theorem~\ref{intro_thm_ordinary_Deltai_MC_via_horizontal_ES}, we deduce the following statements towards anticyclotomic Iwasawa Main Conjectures.

\begin{theorem}[Theorem~\ref{thm_ordinary_definite_ac}, Main Conjectures of Definite Anticyclotomic Iwasawa Theory]
\label{intro_thm_ordinary_definite_ac} 
Suppose that $N_f^-$ is a square-free product of an even number of primes, 
\ref{item_RIa} holds with $a=0$, and that $\overline{\rho}_{\f}$ is $p$-distinguished. Assume that the hypothesis \ref{item_tau} holds.
\item[i)]  We have the divisibility $$\Char_{\LL_{\cO}(\Gamma_\ac)}\left( \widetilde{H}^2_{\rm f}(G_{K,\Sigma},\TT_{f,\chi}^\ac;\Delta^{(1)})\right)\otimes\QQ_p\,\, \Big{|}\,\,\, \LL_{L}(\Gamma_\ac)\cdot L_p(f_{/K}\otimes\chi,  {\Sigma}^{(1)}_{\rm cc})\big{\vert}_{\Gamma_\ac}
\,.$$
\item[ii)] Assume in addition that $\LL_\f$ is regular. Then,
$$\Char_{\LL_{\f}(\Gamma_\ac)}\left( \widetilde{H}^2_{\rm f}(G_{K,\Sigma},\TT_{\f,\chi}^\ac;\Delta^{(1)})\right)\otimes\QQ_p\,\, \Big{|}\,\,\, \LL_{L}(\Gamma_\ac)\cdot H_{\f}L_p(\f_{/K}\otimes\chi,  \mathbb{\Sigma}^{(1)}_{\rm cc})\big{\vert}_{\Gamma_\ac} 
\,.$$
\end{theorem}

The twisting morphisms ${\rm Tw}:\LL(H_{p^\infty})\lra \LL(H_{p^\infty})$ and ${\rm Tw}_\f: \LL_\f(H_{p^\infty})\lra \LL_\f(H_{p^\infty})$ which appear in Theorem~\ref{intro_thm_ordinary_Delta1_Selmer_torsion_via_horizontal_ES_indefinite} below are introduced in Definitions~\ref{def_twisting_morphism} and \ref{def_twisting_morphism_Hida}.

\begin{theorem}[Theorem~\ref{thm_ordinary_Delta1_Selmer_torsion_via_horizontal_ES_indefinite}, Main Conjectures of Indefinite Anticyclotomic Iwasawa Theory]
\label{intro_thm_ordinary_Delta1_Selmer_torsion_via_horizontal_ES_indefinite} Suppose that $N_f^-$ is a square-free product of an even number of primes, 
\ref{item_RIa} holds with $a=0$,  and that $\overline{\rho}_{\f}$ is $p$-distinguished. Assume that the hypothesis \ref{item_tau} holds.
\item[i)]  Both $\LL_{\cO}(\Gamma_\ac)$-modules $\widetilde{H}^1_{\rm f}(G_{K,\Sigma},\TT_{f,\chi}^\ac;\Delta^{(1)})$ and $\widetilde{H}^2_{\rm f}(G_{K,\Sigma},\TT_{f,\chi}^\ac;\Delta^{(1)})$ have rank one. Moreover, we have a containment
    $$ 
    \partial_\cyc^* \left({\rm Tw}\left(L_p(f_{/K}\otimes\chi,  {\Sigma}^{(1)})\right){\big{\vert}_{\Gamma_K}}\right)\,\, \subset \,\,{\rm Reg}_{\TT_{f,\chi}^\ac}\cdot\Char_{\LL_{\cO}(\Gamma_\ac)}\left( \widetilde{H}^2_{\rm f}(G_{K,\Sigma},\TT_{f,\chi}^\ac;\Delta^{(1)})_{\rm tor}\right)\otimes\QQ_p\,.$$
\item[ii)] Assume in addition that $\LL_\f$ is regular. Then, $\widetilde{H}^1_{\rm f}(G_{K,\Sigma},\TT_{\f,\chi}^\ac;\Delta^{(1)})$ and $\widetilde{H}^2_{\rm f}(G_{K,\Sigma},\TT_{\f,\chi}^\ac;\Delta^{(1)})$ have rank one as $\LL_{\f}(\Gamma_\ac)$-modules. Moreover, we have a containment
    $$(\LL_{\f}(\Gamma_\ac)\otimes_{\ZZ_p}\QQ_p) H_{\f}
    \partial_\cyc^* \left({\rm Tw}_\f\left(L_p(\f_{/K}\otimes\chi,  \mathbb{\Sigma}^{(1)})\right){\big{\vert}_{\Gamma_K}}\right)\,\subset\, {\rm Reg}_{\TT_{\f,\chi}^\ac}\cdot\Char_{\LL_{\f}(\Gamma_\ac)}\left( \widetilde{H}^2_{\rm f}(G_{K,\Sigma},\TT_{\f,\chi}^\ac;\Delta^{(1)})_{\rm tor}\right)\otimes\QQ_p\,.$$
    

\end{theorem}

Similar to Theorem~\ref{intro_thm_ordinary_Deltai_MC_via_horizontal_ES}, more general versions of Theorems~\ref{intro_thm_ordinary_definite_ac} and \ref{intro_thm_ordinary_Delta1_Selmer_torsion_via_horizontal_ES_indefinite} without assuming $a=0$ are proved below (c.f. Theorems~\ref{thm_ordinary_definite_ac} and \ref{thm_ordinary_Delta1_Selmer_torsion_via_horizontal_ES_indefinite}), at the expense of introducing an error term $\mathscr{C}^{\rm ac}(u)$  (see Definition~\ref{defn_C_ac_u} for its definition) on the right-hand side of the divisibilities in Theorem~\ref{intro_thm_ordinary_definite_ac} and on the left-hand side of the inclusions of Theorem~\ref{intro_thm_ordinary_Delta1_Selmer_torsion_via_horizontal_ES_indefinite} respectively.

These results simultaneously extend the previous works by Bertolini--Darmon, Castella, Castella--Wan, Castella--Hsieh Chida--Hsieh, Howard, Fouquet, Skinner--Urban and the present authors that we have mentioned above. In the main body of our article, we have included variants of these results that were obtained in earlier works, so as to both allow a comparison with known results and also to offer a picture that is as complete as possible describing the progress so far towards the Main Conjectures for ${\rm GL}_2\times {\rm Res}_{K/\QQ}{\rm GL}_1$.

Another benefit of the current approach is that it allows the study of the far more challenging case when the prime $p$ remains inert in $K/\QQ$. This is the subject of the {sequel~\cite{BLInertOrdMC}.}
\subsection{Overview} We summarize the contents of the current article. 
\begin{itemize}
    \item We introduce various collections of Hecke characters which are \emph{critical} in a suitable sense as well as theta-series which come attached to them in \S\ref{sec_critical_Hecke_chars}.  We review Beilinson--Flach elements for Rankin--Selberg convolutions of cuspidal eigenforms with the theta-series of these Hecke characters in \S\ref{sec_BF_elements}. 
    \item After introducing the Selmer complexes and Hida's $p$-adic $L$-functions which we shall repeatedly use in this article in \S\ref{sec_Selmer_complexes_padic_L_fucntions}, we formulate the Iwasawa Main Conjectures for $\GL_2\times{\rm Res}_{K/\QQ}\GL_1$ in \S\ref{subsec_IMC_fchi} and discuss previous work in certain cases (for most part, due to Castella, Skinner--Urban, Wan and the present authors).
    \item  In \S\ref{subsubsec_IwThe1_ord_families}, we allow variation in the non-CM factor in the Rankin--Selberg convolutions in question and formulate the Main Conjectures in that set up. We review earlier works for families (due to Skinner--Urban, Wan and Castella--Wan) in \S\ref{subsubsec_IwThe1_ord_families}  and move on to prove the the main results of the current article in \S\ref{sec_results_split}.
    \item In \S\ref{subsec_BK_conjectures}, we present our results towards Bloch--Kato conjectures.  Our argument therein relies on the general machinery we develop in Appendix~\ref{appendix_sec_ES_main}, which might be of independent interest. 
    \item We combine these results together with the criterion for divisibility we prove in \cite[Appendix~A]{BLInertOrdMC} to prove our Iwasawa theoretic results over the $\ZZ_p^2$-extension of $K$ in \S\ref{subsec_split_ord}. 
    \item After reviewing Nekov\'a\v{r}'s general descent formalism in \S\ref{subsubsec_Nek_descent}, we apply his results in \S\ref{subsubsec_descent_to_Gamma} to obtain our Iwasawa theoretic results for any $\ZZ_p$-extension of $K$.
    \item In the particular case where the $\ZZ_p$-extension in question is the anticyclotomic $\ZZ_p$-extension of $K$, we combine Nekov\'a\v{r}'s general descent formalism and our results in \S\ref{subsec_split_ord} with the work of Chida--Hsieh and Hsieh to the anticyclotomic main conjectures. Our results that concern the definite anticyclotomic case are proved in \S\ref{subsubsec_split_definite_ord}, and those that indefinite anticyclotomic case (where we obtain a $\LL_\f(\Gamma_\ac)$-adic Birch and Swinnerton-Dyer type formula) are in \S\ref{subsubsec_split_indefinite_ord}. 
    \item We finish our paper with Appendix~\ref{appendix:corrige} where we correct an inaccuracy in~\cite{BLForum} pertaining to the interpolation formulae for Hida's $p$-adic $L$-functions presented in op. cit. 
\end{itemize}

\subsubsection*{Acknowledgements} The authors thank Ashay Burungale, Robert Pollack, Christopher Skinner and Preston Wake for enlightening conversations on various technical subtleties. They would also like to thank the anonymous referee for his/her very constructive and valuable comments and suggestions on an earlier version of the article.


\section{Critical Hecke characters, theta series and CM branches Hida families}
\label{sec_critical_Hecke_chars}
\subsection{Critical Hecke characters}
In this section where we introduce certain sets of Hecke characters, our discussion  follows \cite[\S4]{bertolinidarmonprasanna13} very closely. As in the introduction, let us fix a cuspidal eigenform $f$ of weight $k_f+2$ as well as a ray class character $\chi$ modulo $\ff$ and of finite order. We also assume that the level $N_f$ of $f$ is coprime to $pD_K\mathbf{N}\ff$. Recall that we have assumed $\chi\neq \chi^c$, in particular $\chi$ verifies \ref{item_RIa} for some natural number $a$. Fix forever such an $a$.
\begin{defn}
\label{defn_criticality}
\item[i)] We say that a Hecke character $\psi$ of $K$ with $p$-power conductor is \textbf{critical} (relative to the eigenform $f$ and the character $\chi$ we have fixed) if $s=1$ is a critical value for the $L$-function $L(f_{/K}\otimes\chi^{-1}\psi^{-1},s)$. We shall denote the set of critical Hecke characters by $\Sigma_{\rm crit}$. 
\item[ii)] Given $\psi\in \Sigma_{\rm crit}$ with infinity-type $(\ell_1,\ell_2)$, let us put $\ell:=|\ell_1-\ell_2|$ and $\ell_0:=\min(\ell_1,\ell_2)$.
\end{defn}
We note that $\psi\in\Sigma_{\rm crit}$ if and only if $\psi|\cdot|$ belongs to $\Sigma$ of \cite[\S4.1]{bertolinidarmonprasanna13}. We prefer to work with this shift since we will work with the homological (rather than cohomological) Galois representations attached to modular forms.

It follows from the discussion in \cite[\S4.1]{bertolinidarmonprasanna13} that a Hecke character $\psi$ with $p$-power conductor and infinity-type $(\ell_1,\ell_2)$ is critical if and only if one of the following holds:
\begin{itemize}
    \item $0\leq \ell_1,\ell_2\leq k_f$.
    \item $\ell_1\geq k_f+1$ and $\ell_2\leq -1$ or $\ell_2\geq k_f+1$ and $\ell_1\leq -1$.
\end{itemize}

\begin{defn}
\label{defn_criticality_1_2}
We define $\Sigma_{\rm crit}^{(1)}\subset \Sigma_{\rm crit}$ as the set of Hecke characters $\psi$ whose infinity-type $(\ell_1,\ell_2)$ verifies
$$0\leq \ell_1,\ell_2\leq k_f\,.$$
We define $\Sigma_{\rm crit}^{(2)}\subset \Sigma_{\rm crit}$ as the set of Hecke characters $\psi$ whose infinity-type $(\ell_1,\ell_2)$ verifies
$$\ell_1\geq k_f+1 \hbox{ and } \ell_2\leq -1; \hbox{ or } \ell_2\geq k_f+1 \hbox{ and } \ell_1\leq -1.$$
\end{defn}

\begin{remark}
The set $\Sigma_{\rm crit}$ is determined by the weight $k_f+2$ of the eigenform $f$. When we need to emphasize this dependence, we shall write $\Sigma_{\rm crit}(k_f)$ in place of $\Sigma_{\rm crit}$. We similarly use the notation  $\Sigma_{\rm crit}^{(1)}(k_f)$, $\Sigma_{\rm crit}^{(2)}(k_f)$ wherever appropriate. 
\end{remark}

\begin{defn}
\label{defn_central_criticality_1_2}
We say that $\psi \in \Sigma_{\rm crit}$ is central critical if $\ell_1+\ell_2=k_f$ and $\chi\psi\big{\vert}_{\mathbb{A}_\QQ^\times}\cdot |\cdot|^{-k_f}=\epsilon_f$. For each $i\in \{1,2\}$, we let $\Sigma^{(i)}_{\rm cc}\subset \Sigma^{(i)}_{\rm crit}$ denote the collection of central critical characters and we put $\Sigma_{\rm cc}=\Sigma^{(1)}_{\rm cc} \bigsqcup \Sigma^{(2)}_{\rm cc}$.
\end{defn}
\begin{remark}
\item[i)] Since we assume that $\ff$ is coprime to $pN_f$, the requirement that $\chi\psi\big{\vert}_{\mathbb{A}_\QQ^\times}\cdot |\cdot|^{-k_f}=\epsilon_f$ implies that 
$\chi\big{\vert}_{\mathbb{A}_\QQ^\times}=\mathds{1}$. In other words, the set $\Sigma_{\rm cc}$ is non-empty only when the character $\chi$ is anticyclotomic. When this is the case, $\psi\in \Sigma_{\rm cc}$ if and only if $\psi\big{\vert}_{\mathbb{A}_\QQ^\times}\cdot |\cdot|^{-k_f}=\epsilon_f$. In particular, the set $\Sigma_{\rm cc}$ is determined by the weight $k_f+2$ of $f$ and its nebentype $\epsilon_f$. We shall still write $\Sigma_{\rm cc}(k_f)$, $\Sigma_{\rm cc}^{(1)}(k_f)$ and $\Sigma_{\rm cc}^{(2)}(k_f)$ whenever we would like to emphasize this dependence (but will not record the dependence on $\epsilon_f$).
\item[ii)] Suppose $c \mid N_f$ is the conductor $\epsilon_f$ and $\psi\in \Sigma_{\rm cc}(k_f)$, so that  $\psi\big{\vert}_{\mathbb{A}_\QQ^\times}= |\cdot|^{k_f}\epsilon_f$. In particular, the conductor of $\psi$ is divisible by $c$. 

In what follows, we will be interested in the $G_K$-representation $T_{f,\chi\psi}:=R_f^*\otimes\chi\psi$ and for technical reasons (which are relevant to the use of Beilinson--Flach elements), we will need to assume that the conductor of $\chi\psi$ is coprime to $N_f$. The discussion above tells us that $\psi\in \Sigma_{\rm cc}(k_f)$ with this additional requirement exists if $c=1$; namely, only if $\epsilon_f=\mathds{1}$ and $f\in S_{k_f+2}(\Gamma_0(N_f))$. 

From now on, whenever $\Sigma_{\rm cc}(k_f)$ is mentioned, we will always assume that $\epsilon_f$ is the trivial character.
\end{remark}

\begin{defn}
\label{def_twisting_morphism}
\item[i)] We let ${\rm Tw}: \LL(H_{p^\infty})\rightarrow \LL(H_{p^\infty})$ denote the twisting morphism induced by $\gamma \mapsto \chi_\cyc^{-k_f/2}(\gamma)\gamma$, where we have defined (by slight abuse) $\chi_\cyc$ to denote the composition of the arrows $$\chi_\cyc: H_{p^\infty}\twoheadrightarrow \Gal(K(\mu_{p^\infty})/K)\stackrel{\chi_\cyc}{\lra} \ZZ_p^\times.$$
Given a $p$-adic character $\Xi: H_{p^\infty}\rightarrow \mathbb{C}_p$, we let ${\rm Tw}(\Xi)$ denote the character given as the composition
$${\rm Tw}(\Xi): H_{p^\infty}\hookrightarrow \LL(H_{p^\infty}) \stackrel{{\rm Tw}}{\lra}\LL(H_{p^\infty}) \stackrel{\Xi}{\lra}\mathbb{C_p}\,.$$
More explicitly, ${\rm Tw}(\Xi)=\Xi\,\chi_\cyc^{-k_f/2}$. 
\item[ii)] If $\xi$ is a Hecke character of $K$, we define the Hecke character ${\rm Tw}(\xi)=\xi|\cdot|^{k_f/2}$. The $p$-adic avatar $\xi_p$ of $\xi$ verifies $\left({\rm Tw}(\xi)\right)_p=\xi_p\,\chi_\cyc^{-k_f/2}={\rm Tw}(\xi_p)$. The map ${\rm Tw}$ is a bijection on the set of Hecke characters of $K$; we denote its inverse by ${\rm Tw}^{-1}$. It is clear that ${\rm Tw}^{-1}(\xi)=\xi|\cdot|^{-k_f/2}$.
\end{defn}

\begin{remark} 
\label{rem_cc_vs_Galois}
 Suppose that $\epsilon_f=\mathds{1}$ and let $\psi \in \Sigma_{\rm cc}(k_f)$ be a Hecke characters with $p$-power conductor. The Galois character associated to the $p$-adic avatar $\psi_p\chi_\cyc^{k_f/2}$ of the Hecke character ${\rm Tw}^{-1}(\psi)=\psi|\cdot|^{-k_f/2}$ factors through $G_{p^\infty}:=\Gal(D_\infty/K)$, where we recall that $D_\infty:=\varinjlim D_n$ is the compositum of ring class extension of $K$ of conductor a power of $p$. 

Conversely, suppose $\xi_p$ is a continuous character of $G_{p^\infty}$, then the associated $p$-adic Hecke character (which we still denote by $\xi_p$) of $K$ has $p$-adic type $(a,-a)$. We say that $\xi_p$ is algebraic if $a\in \ZZ$. When that is the case, $\xi_p$ appears as the $p$-adic avatar of a Hecke character $\xi$ of infinity type $(a,-a)$, for which we have  $\xi\vert_{\mathbb{A}_\QQ^\times}=\mathds{1}$. Moreover, we have ${\rm Tw}(\xi)=\xi|\cdot|^{k_f/2} \in \Sigma_{\rm cc}(k_f)$ with $p$-power conductor.

In summary, there is a natural bijection between $\Sigma_{\rm cc}(k_f)$ and continuous algebraic characters of $G_{p^\infty}$.\end{remark}

\begin{remark}
  The compositum $D_{\infty}$  is a finite extension of the anticyclotomic $\ZZ_p$-extension $K^{\ac}$ of $K$. Recall that $H_{p^\infty}$ denotes the ray class group of $K$ modulo $p^\infty$. We have natural surjections $H_{p^\infty}\twoheadrightarrow G_{p^\infty}\twoheadrightarrow \Gamma_\ac$.
\end{remark}

\begin{remark}
\label{rem_Sigma_Here_BDP_Forum}
{Our $\Sigma^{(i)}_{\rm crit}$ agrees with the set denoted by $\Sigma^{(i)}$ in \cite[Definition 4.1]{bertolinidarmonprasanna13}. Moreover, $\xi \in \Sigma^{(i)}_{\rm crit}$ if and only if $\xi^{-1}|\cdot|^{k_f/2}$ belongs to the set $\Sigma^{(i)}$ in \cite{BLForum}.} 
\end{remark}
\subsubsection{Theta series}
\begin{defn}
\label{defn_theta_series}
Let $\xi$ be a Hecke character of $K$ of conductor $\ff$ and infinity-type $(\ell_1,\ell_2)$.
\item[i)]  We put $\xi_0:=\xi^{-1}|\cdot|^{\ell_1}$ so that $\xi_0$ is a Hecke character of infinity type $(0,\ell_1-\ell_2)$. 

\item[ii)] We define the theta series $\theta_{\xi_0}$ of $\xi_0$ on setting
\begin{align*}
    \theta_{\xi_0}&=\sum_{(\frak{a},\ff)=1}\xi_0(\frak{a})q^{\mathbf{N}(\frak{a})}\\
    &=\sum_{(\frak{a},\ff^c)=1}\xi_0^c(\frak{a})q^{\mathbf{N}(\frak{a})}\in M_{\ell+1}(\Gamma_1(N_\xi),\epsilon_\xi)
\end{align*}
is an eigenform of weight $\ell+1$, level $N_{\xi}:=|D_K|\,\mathbf{N}\ff$ and nebentype $\epsilon_\xi:=|\cdot|^{\ell_2-\ell_1}\xi_0\vert_{\mathbb{A}_\QQ^\times}\,\epsilon_K$. 
\end{defn}
Unless $\ell=0$ and $\xi=\Xi\circ \mathbb{N}_{K/\QQ}$ for some character of $\mathbb{A}_{\QQ}^\times$, the eigenform $\theta_{\xi_0}$ is cuspidal (see \cite[Theorem 4.8.2]{miyake06}). When $\ell\neq 0$, we let $R_{\theta_{\xi_0}}^*$ denote the cohomological lattice in Deligne's repesentation associated to $\theta_{\xi_0}$. We define the canonical lattice $R_{\theta_{\xi_0}}^*$ inside the Deligne--Serre representation when $\ell=0$ as the specialization of the big Galois representation $R_{\theta_{\mathbb{g}_{\xi_0}}}^*$ associated to the canonical Hida family $\mathbb{g}_{\xi_0}$ (c.f. \S\ref{subsec_CMHida} for details).

\begin{lemma}
\label{lemma_lattices_theta_pointwise}
Suppose $\ell\neq 0$. We have an isomorphism
\begin{equation}
\label{eqn_lattice_pointwise_p_inverted}
    R_{\theta_{\xi_0}}^*\otimes_{\Zp}\Qp\stackrel{\sim}{\lra}\left({\rm Ind}_{K/\QQ}\,\xi_{0,p}^{-1}\right)\otimes_{\Zp}\Qp=\left({\rm Ind}_{K/\QQ}\,\xi_p\right)(\ell_1)\otimes_{\Zp}\Qp\,.
\end{equation}
where we write $\psi_p$ to denote the $p$-adic avatar of the Hecke character $\psi$ and abusively denote the free $\ZZ_p$-module of finite rank on which $G_K$ acts via this character by the same symbol. If $R_{\theta_{\xi_0}}^*$ is residually irreducible (i.e. whenever the Hecke character ${\xi}_{0,p}$ verifies \ref{item_RIa} with $a=0$), then \eqref{eqn_lattice_pointwise_p_inverted} is induced from an isomorphism 
\begin{equation}
\label{eqn_lattice_pointwise_integral_res_irred}
    R_{\theta_{\xi_0}}^*\stackrel{\sim}{\lra}{\rm Ind}_{K/\QQ}\,\xi_{0,p}^{-1}\,.
\end{equation}
More generally, if the Hecke character ${\xi}_{0,p}$ verifies \ref{item_RIa}, then the Deligne-lattice $R_{\theta_{\xi_0}}^*$ is contained in an isomorphic copy of ${\rm Ind}_{K/\QQ}\,\xi_{0,p}^{-1}$ with index bounded only in terms of $a$.
\end{lemma}
\begin{proof}
The isomorphism in \eqref{eqn_lattice_pointwise_p_inverted} is well-known and the asserted equality is clear from definitions. When $R_{\theta_{\xi_0}}^*$ is residually irreducible (which is precisely the case if ${\xi}_{0,p}$ verifies \ref{item_RIa} with $a=0$), then the lattice $R_{\theta_{\xi_0}}^*$ is the unique $G_\QQ$-invariant lattice in $R_{\theta_{\xi_0}}^*\otimes_{\Zp}\Qp$, up to homothety and similarly the lattice ${\rm Ind}_{K/\QQ}\,\xi_{0,p}^{-1}$ inside $\left({\rm Ind}_{K/\QQ}\,\xi_{0,p}^{-1}\right)\otimes_{\Zp}\Qp$. This proves the existence of an isomorphism \eqref{eqn_lattice_pointwise_integral_res_irred} in this case. The remaining case follows from the following observation. Suppose $X$ is a free ${\cO}$-module of rank $d$ on which $G_\QQ$ acts irreducibly with the property that if $W_r$ is a $G_\QQ$-stable submodule of ${\varpi}^{-r}X/X$ ($r\in \mathbb{N}$) then ${\varpi}^aW_r$ is a free ${\cO/\varpi^{r-a}\cO}$-module of rank  $d$. Then  up to homothety, all $G_\QQ$-stable lattices inside $X\otimes_{\Zp}\Qp$ arise as the inverse images of $G_\QQ$-stable submodules of ${\varpi}^{-a}X/ X$ (under the map $X\otimes\Qp\to X\otimes \Qp/\Zp$).
\end{proof}

\begin{defn}
We say that a Hecke character $\xi$ is $p$-crystalline if the $G_K$-representation associated to $\xi_p$ is crystalline at both primes $\p$ and $\p^c$ above $p$.
\end{defn}

\begin{defn}\label{def_Gal_rep_xi_general}
Given a Hecke character $\xi$ with infinity type $(\ell_1,\ell_2)$ we define $R_{\theta_{\xi}}^*:=R_{\theta_{\xi_0}}^*(-\ell_1)$.
\end{defn}


\subsection{CM Hida families}
\label{subsec_CMHida}
For each $\q \in \{\p,\p^c\}$, we recall that $K_\infty^{(\q)}$ stands for the composite of all subfields of $K_\infty$ that are unramified at $\q^c$. We also recall that $\Gamma_\q:=\Gal(K_\infty^{(\q)}/K)\cong \ZZ_p$. For a fractional ideal $\frak{a}\subset \cO_K$ coprime to $\ff \p$, we shall write $\frak{Art}_{\p}([\frak{a}]) \in \Gamma_{\p}$ for its image under the geometrically-normalised Artin map composed with the canonical surjection $H_{\ff \p^\infty}\twoheadrightarrow \Gamma_\p$.
\begin{defn}
\label{defn_theta_series_in_families}
We set ${\displaystyle \mathbb{a}_n(\mathbb{g}_\chi):=\sum_{\substack{ (\frak{a},\ff\p)=1, \mathbf{N}\frak{a}=n}} \chi^{-1}(\frak{a})\,\frak{Art}_{\p}([\frak{a}])}$
and let 
$${\mathbb{g}}_\chi:=\sum_{n=1}^\infty \mathbb{a}_n(\mathbb{g}_\chi)q^n\in \LL_{\cO}(\Gamma_{\p})[[q]]$$
denote the canonical branch (determined by $\chi$) of the CM Hida family of tame level $|D_K|\,\mathbf{N}\ff$.
\end{defn}
\begin{remark}
Suppose $\xi$ is a $p$-crystalline Hecke character  modulo $\p^\infty$ with infinity-type $(0,1-\ell)$ and $2(p-1)\mid \ell-1$. Then
$$\sum_{(\frak{a},\ff \p)=1} \chi^{-1}\xi^{-1}(\frak{a})q^{\mathbf{N}\frak{a}} \in S_{\ell}(\Gamma_1(N_\chi p),\epsilon_\chi)$$
is the unique crystalline weight-$\ell$ specialization of the Hida family $\mathbb{g}_\chi$.
\end{remark}
Let $\LL_{\chi}$ denote the (cuspidal and new) branch of the Hida's universal Hecke algebra, in the sense of \cite[\S7.5]{KLZ2}, corresponding to $\mathbb{g}_\chi$. We define a $\ZZ_p$-algebra homomorphism
$$\phi_\chi:\LL_\chi\lra \LL_{\cO}(\Gamma_\p)$$ 
which is given by 
$$T_q \longmapsto \sum_{\substack{\mathbf{N}\q=q}}\chi^{-1}(\q)\,\frak{Art}_{\p}([\q]) \quad \forall q\neq p \quad;\quad U_p\longmapsto  \chi^{-1}(\p^c)\,\frak{Art}_{\p}([\p^c]) $$
$$\langle d \rangle \longmapsto \epsilon_\chi(d)\frak{Art}_{\p}([d])\qquad \forall d\in \ZZ \hbox{ with } (d,p|D_K|\mathbf{N}\ff)=1\,.$$

We let $V_{{\mathbb{g}}_\chi}^*\cong {\rm Frac}(\LL_\chi)^{\oplus 2}$ denote Hida's big Galois representation attached to ${\mathbb{g}}_\chi$ with coefficients in the field of fractions of $\LL_\chi$. There is a natural finitely generated $\LL_\chi$-submodule $R_{{\mathbb{g}}_\chi}^*\subset V_{{\mathbb{g}}_\chi}^*$ (which we shall refer to as the Hida--Ohta lattice; c.f. \cite{KLZ2} \S7.2 and \S7.5 for its definition\footnote{We note that our $\LL_\chi$ is denoted by $\LL_\mathbf{a}$ in op. cit. where $\mathbf{a}$ is the minimal prime ideal that is determined by $\mathbb{g}_\chi$, of the algebra $\LL_{\mathbb{g}_{\overline{\chi}}}$ (that corresponds to a Hida family $\mathbb{g}_\chi$ in the sense of \cite{KLZ2}, \S7.2), which is the localization of Hida's universal ordinary Hecke algebra at the maximal ideal which contains $\mathbf{a}$. In the notation of op. cit., our $R_{{\mathbb{g}}_\chi}^*$ is simply given as $M(\mathbb{g}_{\overline{\chi}})^*\otimes_{\LL_{\mathbb{g}_{\overline{\chi}}}} \LL_\mathbf{a}$.}) that verifies $R_{{\mathbb{g}}_\chi}^*\otimes_{\LL_\chi} {\rm Frac}(\LL_\chi)\xrightarrow{\sim} V_{{\mathbb{g}}_\chi}^*$. 

\begin{defn}
\label{defn_Hida_family_to_universal_rep}
We set ${\bf R}_\chi^*:=R_{{\mathbb{g}}_\chi}^*\otimes_{\phi_\chi}\LL_{\cO}(\Gamma_\p)$.
\end{defn}

\subsubsection{Induced representations}
\label{subsubsec_Ind}
Let $\LL_\cO(\Gamma_\p)^{\sharp}$ denote the free $\LL_\cO(\Gamma_\p)$-module of rank one on which $G_K$ acts via the canonical character $\mathbf{k}:\,G_K\twoheadrightarrow \Gamma_\p \hookrightarrow \LL_{\cO}(\Gamma_\p)^\times$. As in the proof of \cite[Proposition 6.1.3]{HT94}, we obtain an isomorphism
\begin{equation}
\label{eqn_HidaOhtavsInduce_overFrac}
    {\bf R}_\chi^*\otimes_{\LL_\cO(\Gamma_\p)} {\rm Frac}(\LL_\cO(\Gamma_\p))\xrightarrow[v_\chi]{\sim} \left({\rm Ind}_{K/\QQ}\, \LL_\cO(\Gamma_\p)^{\sharp}\otimes \chi\right)\otimes_{\LL_\cO(\Gamma_\p)} {\rm Frac}(\LL_\cO(\Gamma_\p))
\end{equation}
on comparing the traces of Frobenii at primes coprime to $p|D_K|\mathbf{N}\ff$. We note that we have again denoted abusively by $\chi$ the free $\cO$-module of rank one on which $G_K$ acts by $\chi$. 

It follows from \eqref{eqn_HidaOhtavsInduce_overFrac} that there exists a submodule $\mathbb{I}_{\mathbf{k}\otimes\chi}\subset  \left({\rm Ind}_{K/\QQ}\, \LL_\cO(\Gamma_\p)^{\sharp}\otimes \chi\right)\otimes_{\LL_\cO(\Gamma_\p)} {\rm Frac}(\LL_\cO(\Gamma_\p))$ such that $\mathbb{I}_{\mathbf{k}\otimes\chi}\cong {\rm Ind}_{K/\QQ}\, \LL_\cO(\Gamma_\p)^{\sharp}\otimes \chi$ as $\LL_\cO(\Gamma_\p)[[G_\QQ]]$-modules and which admits a $G_\QQ$-equivariant morphism 
\begin{equation}
\label{eqn_HidaOhtavsInduce_vague}
    u(\mathbb{I}_{\mathbf{k}\otimes\chi}): {\bf R}_\chi^*\lra \mathbb{I}_{\mathbf{k}\otimes\chi} \qquad\qquad\qquad\hbox{ verifying } \qquad u(\mathbb{I}_{\mathbf{k}\otimes\chi})\otimes_{\LL_\cO(\Gamma_\p)} {\rm Frac}(\LL_\cO(\Gamma_\p))=v_\chi.
\end{equation}

\begin{remark}
\label{rem_Ind_mod_P}
For all but finitely many height-one primes $\cP$ of $\LL_{\cO}(\Gamma_\p)$, the induced morphism $u(\mathbb{I}_{\mathbf{k}\otimes\chi})\otimes \kappa(\cP)$ is an isomorphism (where $\kappa(\cP)={\rm Frac}(\LL_{\cO}(\Gamma_\p)/\cP)$ is the residue field at $\cP$). Moreover, if $\cP$ is such an arithmetic prime corresponding to $\theta_{\xi_0}$ for a crystalline Hecke character $\xi_0$ whose $p$-adic avatar $\xi_{0,p}$ takes values in $\cO$, then $u(\mathbb{I}_{\mathbf{k}\otimes\chi})\otimes \kappa(\cP)$ recovers the isomorphism \eqref{eqn_lattice_pointwise_p_inverted}. Furthermore, the map $u(\mathbb{I}_{\mathbf{k}\otimes\chi})\otimes \LL_{\cO}(\Gamma_\p)/\cP$ gives rise to a morphism $R_{\theta_{\xi_0}}^*\to {\rm Ind}_{K/\QQ}\,\xi_{0,p}^{-1}\otimes \chi$, which identifies $R_{\theta_{\xi_0}}^*$ with a submodule of a homothetic copy of ${\rm Ind}_{K/\QQ}\,\xi_{0,p}^{-1}\otimes \chi$ with index bounded only in terms of $a$.
\end{remark}

Consider the following condition:
\begin{enumerate}
   \item[\mylabel{item_Ind}{{\bf (Ind)}}] There exists a choice of a pair  $(\mathbb{I}_{\mathbf{k}\otimes\chi}, u(\mathbb{I}_{\mathbf{k}\otimes\chi}))$ such that the cokernel of $u(\mathbb{I}_{\mathbf{k}\otimes\chi})$ is pseudo-null over $\Lambda_\cO(\Gamma_\p)$.
\end{enumerate}

\begin{remark}
\label{rem_Ind_lattice}
One expects that the condition \ref{item_Ind} (or its suitable variant, where one would consider the Ohta lattice associated to the closed modular curve) always holds true. The forthcoming work of Burungale--Skinner--Tian treats this problem in the extreme case when the branch character $\chi$ is trivial. We are grateful to A. Burungale for cautioning us about \ref{item_Ind} and informing us about his joint work in progress with Skinner and Tian.

We show in Proposition~\ref{prop_uniqueness_of_the_lattice_in_residually_irred_case} below that the hypothesis \ref{item_Ind} holds true if we assume that the CM family $\mathbb{g}_\chi$ is residually non-Eisenstein. In explicit terms, this assumption on $\mathbb{g}_\chi$ amounts to the requirement that \ref{item_RIa} holds true with $a=0$.
\end{remark}

\begin{proposition}
\label{prop_uniqueness_of_the_lattice_in_residually_irred_case}
Suppose that the hypothesis \ref{item_RIa} holds with $a=0$. Then the Ohta lattice $\mathbf{R}_\chi^*$ is induced: It is isomorphic to $\Ind_{K/\QQ}\Lambda_\cO(\Gamma_\p)^\sharp\otimes\chi$ as a $\Lambda_\cO(\Gamma_\p)$-module. 
\end{proposition}
\begin{proof}
Our hypothesis  on \ref{item_RIa} guarantees that $\bR_\chi^*$ is a free $\Lambda_\cO(\Gamma_\p)$-module of rank two. Without loss of generality,  we may assume that $\bR_\chi^*$ is contained inside $\Ik$; note that this can be achieved after passing to a $\Lambda_\cO(\Gamma_\p)$-homothetic copy of $\Ik$. On fixing $\LL_\cO(\Gamma_\p)$-bases of $\bR_\chi^*$ and $\Ik$, we may consider $\uIk$ as a $2\times 2$ matrix $\cM=\begin{pmatrix}x&y\\z&t\end{pmatrix}$, with $x,y,z,t\in \LL_\cO(\Gamma_\p)$.

Moreover, since $\LL_\cO(\Gamma_\p)$ is a unique factorization domain, after multiplying by scalars if necessary, we may assume  that $x,y,z,t$ do not share a common factor. With this arrangement in place, we contend to show that $\det(\cM)\in\LL_\cO(\Gamma_\p)^\times$. This would prove that $\uIk$ is an isomorphism and conclude the proof of the proposition.

Suppose the contrary. Then there exists a height-one prime $\cP$ such that $\det(\cM)\in \cP$. Let us write $A_\cP$ for the discrete valuation ring $\Lambda_\cO(\Gamma_\p)/\cP$. We also denote
by $\uIk_\cP$ the map 
$$\uIk\otimes A_\cP:\mathbf{R}_\chi^*/\cP\mathbf{R}_\chi^*\longrightarrow \Ik/\cP\Ik\,.$$ 
Let us consider the short exact sequence of $G_\QQ$-representations
\begin{equation}
\label{eqn_app_u_coker}
    0\lra \bR_\chi^*\stackrel{u}{\lra} \Ik\lra C\lra0,
\end{equation}
where $C\xrightarrow{\sim} \LL_\cO(\Gamma_\p)^{\oplus 2}/\cM \LL_\cO(\Gamma_\p)^{\oplus 2}$ as $\LL_{\cO}(\Gamma_\p)$-modules (and the isomorphism is induced by the chosen basis of $\Ik$). On applying the right exact functor $-\otimes_{\LL_\cO(\Gamma_\p)}A_\cP$ and recalling that $\Ik$ is $\LL_\cO(\Gamma_\p)$-torsion free, \eqref{eqn_app_u_coker} yields the exact sequence
\begin{equation}
\label{eqn_app_u_coker_2}
0\lra C[\cP]\stackrel{\frak{d}}{\lra} \bR_\chi^*/\cP\bR_\chi^*\lra \Ik/\cP\Ik\lra C/\cP C \lra 0\,.
\end{equation}
Since $\cM$ modulo $\cP$ is not the zero matrix and $\det(\cM)\in \cP$ by assumption, it follows that $C/\cP C$ an $\cA_\cP$-module of rank 1. Furthermore, recalling that $\bR_\chi^*$ and $\Ik$ are free of rank two over $\Lambda_\cO(\Gamma_\p)$, we infer that $\bR_\chi^*/\cP\bR_\chi^*$ and $\Ik/\cP \Ik$ are both free of rank two over $\cA_\cP$, and therefore, by the exactness of the sequence \eqref{eqn_app_u_coker_2}, we conclude that that $C[\cP]$ is an $\cA_\cP$-module of rank 1. In particular, the quotient
\[
\left(\bR_\chi^*/\cP\bR_\chi^*\right)\big{/}\frak{d}\left(C[\cP]\right)
\]
 is a Galois-stable $\cA_\cP$-submodule of rank one inside $\Ik/\cP\Ik$. This is impossible, since the residual representation of $\Ik$ is absolutely irreducible (thanks to our running hypothesis that \ref{item_RIa} holds with $a=0$) and has dimension $2$.
\end{proof}

\subsubsection{} \label{subsubsec_local_properties_CMgg} We fix a lattice $\mathbb{I}_{\mathbf{k}\otimes\chi}$ and a morphism $u(\mathbb{I}_{\mathbf{k}\otimes\chi})$ verifying \eqref{eqn_HidaOhtavsInduce_vague}, and identify $\mathbb{I}_{\mathbf{k}\otimes\chi}$ with ${\rm Ind}_{K/\QQ}\, \LL_\cO(\Gamma_\p)^{\sharp}\otimes \chi$. Put $u:=u(\mathbb{I}_{\mathbf{k}\otimes\chi})$ and $\mathscr{C}(u):=\Char_{\LL_\cO(\Gamma_\p)}({\rm coker}(u))$. Note that whenever the hypothesis \ref{item_Ind} holds for the pair, $\mathscr{C}(u)$ can be taken to be $(1)$.

We record basic properties of the Galois representations ${\bf R}_{\chi}^*$ and $\mathbb{I}_{\mathbf{k}\otimes\chi}={\rm Ind}_{K/\QQ}\, \LL(\Gamma_\p)^{\sharp}\otimes \chi$ which will be utilized in subsequent sections.
\begin{enumerate}[(a)]
\item We have a natural decomposition of $G_K$-representations
\begin{equation}
\label{eqn_CM_Hida_Rep}
\mathbb{I}_{\mathbf{k}\otimes\chi}\vert_{G_K}= (\LL_\cO(\Gamma_\p)^{\sharp}\otimes \chi) \oplus (\LL_\cO(\Gamma_{\p^c})^{\sharp}\otimes \chi^c).
\end{equation}
We also have the map and identification
\begin{equation}
\label{eqn_CM_Hida_Rep_extended}
R_f^*\otimes {\bf R}_{\chi}^*\,\widehat{\otimes}\,\LL_{\cO}(\Gamma_\cyc) \xrightarrow{u} R_f^*\otimes \mathbb{I}_{\mathbf{k}\otimes\chi}\,\widehat{\otimes}\,\LL_{\cO}(\Gamma_\cyc) = \left({\rm Ind}_{K/\QQ}\, R_f^*\otimes\chi \right)\otimes\LL_{\cO}(\Gamma_K)^\sharp\,.
\end{equation}
of $G_\QQ$-representations, where the identification $\LL_\cO(\Gamma_\p)^\sharp\widehat{\otimes}_{_\cO} \LL_\cO(\Gamma_\cyc)^\sharp=\LL_\cO(\Gamma_K)^\sharp$ is given by the Verschiebung maps. A similar statement holds if we replace $f$ with a Hida family $\f$.
\item  We let $F^+R_{\mathbb{g}_\chi}^*$ and $F^-R_{\mathbb{g}_\chi}^*$ denote the $G_{\QQ_p}$-stable sub-quotients of $R_{\mathbb{g}_\chi}^*$, arising from \cite[Theorem~7.2.3]{KLZ2}. Since $\mathbb{g}_\chi$ is cuspidal, both $F^+R_{\mathbb{g}_\chi}^*$ and $F^-R_{\mathbb{g}_\chi}^*$  are finite projective over the weight space and in fact, $F^+R_{\mathbb{g}_\chi}^*$ is free of rank one over $\LL_{\chi}$ (c.f. \cite{KLZ2}, Theorem 7.2.3(v)). We set $F^\pm {\bf R}_{\chi}^*:=F^\pm R_{\mathbb{g}_\chi}^*\otimes_{\phi_\chi} \LL_\cO(\Gamma_\p)$. Both $\LL_{\cO}(\Gamma_{\p})$-modules $F^+{\bf R}_{\chi}^*(-\mathbf{k})$ and $F^-{\bf R}_{\chi}^*$ carry an unramified action of $G_{\QQ_p}$ and the geometric Frobenius acts on $F^-{\bf R}_{\chi}^*$ by  multiplication by the image $\phi_\chi(\mathbb{a}_p(\mathbb{g}_\chi))$ of the $U_p$-eigenvalue.

\item For $\q=\p,\p^c$, let us identify $G_{\QQ_p}=G_{K_\q}$ with a decomposition group inside $G_K$. Through the decomposition \eqref{eqn_CM_Hida_Rep}, $u$ gives rise to $G_{K_\q}$-equivariant morphisms
\begin{equation}
\label{eqn_CM_Hida_Rep_extended_further}
F^+{\bf R}_{\chi}^*\xrightarrow{u_+^{(\q)}} \LL_\cO(\Gamma_\q)^{\sharp}\otimes \chi^{?}\qquad,\qquad F^-{\bf R}_{\chi}^*\xrightarrow{u_-^{(\q)}} \LL_\cO(\Gamma_{\q^c})^{\sharp}\otimes \chi^{\invques}\,.
\end{equation}
where $?=\begin{cases}
\emptyset& \q=\p\\
c& \q=\p^c
\end{cases}$ and $\invques=\begin{cases}
\emptyset& \q=\p^c\\
c& \q=\p
\end{cases}$, such that $u_\pm^{(\q)}\otimes_{\LL(\Gamma_\q)}{\rm Frac}(\LL(\Gamma_\q))$ are isomorphisms. Let us put $\mathscr{C}(u_\pm^{(\q)}):=\Char_{\LL_\cO(\Gamma_\q)}({\rm coker}(u_\pm^{(\q)}))$. 
\end{enumerate}
Note that the morphism $u_+^{(\q)}$ is always injective since $F^+{\bf R}_{\chi}^*$ is free of rank one. 

Consider the following conditions:
\begin{enumerate}
   \item[\mylabel{item_Indpm}{$\hbox{{\bf (Ind)}}_\pm^{\q}$}] There exists a choice of a pair  $(\mathbb{I}_{\mathbf{k}\otimes\chi}, u(\mathbb{I}_{\mathbf{k}\otimes\chi}))$ such that $\mathscr{C}(u_\pm^{(\q)})=(1)$.
\end{enumerate}
Observe that $\mathscr{C}(u_\pm^{(\q)})$ divides $\mathscr{C}(u)$. In particular, the condition \ref{item_Ind} implies \ref{item_Indpm}.

\subsubsection{Shapiro's lemma} 
We  fix a cuspidal Hida family $\f$ of tame level $N_f$ as before. Recall that $R_{\f}^*$ denotes Hida's $\Lambda_\f$-adic representation attached to $\f$ and $T_{\f,\chi}:=R_{\f}^*\otimes \chi$. Let $m$ be a positive integer coprime to $pN_f\,\mathbf{N}\ff\,|D_K|$.

\begin{defn}
\label{defn_+_BK_ffGamma_K}
We define $H^1_{\FFF_+}(\QQ(\mu_m)_p,R_{\f}^*\,\widehat{\otimes}\,{\bf R}_\chi^*\,\widehat{\otimes}\,\LL(\Gamma_\cyc)^\sharp)$ as the kernel of the natural morphism
\begin{equation}
\label{eqn_proj_ff_Gamma_K_minusminus}
    H^1(\QQ(\mu_m)_p,R_{\f}^*\,\widehat{\otimes}\,{\bf R}_\chi^*\,\widehat{\otimes}\,\LL(\Gamma_\cyc)^\sharp)\lra H^1(\QQ(\mu_m)_p,F^-R_{\f}^*\,\widehat{\otimes}\, F^-{\bf R}_\chi^*\,\widehat{\otimes}\,\LL(\Gamma_\cyc)^\sharp)\,.
\end{equation}
on semi-local cohomology groups at $p$. We also define the submodule  
\begin{align*}
    H^1_{\FFF_+}(K(\mu_m)_p,R_{\f}^*\otimes\chi\,\widehat{\otimes}\LL(\Gamma_K)^\sharp):=H^1(K(\mu_m)_\p,R_{\f}^*\otimes\chi\,\widehat{\otimes}\LL(\Gamma_K)^\sharp)\oplus H^1(K(\mu_m)_{\p^c},F^+R_{\f}^*\otimes\chi\,\widehat{\otimes}\LL(\Gamma_K)^\sharp)
    \end{align*}
    of the semi-local cohomology group $H^1(K_p,R_{\f}^*\otimes\chi\,\widehat{\otimes}\LL(\Gamma_K)^\sharp)$.
\end{defn}

Equivalently, $H^1_{\FFF_+}(\QQ(\mu_m)_p,R_{\f}^*\,\widehat{\otimes}\,{\bf R}_\chi^*\,\widehat{\otimes}\,\LL(\Gamma_\cyc)^\sharp)$ is the image of the injective map
\begin{equation}
\label{eqn_proj_ff_Gamma_K_minusminus_bis}
H^1\left(\QQ(\mu_m)_p,(F^+R_{\f}^*\,\widehat{\otimes}\, {\bf R}_\chi^*+ R_{\f}^*\,\widehat{\otimes}\, F^+{\bf R}_\chi^*)\,\widehat{\otimes}\,\LL(\Gamma_\cyc)^\sharp\right)\lra H^1(\QQ(\mu_m)_p,R_{\f}^*\,\widehat{\otimes}\,{\bf R}_\chi^*\,\widehat{\otimes}\,\LL(\Gamma_\cyc)^\sharp)
\end{equation}

Let $u$ be a morphism be as in \S\ref{subsubsec_local_properties_CMgg}. By definitions and semi-local Shapiro's lemma, we have the following natural commutative diagrams:
\begin{equation}
    \label{eqn_shapiro_local_pe_minus}
    \begin{aligned}
    \xymatrix{
    H^1_{\FFF_+}(\QQ(\mu_m)_p,R_{\f}^*\,\widehat{\otimes}\,{\bf R}_\chi^*\,\widehat{\otimes}\,\LL(\Gamma_\cyc)^\sharp)\ar[r]^(.45){r_{+}^{(-+)}}\ar[d]_{u}&  H^1(\QQ(\mu_m)_p,F^-R_{\f}^*\,\widehat{\otimes}\,F^+{\bf R}_\chi^*\,\widehat{\otimes}\,\LL(\Gamma_\cyc)^\sharp)\ar[d]^{u_+^{(\p)}}\\
    H^1_{\FFF_+}(K(\mu_m)_p,R_{\f}^*\otimes\chi\,\widehat{\otimes}\LL(\Gamma_K)^\sharp)\ar[r]^(.5){r_{\p}^{(-)}}&  H^1(K(\mu_m)_\p,F^-R_{\f}^*\otimes\chi\,\widehat{\otimes}\,\LL(\Gamma_K)^\sharp),
    }
    \end{aligned}
\end{equation}

\begin{equation}
    \label{eqn_shapiro_local_pe_c_minus}
    \begin{aligned}
    \xymatrix{
    H^1_{\FFF_+}(\QQ(\mu_m)_p,R_{\f}^*\,\widehat{\otimes}\,{\bf R}_\chi^*\,\widehat{\otimes}\,\LL(\Gamma_\cyc)^\sharp)\ar[r]^(.45){r_{+}^{(+-)}}\ar[d]_{u}&  H^1(\QQ(\mu_m)_p,F^+R_{\f}^*\,\widehat{\otimes}\,F^-{\bf R}_\chi^*\,\widehat{\otimes}\,\LL(\Gamma_\cyc)^\sharp)\ar[d]^{u_-^{(\p^c)}}\\
    H^1_{\FFF_+}(K(\mu_m)_p,R_{\f}^*\otimes\chi\,\widehat{\otimes}\LL(\Gamma_K)^\sharp)\ar[r]^(.5){r_{\p^c}^{(-)}}&  H^1(K(\mu_m)_{\p^c},F^+R_{\f}^*\otimes\chi\,\widehat{\otimes}\,\LL(\Gamma_K)^\sharp).
    }
    \end{aligned}
\end{equation}


\section{Beilinson--Flach elements}
\label{sec_BF_elements}

We let $\chi$ denote a ray class character of $K$ with conductor $\ff$. As above, we fix a cuspidal Hida family $\f$ of tame level $N_f$. We recall that $R_{\f}^*$ denotes Hida's $\Lambda_\f$-adic representation attached to $\f$ and $T_{\f,\chi}:=R_{\f}^*\otimes \chi$. For each positive integer $m$ coprime $pN_f\,\mathbf{N}\ff\,|D_K|$, we let
\[
\mathbb{BF}^{\f,\chi}_{m}\in H^1(\QQ(\mu_m),R_{\f}^*\,\widehat{\otimes}\,{\bf R}_{\chi}^*\,\widehat{\otimes}_{\ZZ_p}\,\LL(\Gamma_\cyc)^{\sharp})
\]
denote the Beilinson--Flach elements given as the image of the element in \cite[Definition 8.1.1]{KLZ2} (after canceling out the smoothing factor involving the auxiliary parameter $c$) under the morphism induced from $\phi_\chi$. We also put
$${}^u\mathbb{BF}^{\f,\chi}_{m}\in H^1(K(\mu_m),T_{\f,\chi}\,\widehat{\otimes}_{\ZZ_p}\,\LL(\Gamma_K)^{\sharp})$$
for each choice of $u$ as in \S\ref{subsubsec_local_properties_CMgg}, as the image of $\mathbb{BF}^{\f,\chi}_{m}$ under the morphism induced from $u$ and Shapiro's lemma. When $m=1$, in place of $\mathbb{BF}^{\f,\chi}_{1}$ and ${}^u\mathbb{BF}^{\f,\chi}_{1}$, we shall write $\mathbb{BF}_{\f,\chi}$ and ${}^u\mathbb{BF}_{\f,\chi}$ respectively.

Suppose $f^{\alpha}\in S_{k_f+2}(\Gamma_1(N_fp),\epsilon_f)$ be a specialization of the Hida family $\f$, where $f$ is a cuspidal eigenform of level $N_f$ with $f^{\alpha}$ its $p$-ordinary stabilization. Recall that $R_{f}^*$ stands for Deligne's representation attached to $f$ and $T_{f,\chi}:=R_{f}^*\otimes\chi$. For each positive integer $m$ coprime $pN_f\,\mathbf{N}\ff\,|D_K|$, we let
\[
\mathbb{BF}^{f^{\alpha},\chi}_{m}\in H^1(\QQ(\mu_m),R_{f}^*\,\widehat{\otimes}\,{\bf R}_{\chi}^*\,\widehat{\otimes}_{\ZZ_p}\,\LL(\Gamma_\cyc)^{\sharp})
\]
denote the image of $\mathbb{BF}^{\f,\chi}_{m}$ under the natural specialization map determined by $f^{\alpha}$. Similarly, we define ${}^u\mathbb{BF}^{f^{\alpha},\chi}_{m}\in H^1(K(\mu_m),T_{f,\chi}\,\widehat{\otimes}_{\ZZ_p}\,\LL(\Gamma_K)^{\sharp})$ denote the image of ${}^u\mathbb{BF}^{\f,\chi}_{m}$ for each $u$ as above. When $m=1$, in place of $\mathbb{BF}^{f^{\alpha},\chi}_{1}$ and ${}^u\mathbb{BF}^{f^{\alpha},\chi}_{1}$, we shall write $\mathbb{BF}_{f^{\alpha},\chi}$ and ${}^u\mathbb{BF}_{f^{\alpha},\chi}$ respectively.

\subsection{Self-dual twists}
\label{subsec_central_critical_twists}
We assume in this subsection that $\chi^c=\chi^{-1}$ and $\varepsilon_f=\mathds{1}$. When we are studying anticyclotomic variation, we will work with the twisted variants 
$${}^u\mathbb{BF}^{\f,\chi,\dagger}_{m}\in H^1(K(\mu_m),T_{\f,\chi}^\dagger\,\widehat{\otimes}_{\ZZ_p}\,\LL(\Gamma_K)^{\sharp})
$$
$${}^u\mathbb{BF}^{f^{\alpha},\chi,
\dagger}_{m}\in H^1(K(\mu_m),T_{f,\chi}^\dagger\,\widehat{\otimes}_{\ZZ_p}\,\LL(\Gamma_K)^{\sharp})\,.$$
We will also put ${}^u\mathbb{BF}_{\f,\chi}^\dagger:={}^u\mathbb{BF}^{\f,\chi,\dagger}_{1}$ and ${}^u\mathbb{BF}_{f^{\alpha},\chi}^{\dagger}:={}^u\mathbb{BF}^{f^{\alpha},\chi,
\dagger}_{1}$ to ease notation.

\subsection{$p$-local properties of Beilinson--Flach elements}

\begin{defn}
\label{defn_res_singular_quotients}
For $?=f,\f$ and any positive integer $m$ coprime to $p\frak{f}N_f|D_K|$, we let $\res_p^{(--)}$ denote the composition
$$ H^1(\QQ(\mu_m),R_{?}^*\,\widehat{\otimes}\,{\bf R}_{\chi}^*\,\widehat{\otimes}_{\ZZ_p}\,\LL(\Gamma_\cyc)^{\sharp})\xrightarrow{\Xi\,\circ\,\res_p} H^1(\QQ(\mu_m)_p, F^-R_{?}^*\,\widehat{\otimes}\,F^-{\bf R}_{\chi}^*\,\widehat{\otimes}_{\ZZ_p}\,\LL(\Gamma_\cyc)^{\sharp})\,,$$
where $\Xi$ is as given in Definition~\ref{defn_+_BK_ffGamma_K}.

We also define
$$\res_{\p^c}^-:  H^1(K(\mu_m),R_{?}^*\otimes\chi\,\widehat{\otimes}\LL(\Gamma_K)^\sharp)\lra H^1(K(\mu_m)_{\p^c},F^-R_{?}^*\otimes\chi\,\widehat{\otimes}\LL(\Gamma_K)^\sharp)\,$$
to be the composition of the localization map at $\p$ with the natural map induced by the projection $R_?^*\rightarrow F^-R_?^*$.
\end{defn}
\begin{proposition}
 \label{prop_local_conditions_for_BF_CMfamilies}
For every positive integer $m$ coprime to $p\frak{f}N_f|D_K|$, we have
$$\res_{\p^c}^{-}\left({}^u\mathbb{BF}^{?,\chi}_{m}\right) =0$$
for $?\in\{f^\alpha,\f\}$.
\end{proposition}
\begin{proof}
For $?=f^\alpha,\f$, we infer from \cite[Proposition 8.1.7]{KLZ2} that $\res_p^{(--)}\left(\mathbb{BF}^{?,\chi}_{m}\right)=0$. The proof follows from the commutativity of the following diagram 
$$\xymatrix{
H^1(\QQ(\mu_m),R_{?}^*\,\widehat{\otimes}\,{\bf R}_{\chi}^*\,\widehat{\otimes}_{\ZZ_p}\,\LL(\Gamma_\cyc)^{\sharp})\ar[rr]^(.47){\res_p^{(--)}}\ar[d]_u&& H^1(\QQ(\mu_m)_p, F^-R_{?}^*\,\widehat{\otimes}\,F^-{\bf R}_{\chi}^*\,\widehat{\otimes}_{\ZZ_p}\,\LL(\Gamma_\cyc)^{\sharp})\ar[d]^{u_{-}^{(\p^c)}}\\
H^1(K(\mu_m),R_{?}^*\otimes{\chi}\,\widehat{\otimes}_{\ZZ_p}\,\LL(\Gamma_\cyc)^{\sharp})\ar[rr]^(.5){\res_{\p^c}^{-}}&& H^1(K(\mu_m)_{\p^c}, F^-R_{?}^*\otimes{\chi}\,\widehat{\otimes}_{\ZZ_p}\,\LL(\Gamma_\cyc)^{\sharp})
}
$$
where the vertical arrows are induced from Shapiro's lemma and the morphism $u$.
\end{proof}

\begin{defn}
For $?=f,\f$ we denote by $\res_\p^{(1)}$ the composition of the arrows
$$\ker\left(\res_{\p^c}^{-}\right)\xrightarrow{\res_p} H^1_{\FFF_+}(K_p,R_{?}^*\otimes{\chi}\,\widehat{\otimes}_{\ZZ_p}\,\LL(\Gamma_\cyc)^{\sharp})\xrightarrow{r_\p^{(-)}} H^1(K_\p,F^-R_{?}^*\otimes{\chi}\,\widehat{\otimes}_{\ZZ_p}\,\LL(\Gamma_\cyc)^{\sharp})\,.$$

\end{defn}

\section{Selmer complexes and $p$-adic $L$-functions}
\label{sec_Selmer_complexes_padic_L_fucntions}
\subsection{Selmer complexes}
\label{subsec_Selmer_complexes}
We start with a very general definition of a Selmer complex.
\begin{defn}
\label{defn_selmer_complex_general_ord}
Consider any complete local Noetherian ring $R$ and a free $R$-module $X$ of finite rank which is endowed with a continuous action of $G_{K,\Sigma}$. Assume that for each prime $\q$ of $K$ above $p$, we are given a free $R$-direct summand $F^+_{\q}X$ of $X$, which is stable under the action of $G_{K_\q}$.

\item[i)] The Selmer complex 
$$\widetilde{{\bf R}\Gamma}_{\rm f}(G_{K,\Sigma},X;\Delta_X)\in D_{\rm ft} (_{R}{\rm Mod})$$
with local conditions $\Delta_X$ is given as in \cite[\S6.1]{nekovar06}. We denote its cohomology by $\widetilde{H}^\bullet_{\rm f}(G_{K,\Sigma},X;\Delta_X)$.

\item[ii)]We shall write $\Delta^{(1)}$ for local conditions which are unramified for all primes in $\Sigma$ that are coprime to $p$ $($see \cite[\S8]{nekovar06} for details$)$ and which are given by the Greenberg conditions $($see \cite[\S6.7]{nekovar06}$)$ with the choice
$$j_{\q}^+:F^+_{\q}X\stackrel{}{\longrightarrow} X$$
for each prime $\q$ of $K$ above $p$.

\item[iii)]Suppose $(p)=\p\p^c$ splits in $K$. We shall write $\Delta^{(2)}$ for local conditions which are unramified for all primes in $\Sigma$ that are coprime to $p$ $($see \cite[\S8]{nekovar06} for details$)$ and which are given by the Greenberg conditions $($see \cite[\S6.7]{nekovar06}$)$ with the choices
$$j_{\p}^+:X\stackrel{=}{\longrightarrow} X$$
$$j_{\p^c}^+:\{0\}\hookrightarrow X$$
at the primes $\p$ and $\p^c$, respectively.

\end{defn}

We recall the Galois representations we have introduced in Definition~\ref{def_intro_Gal_rep_RS}. We shall plug these in place of $X$ in Definition~\ref{defn_selmer_complex_general_ord} to define the Selmer groups relevant to our discussion. 

Since we assumed that $f$ is $p$-ordinary, the $G_{\QQ_p}$-representation $R_f^*$ admits a $G_{\QQ_p}$-stable direct summand $F^+R_f^*$ of rank one (with the additional property that the quotient $F^-R_f^*:=R_f^*/F^+R_f^*$ is unramified). This gives rise to a rank-one direct summand $F^+\TT_{f,\chi}^{(\Gamma)}$  of $\TT_{f,\chi}^{(\Gamma)}$ (resp., a rank-one direct summand $F^+\TT_{f,\chi}^{\dagger}\subset\TT_{f,\chi}^{\dagger}$).

\begin{defn}
\label{defn_Selmer_complex_split_ord}
Let $\Sigma$ denote the set of places of $K$ which divide $p\ff N_f\infty$.
\item[i)] The complex 
$$\widetilde{{\bf R}\Gamma}_{\rm f}(G_{K,\Sigma},\TT_{f,\chi}^{(\Gamma)};\Delta^{(1)})\in D_{\rm ft} (_{\LL_\cO(\Gamma)}{\rm Mod})$$
will be called the Greenberg Selmer complex corresponding $  \Sigma^{(1)}_{\rm crit}$. Here, $\Delta^{(1)}$ stands for the local conditions which are unramified for all primes in $\Sigma$ that are coprime to $p$ and which are given by the Greenberg conditions with the choice
$$j_{\q}^+:\,F^+\TT_{f,\chi}^{(\Gamma)} \stackrel{}{\longrightarrow} \TT_{f,\chi}^{(\Gamma)}$$
for each $\q\in \{\p,\p^c\}$.

\item[ii)] The complex 
$$\widetilde{{\bf R}\Gamma}_{\rm f}(G_{K,\Sigma},\TT_{f,\chi}^{(\Gamma)};\Delta^{(2)})\in D_{\rm ft} (_{\LL_\cO(\Gamma)}{\rm Mod})$$
will be called the Greenberg Selmer complex corresponding $ \Sigma^{(2)}_{\rm crit}$. 

We similarly define $\widetilde{{\bf R}\Gamma}_{\rm f}(G_{K,\Sigma},\TT_{f,\chi}^{\dagger};\Delta^{(i)})\in D_{\rm ft} (_{\LL_\cO(\Gamma_K)}{\rm Mod})$ for $i=1,2$.
\end{defn}

\begin{remark}
\label{remark_shapiro}
For each $i\in\{1,2\}$, we have a natural isomorphism
$$ \varprojlim_{\rm cor}\widetilde{H}_{\rm f}^\bullet(G_{K^\prime,\Sigma^\prime},T_{f,\chi};\Delta^{(i)})\stackrel{\sim}{\lra}\widetilde{H}_{\rm f}^\bullet(G_{K,\Sigma},\TT_{f,\chi}^{(\Gamma)};\Delta^{(i)})$$
by \cite[Proposition 8.8.6]{nekovar06} (induced by Shapiro's Lemma). Here $K^\prime/K$ runs through finite subextensions of $K_\Gamma/K$ and we denote by $\Sigma^\prime$ the set of primes of $K^\prime$ that lie above the primes in $\Sigma$. 
\end{remark}

\subsection{Hida's $p$-adic Rankin--Selberg $L$-functions}
\label{subsec_Hida_fGammaK}
In this section, we review the $p$-adic $L$-functions of Hida attached to families of motives of ${\rm GL}_2\times {\rm Res}_{K/\QQ}{\rm GL}_1$-type. Recall that  $H_{p^\infty}$ denotes the ray class group of $K$ modulo $p^\infty$. 
\subsubsection{$p$-adic $L$-functions for $f_{/K}\otimes \chi$}
\begin{defn}
\label{defn_Hidas_padic_Lfunctions}
\item[i)] We let $L_p(f_{/K}\otimes\chi,  \Sigma^{(1)}_{\rm crit})\in \LL_L(H_{p^\infty})$ denote the branch of Hida's $p$-adic $L$-function defined by
\[
\psi_p\mapsto L_p(f/K,\Sigma^{(1)})(\chi^{-1}\psi_p^{-1}\chi_\cyc^{-k_f/2}),
\]
where $L_p(f/K,\Sigma^{(1)})$ is defined as in Appendix~\ref{appendix:corrige}.

\item[ii)] For any torsion-free quotient $\Gamma$ of $\Gamma_K$, we write $L_p(f_{/K}\otimes\chi,  \Sigma^{(1)}_{\rm crit})\big{\vert}_\Gamma\in \LL_L(\Gamma)$ for the image of Hida's $p$-adic $L$-function $L_p(f_{/K}\otimes\chi,  \Sigma^{(1)}_{\rm crit})$ under the canonical projection $\LL_{L}(\Gamma_K)\twoheadrightarrow \LL_{L}(\Gamma)\,.$

\item[iii)] For $\chi$ as above, let us write $H_\chi \in \LL_{\Phi}(\Gamma_\ac)$ for any generator of the CM congruence module of $\theta_\chi$ $($in the sense of \cite{Hida88AJM, HT91}$)$. We let 
$$L_p(f_{/K}\otimes\chi,  \Sigma^{(2)}_{\rm crit})\in \frac{1}{H_\chi}\LL_{\Phi}(H_{p^\infty})$$ 
(where $\Phi=L\widehat{\QQ_p^{\rm ur}}$) denote the branch of Hida's $p$-adic $L$-function corresponding to $\chi$, defined by
\[
\psi_p\mapsto L_p(f/K,\Sigma^{(2)})(\chi^{-1}\psi_p^{-1}\chi_\cyc^{-k_f/2}),
\]
where $L_p(f/K,\Sigma^{(2)})$ is defined as \cite[Theorem~2.3]{BLForum}.

\item[iv)] For each $\Gamma\in {\rm Gr}(\Gamma_K)$, we let $L_p(f_{/K}\otimes\chi, \Sigma^{(2)}_{\rm crit})\big{\vert}_\Gamma\in\frac{1}{H_\chi} \LL_\Phi(\Gamma)$ denote the image of Hida's $p$-adic $L$-function $L_p(f_{/K}\otimes\chi,  \Sigma^{(2)}_{\rm crit})$ under the map induced from the natural surjection $H_{p^\infty}\twoheadrightarrow \Gamma$.

\item[v)] Suppose $\chi^c=\chi^{-1}$ and $\varepsilon_f=\mathds{1}$. We define $L_p(f_{/K}\otimes\chi, \Sigma^{(1)}_{\rm cc})\in \LL_\Phi(G_{p^\infty})$ as the image of the twisted $p$-adic $L$-function ${\rm Tw}\left( L_p(f_{/K}\otimes\chi, \Sigma^{(1)}_{\rm crit})\right)$ under the natural map induced from $H_{p^\infty}\twoheadrightarrow G_{p^\infty}$. We also put $L_p(f_{/K}\otimes\chi,\Sigma^{(1)}_{\rm cc}){\vert}_{\Gamma_\ac}\in \LL_\Phi(\Gamma_\ac)$ for the image of $L_p(f_{/K}\otimes\chi,\Sigma^{(1)}_{\rm cc})$ and similarly define $L_p(f_{/K}\otimes\chi,\Sigma^{(2)}_{\rm cc}){\vert}_{\Gamma_\ac}\in \frac{1}{H_\chi}\LL_{\Phi}(\Gamma_\ac)$. We note that the generator $H_\chi$ of Hida's CM congruence ideal is a function on anticyclotomic characters and its restriction to $\Gamma_\ac$ is also denoted here by $H_\chi$ by slight abuse.
\end{defn}

\begin{remark}
Note that we have
$${\rm Tw}\left( L_p(f_{/K}\otimes\chi, \Sigma^{(i)}_{\rm crit})\right)(\Xi)=L_p(f_{/K}\otimes\chi,  \Sigma^{(i)}_{\rm crit})({\rm Tw}(\Xi))$$
for each character $\Xi$ of $H_{p^\infty}$. In particular, if $\psi_p=\xi_p\chi_{\cyc}^{-k_f/2}={\rm Tw}(\xi_p)$ is the Galois character associated to the $p$-adic avatar of $\psi \in \Sigma^{(i)}_{\rm cc}$ with $p$-power conductor (so that $\xi_p$ is Galois character associated to the $p$-adic avatar of an anticyclotomic Hecke character $\xi$ such that $\psi=\xi|\cdot|^{k_f/2}={\rm Tw}(\xi)$; see the discussion in Remark~\ref{rem_cc_vs_Galois}), we have
$$L_p(f_{/K}\otimes\chi, \Sigma^{(i)}_{\rm cc})(\xi_p)\,\dot{=}\,L(f_{/K},\chi^{-1}\xi^{-1},k_f/2+1)\,,$$
where ``$\dot{=}$'' stands for equality up to explicit factors. We remind the readers that when $\psi \in \Sigma^{(i)}_{\rm cc}(k_f)$, the value $L(f_{/K},\chi^{-1}\xi^{-1},k_f/2+1)=L(f_{/K},\chi^{-1}\psi^{-1},1)$ is the central critical value for the Rankin--Selberg motive in consideration. While the values that are interpolated here are, a priori, the values of an imprimitive $L$-series, they actually agree with the values of the primitive (``motivic'') $L$-series 
thanks to our running assumption that $N_f$ is coprime $D_K\mathbf{N}\ff$, c.f. \cite[Remark~2.2]{loeffler18}.
\end{remark}

\begin{remark}
\label{rem_compare_with_Forum}
 As explained in the first paragraph of the proof of \cite[Theorem 3.19]{BLForum}, $H_\chi L_p(f_{/K}\otimes\chi,  \Sigma^{(2)}_{\rm crit})$ is always non-zero. 

Similarly, if either $k_f>0$ or else $k_f=0$ and there is no prime $v\mid\ff$ such that $v^c\mid\ff$, then it follows that $L_p(f_{/K}\otimes\chi,  \Sigma^{(1)}_{\rm crit})$ is non-trivial. Indeed, in the former scenario, one may choose $\psi\in \Sigma^{(1)}_{\rm crit}$ such that $s=0$ falls within the region of absolute convergence for the $L$-series $L(f_{/K}, \chi\psi,s)$. In the latter case, the asserted non-triviality is a consequence of Rohrlich's generic non-vanishing result in \cite{rohrlich88Annalen} and the interpolation formulae.
\end{remark}

\subsubsection{}\label{subsubsec_Hida_padicL_for_ff} Let $\f$ denote a primitive Hida family of tame conductor $N_f$, which admits $f^\alpha$ as a weight $k_f+2$ crystalline specialization.  We let $\LL_\f$ denote the corresponding branch of the Hida's universal ordinary Hecke algebra. When $\overline{\rho}_\f=\overline{\rho}_f$ is absolutely irreducible, there exists a free $\LL_\f$-module $R_\f^*$ of rank two, which is equipped with a continuous action of $G_{\QQ}$ unramified outside primes dividing $pN_f$ and which interpolates Deligne's representations $R_{\f(\kappa)}^*$ associated to arithmetic specializations $\f(\kappa)$ of the Hida family. If we assume in addition that $f$ is $p$-distinguished, then the $G_{\QQ_p}$-representation $R_\f^*$ admits a $G_{\QQ_p}$-stable direct summand $F^+R_\f^*$ of rank one. 

\begin{defn}
\item[i)] Let us denote by $S_\f^{\rm crys}$ the collection of arithmetic specializations $\kappa: \LL_\f \rightarrow L$ which are crystalline (in the sense that the eigenform $\f(\kappa)$ is $p$-old). 
\item[ii)] Suppose $\kappa\in S_\f^{\rm crys}$. Let us denote by $\Sigma^{(1)}_{\rm crit}(\kappa)$ the set of Hecke characters of $p$-power conductor and infinity type $(\ell_1,\ell_2)$ with $0\leq \ell_1,\ell_2\leq w(\kappa)$, where $w(\kappa)+2$ is the weight of the classical eigenform $\f(\kappa)$. Similarly, we write $\Sigma^{(2)}_{\rm crit}(\kappa)$ for the set of Hecke characters of $p$-power conductor and infinity type $(\ell_1,\ell_2)$ with $\ell_1\geq w(\kappa)+1$ and $\ell_2\leq -1$.
\end{defn}

The following is a reformulation of the main theorem of \cite{hida88}, as are the $p$-adic $L$-functions we have recalled in \S\ref{subsec_Hida_fGammaK} for individual members of the family $\f$.
\begin{theorem}[Hida]
\label{thm_Hida_3var_ord_split}
Let $H_\f \in \LL_\f$ denote a generator of Hida's congruence ideal, given as in \cite[\S4]{hida88}. There exist a pair of $p$-adic $L$-functions
$$L_p(\f_{/K}\otimes\chi,\mathbb{\Sigma}^{(1)}) \in \frac{1}{H_\f}\LL_\f\,\widehat{\otimes}_{\Zp}\LL_{L}(H_{p^\infty})$$
$$L_p(\f_{/K}\otimes\chi,\mathbb{\Sigma}^{(2)}) \in \frac{1}{H_\chi}\LL_\f\,\widehat{\otimes}_{\Zp}\LL_{\Phi}(H_{p^\infty})$$
which are characterised by the interpolative properties 
$$L_p(\f_{/K}\otimes\chi,\mathbb{\Sigma}^{(1)})(\kappa)=L_p(\f(\kappa)_{/K}\otimes\chi,{\Sigma}^{(1)}_{\rm crit}(\kappa))$$
$$L_p(\f_{/K}\otimes\chi,\mathbb{\Sigma}^{(2)})(\kappa)=L_p(\f(\kappa)_{/K}\otimes\chi,{\Sigma}^{(2)}_{\rm crit}(\kappa))$$
where $\kappa$ runs through the set $S_\f^{\rm crys}$.
\end{theorem}

For each $\Gamma\in {\rm Gr}(\Gamma_K)$, we define $L_p(\f_{/K}\otimes\chi,\mathbb{\Sigma}^{(i)})|_\Gamma$ as in Definition~\ref{defn_Hidas_padic_Lfunctions}(ii) and (iv).

We remark that the expression $L_p(\f_{/K}\otimes\chi,\mathbb{\Sigma}^{(1)})(\kappa)$ makes sense when $\kappa\in S_\f^{\rm crys}$ since the meromorphic function $L_p(\f_{/K}\otimes\chi,\mathbb{\Sigma}^{(1)})$ is regular at each arithmetic prime thanks to \cite[Theorem 4.2]{hida88}.


\subsection{Reciprocity laws for Beilinson--Flach elements}
\label{subsec_local_prop_and_reciprocity_BF}
We shall now review the reciprocity laws in \cite{KLZ2}, which we shall recast in the framework of the present article. We assume in \S\ref{subsec_local_prop_and_reciprocity_BF} that $\overline{\rho}_\f$ is absolutely irreducible and $p$-distinguished, so that $R_\f^*$ is a free $\LL_\f$-module of rank two and its restriction to $G_{\QQ_p}$ admits a $G_{\QQ_p}$-stable direct summand $F^+R_\f^*$. We put $F^-R_\f^*:=R_\f^*/F^+R_\f^*$. Let $\chi$ be a ray class character of conductor $\ff$ verifying the hypothesis \ref{item_RIa} for some natural number $a$.

\begin{defn}
Let $\f$ be a branch of the Hida family with tame level $N_f$ and let $\chi$ be a ray class character of $K$. We let $\res_{\p}^{(1)}$ denote the composition of the arrows
$$H^1(G_{K,\Sigma},\TT_{\f,\chi})\stackrel{\res_{\p}}{\lra} H^1(K_{\p},\TT_{\f,\chi})\lra H^1(K_{\p},F^-R_{\f}^*(\chi)\otimes\LL(\Gamma_K)^\sharp)$$
and $\res_{\p}^{(2)}$ denote the composition
$$H^1(G_{K,\Sigma},\TT_{\f,\chi})\stackrel{\res_{\p^c}}{\lra} H^1(K_{\p^c},\TT_{\f,\chi})\lra H^1(K_{\p^c},F^+R_{\f}^*(\chi)\otimes\LL(\Gamma_K)^\sharp)\,.$$

We also define $\res_\p^{(i)}:\, H^1(G_{K,\Sigma},\TT_{f,\chi})\to H^1(K_{\p^c},F^\pm R_{\f}^*(\chi)\otimes\LL(\Gamma_K)^\sharp)$  ($i=1,2$)  in a similar fashion.
\end{defn}

Until the end of \S\ref{subsec_local_prop_and_reciprocity_BF}, we shall assume that $\LL_\f$ is a Krull domain\footnote{One can always achieve by passing to a normalization of $\LL_\f$ if necessary. See \cite{TrevorArnold2011} for the effect of this alteration.}; then $\LL_\f(\Gamma):=\LL_\f\,\widehat{\otimes}\LL_{\cO}(\Gamma)$ is also a Krull domain for all $\Gamma\in {\rm Gr}(\Gamma_K)$. This allows us to define the characteristic ideals of finitely generated $\LL_\f$-modules, following \cite[\S3.1.5]{skinnerurbanmainconj}. If in addition $\LL_\f$ is regular (so that $\LL_\f(\Gamma)$ is a unique factorization domain; we note that $\LL_\f(\Gamma)$ needs not be a unique factorisation domain assuming only that $\LL_\f$ is) and $M$ is a torsion $\LL_\f(\Gamma)$-module, then we have 
\begin{equation}
\label{eqn_char_ideals}
\Char_{\LL_\f(\Gamma)}(M)=\prod_{P} P^{{\rm length}_{\LL_\f(\Gamma)_P}(M_P)}
\end{equation}
 for its characteristic ideal, where the product is over all all height-one primes of $\LL_\f(\Gamma)$. If $M$ is not torsion, we set $\Char_{\LL_\f(\Gamma)}(M):=0$.

\begin{proposition}
\label{prop_reciprocity_law} Fix $u$ as in \S\ref{subsubsec_local_properties_CMgg}.
\item[i)] Let $f \in S_{k_f+2}(\Gamma_1(N_f),\epsilon_f)$ is a cuspidal eigenform such that $\mathbf{N}\ff\,|D_K|$ and $pN_f$ are coprime. Then,
$$\Char_{\LL_{\cO}(\Gamma_K)}\left(H^1(K_{\p},F^-R_{f}^*(\chi)\otimes_{\ZZ_p}\LL(\Gamma_K)^\sharp)\Big{/}\LL_{\cO}(\Gamma_K)\cdot\res_{\p}^{(1)}\left({}^u\mathbb{BF}_{f,\chi}\right) \right)\,\Big|\,\mathscr{C}(u)H_f L_p(f_{/K}\otimes\chi,\Sigma {^{(1)}_{\rm crit}}) {\big|_{\Gamma_K}}\cdot\LL_\cO(\Gamma_K)\,,$$
where $H_f$ is the specialization of Hida's congruence ideal $H_{\f}$ to weight $f^{\alpha}$ and
$$\Char_{\LL_{\cO}(\Gamma_K)}\left(H^1(K_{\p^c},F^+R_{f}^*(\chi)\otimes_{\ZZ_p}\LL(\Gamma_K)^\sharp)\Big{/}\LL_{\cO_\Phi}(\Gamma_K)\cdot\res {_{\p}^{(2)}}\left({}^u\mathbb{BF}_{f,\chi}\right) \right){\otimes_\cO \cO_\Phi}\,\Big|\,\mathscr{C}(u)H_\chi L_p(f_{/K}\otimes\chi,{\Sigma} {^{(2)}_{\rm crit}}) {\big|_{\Gamma_K}}\cdot\LL_{{\cO_\Phi}}(\Gamma_K)\,.$$
\item[ii)] Let $\f$ be a branch of the Hida family with tame level $N_f$. Suppose that $\LL_\f$ is a Krull domain. Then,
$$\Char_{\LL_\f(\Gamma_K)}\left(H^1(K_{\p},F^-R_{\f}^*(\chi)\otimes_{\ZZ_p}\LL(\Gamma_K)^\sharp)\Big{/}\LL_\f(\Gamma_K)\cdot\res_{\p}^{(1)}\left({}^u\mathbb{BF}_{\f,\chi}\right) \right)\,\Big|\,\mathscr{C}(u)H_\f L_p(\f_{/K}\otimes\chi,\mathbb{\Sigma} {^{(1)}}) {\big|_{\Gamma_K}}\cdot \LL_\f(\Gamma_K)\,,$$
$$\Char_{\LL_\f(\Gamma_K)}\left(H^1(K_{\p^c},F^+R_{\f}^*(\chi)\otimes\LL(\Gamma_K)^\sharp)\Big{/}\LL_\f(\Gamma_K)\cdot\res {_{\p}^{(2)}}\left({}^u\mathbb{BF}_{\f,\chi}\right) \right){\otimes_\cO\cO_\Phi}\,\Big|\,\mathscr{C}(u)H_\chi L_p(\f_{/K}\otimes\chi,\mathbb{\Sigma} {^{(2)}}) {\big|_{\Gamma_K}}\cdot\LL_{{\f,\cO_\Phi}}(\Gamma_K)\,.$$
\end{proposition}

\begin{proof}
We first recall the construction of the Perrin-Riou map in \cite[Theorems~8.2.3 and 8.2.8]{KLZ2}.
Note that $F^+{\bf R}_\chi^*(-\mathbf{k})$ is an unramified $G_{\Qp}$-module. On realizing it as an inverse limit of finitely generated  $\cO$-linear $G_{\Qp}$-representations, say $\varprojlim_{i\in I} M_i$, we have an injective Perrin-Riou map
\[
\cL_{F^-R_f^*\otimes F^+{\bf R}_\chi^*(-\mathbf{k})}:H^1(\Qp,F^-R_f^*\otimes F^+{\bf R}_\chi^*(-\mathbf{k})\,\widehat{\otimes}\,\LL(\Gamma_\cyc)^\sharp)\lra \mathbf{D}( F^-R_f^*\otimes F^+{\bf R}_\chi^*(-\mathbf{k}))\,\widehat{\otimes}\,\Lambda_\cO(\Gamma_\cyc)
\]
by taking inverse limit of the Perrin-Riou maps
\[
H^1(\Qp,F^-R_f^*\otimes M_i\,{\otimes}\,\LL(\Gamma_\cyc)^\sharp)\lra \mathbf{D}( F^-R_f^*\otimes M_i)\,\widehat{\otimes}\,\Lambda_\cO(\Gamma_\cyc)=\mathbf{D}( F^-R_f^*)\otimes\mathbf{D}( M_i)\,\widehat{\otimes}\,\Lambda_\cO(\Gamma_\cyc)
\]
as $i$ runs through $I$. Here, $\mathbf{D}(-)$ is defined as in \cite[Definitions~8.2.1 and 8.2.7]{KLZ2}. By an argument similar to \cite[proof of Lemma~2.6]{BFSuper} and \cite[proof of Lemma~3.3.3]{BLInertOrdMC}, we see that 
\[
\image(\cL_{F^-R_f^*\otimes F^+{\bf R}_\chi^*(-\mathbf{k})})=\image(\cL_{F^-R_f^*})\,\widehat{\otimes}\,\mathbf{D}({\bf R}_\chi^*(-\mathbf{k})),
\]
where 
\[
\cL_{F^-R_f^*}:H^1(\Qp,F^-R_f^*\,{\otimes}\,\LL(\Gamma_\cyc)^\sharp)\lra \mathbf{D}( F^-R_f^*)\,{\otimes}\,\Lambda_\cO(\Gamma_\cyc)
\]
is the Perrin-Riou map attached to $F^-R_f^*$. Note that $\cL_{F^-R_f^*}$ is pseudo-surjective since $F^-R_f^*$ is an ordinary $G_{\Qp}$-representation with $H^0(\Qp,F^-R_f^*)=0$. Therefore,  after twisting by $\mathbf{k}$, we have the injective map
\[
H^1(\Qp,F^-R_f^*\otimes F^+{\bf R}_\chi^*\,\widehat{\otimes}\,\LL(\Gamma_\cyc)^\sharp)\lra \mathbf{D}( F^-R_f^*\otimes F^+{\bf R}_\chi^*)\,\widehat{\otimes}\,\Lambda_\cO(\Gamma_\cyc)
\]
with pseudo-null cokernel.  The congruence number $H_f$ determines a basis of $\mathbf{D}(F^-R_f^*)$. Combined this with  $u^{(\p)}_+$, we obtain an injective map
\[
\mathbf{D}( F^-R_f^*\otimes\otimes F^+{\bf R}_\chi^*)\,\widehat{\otimes}\,\Lambda_\cO(\Gamma_\cyc)\lra \mathbf{D}( \Lambda(\Gamma_\p)\otimes \chi)\,\widehat{\otimes}\,\Lambda_\cO(\Gamma_\cyc)\cong\Lambda_\cO(\Gamma_K),
\]
whose cokernel is killed by $\mathscr{C}(u)$. This results in the following injective Perrin-Riou map
\[
\cL_{f,\chi}^{(-,+)}:H^1(\Qp,F^-R_f^*\otimes F^+{\bf R}_\chi^*\,\widehat{\otimes}\,\LL(\Gamma_\cyc)^\sharp)\lra\Lambda_\cO(\Gamma_K)
\]
whose cokernel is killed by $\mathscr{C}(u)$.
Similarly, we may construct an injective Perrin-Riou map
\[
\cL_{f_{/K},\chi}^{(-,+)}:H^1(K_\p,F^-R_f^*\otimes \Lambda(\Gamma_K)^\sharp)\lra \Lambda_\cO(\Gamma_K)
\]
with pseudo-null cokernel. 
Via \eqref{eqn_shapiro_local_pe_minus}, these two maps are compatible in the sense that 
\[
\cL_{f_{/K},\chi}^{(-,+)}\circ u_+^{(\p)}=\cL_{f,\chi}^{(-,+)}.
\]

It follows from \cite[Theorem~10.2.2]{KLZ2} that we have the equality of $\LL_\cO(\Gamma_K)$ of ideals
\[
\left(\cL_{f_{/K},\chi}^{(-,+)}\circ\res_\p^{(-,+)}(^{u}\mathbb{BF}_{\f,\chi})\right)=\mathscr{C}(u_+^{(\p)})H_fL_p(f_{/K}\otimes \chi,\Sigma_{\rm crit}^{(1)})\big|_{\Gamma_K},
\]
where $\mathscr{C}(u_+^{(\p)})$ is the characteristic ideal of the cokernel of $u_+^{(\p)}$, which divides $\mathscr{C}(u)$.
In particular, this gives the isomorphism
\[
H^1(K_{\p},F^-R_{f}^*(\chi)\otimes_{\ZZ_p}\LL(\Gamma_K)^\sharp)\Big{/}\LL_{\cO}(\Gamma_K)\cdot\res_{\p}^{(1)}\left({}^u\mathbb{BF}_{f,\chi}\right) \cong \image \cL_{f_{/K},\chi}^{(-,+)}/\left(\mathscr{C}(u_+^{(\p)})H_fL_p(f_{/K}\otimes \chi,\Sigma_{\rm crit}^{(1)})\big|_{\Gamma_K}\right).
\]
Hence, the first divisibility in part (i) follows. The other three divisibilities can be proved similarly.
\end{proof}

\section{Iwasawa theory for $\GL(2)\times\GL(1)_{/K}$: Prior work}
\label{sec_IwThe1}

\subsection{Main Conjectures for $f_{/K}\otimes\chi$}
\label{subsec_IMC_fchi}
The following assertions are the Iwasawa Main Conjectures in the present setting.
\begin{conj}
\label{IMC_split_definite_indefinite_ord}
For any $\Gamma_\ac\neq\Gamma \in {\rm Gr}(\Gamma_K)$,
\item[i)]$\LL_L(\Gamma)\cdot L_p(f_{/K}\otimes\chi,  \Sigma^{(1)}_{\rm crit}){\big \vert}_\Gamma={\Char}_{\LL_{\cO}(\Gamma)}\left(\widetilde{H}_{\rm f}^2(G_{K,\Sigma},\TT_{f,\chi}^{(\Gamma)};\Delta^{(1)})\right)\otimes_{\Zp}\Qp\,,$ 
\item[ii)] $\LL_\Phi(\Gamma)\cdot H_\chi L_p(f_{/K}\otimes\chi,  \Sigma^{(2)}_{\rm crit}){\big \vert}_\Gamma={\Char}_{\LL_{\cO}(\Gamma)}\left(\widetilde{H}_{\rm f}^2(G_{K,\Sigma},\TT_{f,\chi}^{(\Gamma)};\Delta^{(2)})\right)\otimes_{\cO}\Phi\,.$
\end{conj}
When $\Gamma=\Gamma_\ac$, the following is the form of main conjectures we will study: 
\begin{conj}
\label{IMC_split_definite_indefinite_ord_anticyclo}
Suppose $\chi=\chi^c$ and $\epsilon_f=\mathds{1}$.
\item[i)]$\LL_L(\Gamma_\ac)\cdot L_p(f_{/K}\otimes\chi,  \Sigma^{(1)}_{\rm cc}){\big \vert}_{\Gamma_\ac}={\Char}_{\LL_{\cO}(\Gamma_\ac)}\left(\widetilde{H}_{\rm f}^2(G_{K,\Sigma},\TT_{f,\chi}^{\ac};\Delta^{(1)})\right)\otimes_{\Zp}\Qp\,,$ 
\item[ii)] $\LL_\Phi(\Gamma_\ac)\cdot H_\chi L_p(f_{/K}\otimes\chi,  \Sigma^{(2)}_{\rm cc}){\big \vert}_{\Gamma_\ac}={\Char}_{\LL_{\cO}(\Gamma_\ac)}\left(\widetilde{H}_{\rm f}^2(G_{K,\Sigma},\TT_{f,\chi}^{\ac};\Delta^{(2)})\right)\otimes_{\cO}\Phi\,.$
\end{conj}
We have the following ``twisted'' variant of Conjecture~\ref{IMC_split_definite_indefinite_ord}, which is more convenient when studying the anticyclotomic Iwasawa main conjectures:
\begin{conj}
\label{IMC_split_definite_indefinite_ord_twisted}
\item[i)]$\LL_L(\Gamma_K)\cdot {\rm Tw} \left(L_p(f_{/K}\otimes\chi, \Sigma^{(1)}_{\rm crit})\right){\big \vert}_{\Gamma_K}={\Char}_{\LL_{\cO}(\Gamma_K)}\left(\widetilde{H}_{\rm f}^2(G_{K,\Sigma},\TT_{f,\chi}^{\dagger};\Delta^{(1)})\right)\otimes_{\Zp}\Qp\,,$ 
\item[ii)] $\LL_\Phi(\Gamma_K)\cdot H_\chi {\rm Tw} \left(L_p(f_{/K}\otimes\chi,  \Sigma^{(2)}_{\rm crit})\right){\big \vert}_{\Gamma_K}={\Char}_{\LL_{\cO}(\Gamma_K)}\left(\widetilde{H}_{\rm f}^2(G_{K,\Sigma},\TT_{f,\chi}^{\dagger};\Delta^{(2)})\right)\otimes_{\cO}\Phi\,.$
\end{conj}
\begin{remark}
\label{rem_what_is_known_split_def_indef_ord}
\item[a)] Twisted versions (which are obtained via the twisting map ${\rm Tw}$) of the containments 
\begin{equation}\label{eqn_what_is_known_split_def_indef_ord_1}
    \LL_L(\Gamma)\cdot L_p(f_{/K}\otimes\chi,  \Sigma^{(1)}_{\rm crit}){\big \vert}_\Gamma\subset {\Char}_{\LL_{\cO}(\Gamma)}\left(\widetilde{H}_{\rm f}^2(G_{K,\Sigma},\TT_{f,\chi}^{(\Gamma)};\Delta^{(1)})\right)\otimes_{\Zp}\Qp\,,
    \end{equation}
    \begin{equation}\label{eqn_what_is_known_split_def_indef_ord_2}
      \LL_\Phi(\Gamma)\cdot H_\chi L_p(f_{/K}\otimes\chi,  \Sigma^{(2)}_{\rm crit}){\big \vert}_\Gamma \subset {\Char}_{\LL_{\cO}(\Gamma)}\left(\widetilde{H}_{\rm f}^2(G_{K,\Sigma},\TT_{f,\chi}^{(\Gamma)};\Delta^{(2)})\right)\otimes_{\cO}\Phi\,.
\end{equation}
were essentially\footnote{The results in op. cit. are stated in the case when $\Gamma=\Gamma_K$. However, the techniques  therein also prove the stated containments above without any additional difficulties.} proved in \cite[Theorem 3.19]{BLForum} under the following hypotheses: 
\begin{enumerate}
   \item[\mylabel{item_HIm}{{\bf (Im)}}] \ref{item_tau} holds.
\item[\mylabel{item_HSS}{{\bf (SS)}}] The order of ray class field of $K$ modulo $\ff$ is prime to $p$.
  \item[\mylabel{item_HnEZ}{{\bf (nEZ)}}]  If $k_f\equiv 0 \mod 2(p-1)$, then $v_p(\alpha-\chi(\p^c))=0$.
   \item[\mylabel{item_HDist}{{\bf (Dist)}}] $v_p(\chi(\p)- \chi(\p^c))=0$.
\end{enumerate}
Let us factor $N_f=N^+_fN^-_f$ where $N^+_f$ (resp., $N^-_f$) is only divisible by primes that are split (resp., inert) in $K/\QQ$. Suppose, in addition to the hypotheses listed above, that $N^-$ is square-free. In this setting and when $\Gamma=\Gamma_K$, the validity of Conjecture~\ref{IMC_split_definite_indefinite_ord}(i) is equivalent to the validity of Conjecture~\ref{IMC_split_definite_indefinite_ord}(ii); this is explained in \cite[Theorem 3.19(i)]{BLForum}.

\item[b)] Let $N^+_f$ and $N^-_f$ be as in the previous paragraph. Suppose that 

\begin{enumerate}
\item[\mylabel{item_SU1}{{\bf (SU1)}}] the residual representation $\overline{\rho}_f$ is  irreducible and $p$-distinguished;
    \item[\mylabel{item_SU2plus}{{\bf (SU2$+$)}}] $N^-_f$ is square-free and a product of odd number of primes;
 \item[\mylabel{item_SU3}{{\bf (SU3)}}] $\overline{\rho}_f$ is ramified at all primes dividing $N^-_f$; 
    \item[\mylabel{item_SU4plus}{{\bf (SU4$+$)}}] $k_f \equiv 0 \mod p-1$ and $\varepsilon_f=\mathds{1}$.
\end{enumerate}

Then main results of \cite{skinnerurbanmainconj} show that Conjecture~\ref{IMC_split_definite_indefinite_ord}(i) holds true with $\Gamma=\Gamma_K$ and $\chi=\mathds{1}$. 
\item[c)] Assume that $N_f$ is square-free, that there exists a prime $q\mid N^-_f$ such that $\overline{\rho}_f$ is ramified at $q$, as well as that the following condition holds:
\begin{itemize}
    \item[(wt-2)] $k_f \equiv 0 \mod p-1$ and the $U_p$-eigenvalue on the unique crystalline weight-$2$ specialization of the unique Hida family through $f^{\alpha}$ is not $\pm1$. 
\end{itemize}
In this scenario, Wan in \cite{xinwanwanrankinselberg} proved the containment
$$\LL_\Phi(\Gamma_K)\cdot H_\chi L_p(f_{/K},  \Sigma^{(2)}_{\rm crit}) \supset {\Char}_{\LL_{\cO}(\Gamma)}\left(\widetilde{H}_{\rm f}^2(G_{K,\Sigma},\TT_{f,\mathds{1}};\Delta^{(2)})\right)\otimes_{\cO}\Phi\,.$$
\end{remark}


\subsection{Main conjectures for $f_{/K}\otimes\chi$ as $f$ varies}
\label{subsubsec_IwThe1_ord_families}

Let us now consider the branch $\f$ of a primitive Hida family of tame conductor $N_f$ as in \S\ref{subsubsec_Hida_padicL_for_ff}.

In this set up, for each $\Gamma\in {\rm Gr}(\Gamma_K)$, Definition~\ref{defn_selmer_complex_general_ord} applies with $R=\LL_\f\,\widehat{\otimes}\LL(\Gamma)=:\LL_\f(\Gamma)$ and  $X=\TT_{\f,\chi}^{(\Gamma)}$ given as in Definition~\ref{def_intro_Gal_rep_RS} to obtain the following Greenberg Selmer complexes.
\begin{defn}\label{def_Nekovar_for_families}
\item[i)]Let $\Sigma$ denote the places of $K$ given as in the statement of Definition~\ref{defn_Selmer_complex_split_ord}, we have the Greenberg Selmer complexes in the derived category of continuous $\LL_\f(\Gamma)$-modules
$$\widetilde{{\bf R}\Gamma}_{\rm f}(G_{K,\Sigma},\TT_{\f,\chi}^{(\Gamma)};\Delta^{(i)})$$
for each $i=1,2$, which are defined in a manner identical to Definition~\ref{defn_Selmer_complex_split_ord}. 
\item[ii)]As before, we shall denote by $\widetilde{H}_{\rm f}^\bullet (G_{K,\Sigma},\TT_{\f,\chi}^{(\Gamma)};\Delta^{(i)})$ the cohomology of the Greenberg Selmer complex. 
\end{defn}
We remark that the derived complex $\widetilde{{\bf R}\Gamma}_{\rm f}(G_{K,\Sigma},\TT_{\f,\chi}^{(\Gamma)};\Delta^{(i)})$ can be represented by a perfect complex with degrees concentrated in the range $[0,3]$.

We shall assume until the end of \S\ref{sec_IwThe1} that $\LL_\f$ is a Krull domain. Recall from \eqref{eqn_char_ideals} that one can  define the characteristic ideals of the Selmer groups we have introduced above.  The Iwasawa Main Conjectures in this set up reads:
\begin{conj}
\label{IMC_ord_split_families}
For any $\Gamma_\ac\neq\Gamma\in {\rm Gr}(\Gamma_K)$ we have:
\item[i)] $\left(\LL_\f(\Gamma)\otimes_{\ZZ_p}\QQ_p\right)\cdot H_\f \ L_p(\f_{/K}\otimes\chi,  \mathbb{\Sigma}^{(1)}){\big \vert}_\Gamma=\Char_{\LL_\f(\Gamma)}\left(\widetilde{H}_{\rm f}^2(G_{K,\Sigma},\TT_{\f,\chi}^{(\Gamma)};\Delta^{(1)}) \right)\otimes_{\Zp}\Qp$\,.
\item[ii)] $\left(\LL_\f(\Gamma)\otimes_{\cO}\Phi\right)\cdot H_\chi  L_p(\f_{/K}\otimes\chi,  \mathbb{\Sigma}^{(2)}){\big \vert}_\Gamma={\Char}_{\LL_\f(\Gamma)}\left(\widetilde{H}_{\rm f}^2(G_{K,\Sigma},\TT_{\f,\chi}^{(\Gamma)};\Delta^{(2)})\right)\otimes_{\cO}\Phi\,.$
\end{conj}

When $\Gamma=\Gamma_\ac$, we will need to work with twisted variants of these $p$-adic $L$-functions. Recall the universal weight character $\bbkappa: G_\QQ\lra \LL_\f^\times$ and its fixed square-root $\bbkappa^{\frac{1}{2}}$.
\begin{defn}
\label{def_twisting_morphism_Hida}
\item[i)] We let ${\rm Tw}_{\f}: \LL_\f\,\widehat{\otimes}_{\Zp}\LL(H_{p^\infty}) \lra \LL_\f\,\widehat{\otimes}_{\Zp}\LL(H_{p^\infty})$ denote the $\LL_\f$-linear morphism induced by $\gamma\mapsto \bbkappa^{-\frac{1}{2}}(\gamma)\gamma$ for each $\gamma\in H_{p^\infty}$.
\item[ii)] Suppose $\chi=\chi^c$ and $\epsilon_f=\mathds{1}$. We set 
$$L_p(\f_{/K}\otimes\chi, \mathbb{\Sigma}^{(1)}_{\rm cc}){\big \vert}_{\Gamma_\ac}:={\rm Tw}_{\f}\left( L_p(\f_{/K}\otimes\chi, \mathbb{\Sigma}^{(1)})\right){\big \vert}_{\Gamma_\ac}\in \frac{1}{H_\f}\LL_\f(\Gamma_\ac)\otimes_{\ZZ_p}\QQ_p,$$
$$L_p(\f_{/K}\otimes\chi, \mathbb{\Sigma}^{(2)}_{\rm cc}){\big \vert}_{\Gamma_\ac}:={\rm Tw}_{\f}\left( L_p(\f_{/K}\otimes\chi, \mathbb{\Sigma}^{(2)})\right){\big \vert}_{\Gamma_\ac} \in \frac{1}{H_\chi}\LL_\f(\Gamma_\ac)\otimes_{\cO}\Phi\,.$$
\end{defn}

The following is the form of anticyclotomic main conjectures we will study. 
\begin{conj}
\label{IMC_split_definite_indefinite_ord_anticyclo_families}
Suppose $\chi=\chi^c$ and $\epsilon_f=\mathds{1}$.
\item[i)]$\left(\LL_\f(\Gamma_\ac)\otimes_{\ZZ_p}\QQ_p\right)\cdot H_\f L_p(\f_{/K}\otimes\chi,  \mathbb{\Sigma}^{(1)}_{\rm cc}){\big \vert}_{\Gamma_\ac}={\Char}_{\LL_\f(\Gamma_\ac)}\left(\widetilde{H}_{\rm f}^2(G_{K,\Sigma},\TT_{\f,\chi}^{\ac};\Delta^{(1)})\right)\otimes_{\Zp}\Qp\,,$ 
\item[ii)] $\left(\LL_\f(\Gamma_\ac)\otimes_{\cO}\Phi\right)\cdot H_\chi L_p(\f_{/K}\otimes\chi,  \Sigma^{(2)}_{\rm cc}){\big \vert}_{\Gamma_\ac}={\Char}_{\LL_\f(\Gamma_\ac)}\left(\widetilde{H}_{\rm f}^2(G_{K,\Sigma},\TT_{\f,\chi}^{\ac};\Delta^{(2)})\right)\otimes_{\cO}\Phi\,.$
\end{conj}
We have the following ``twisted'' variant of Conjecture~\ref{IMC_ord_split_families}, which is more convenient when studying the anticyclotomic Iwasawa main conjectures.
\begin{conj}
\label{IMC_split_definite_indefinite_ord_twisted_families}
\item[i)] $\left(\LL_\f(\Gamma)\otimes_{\ZZ_p}\QQ_p\right)\cdot H_\f \ {\rm Tw}_{\f}\left( L_p(\f_{/K}\otimes\chi, \mathbb{\Sigma}^{(1)})\right){\big \vert}_{\Gamma_K}=\Char_{\LL_\f(\Gamma)}\left(\widetilde{H}_{\rm f}^2(G_{K,\Sigma},\TT_{\f,\chi}^{\dagger};\Delta^{(1)}) \right)\otimes_{\Zp}\Qp$\,.
\item[ii)] $\left(\LL_\f(\Gamma)\otimes_{\cO}\Phi\right)\cdot H_\chi  {\rm Tw}_{\f}\left(L_p(\f_{/K}\otimes\chi,  \mathbb{\Sigma}^{(2)})\right){\big \vert}_{\Gamma_K}={\Char}_{\LL_\f(\Gamma)}\left(\widetilde{H}_{\rm f}^2(G_{K,\Sigma},\TT_{\f,\chi}^\dagger;\Delta^{(2)})\right)\otimes_{\cO}\Phi\,.$
\end{conj}

\begin{remark}
Theorem 3.19(i) of \cite{BLForum} shows that the specializations of Conjecture~\ref{IMC_split_definite_indefinite_ord_twisted_families}(i) and Conjecture~\ref{IMC_split_definite_indefinite_ord_twisted_families}(ii) to an $L$-valued specialization $\kappa\in S^{\rm crys}_\f$ with $\Gamma=\Gamma_K$
\begin{equation}
    \label{eqn_IMC_ord_split_families_kappa_1}
    \LL_L(\Gamma_K)\cdot H_{\f}(\kappa) {\rm Tw}_{\f(\kappa)}\left(L_p(\f_{/K}\otimes\chi, {\Sigma}^{(1)}_{\rm crit}(\kappa)\right)\big\vert_{\Gamma_K}={\Char}_{\LL_{\cO}(\Gamma_K)}\left(\widetilde{H}_{\rm f}^2(G_{K,\Sigma},\TT_{\f(\kappa),\chi}^\dagger;\Delta^{(1)})\right)\otimes_{\Zp}\Qp
\end{equation}
\begin{equation}
    \label{eqn_IMC_ord_split_families_kappa_2}
    \LL_\Phi(\Gamma_K)\cdot H_\chi  {\rm Tw}_{\f(\kappa)}\left(L_p(\f(\kappa)_{/K}\otimes\chi,  \Sigma^{(2)}_{\rm crit}(\kappa))\right)\big\vert_{\Gamma_K}={\Char}_{\LL_{\cO}(\Gamma_K)}\left(\widetilde{H}_{\rm f}^2(G_{K,\Sigma},\TT_{\f(\kappa),\chi}^\dagger;\Delta^{(2)})\right)\otimes_{\cO}\Phi
\end{equation}
 are equivalent. This equivalence heavily relies on the non-triviality of a certain Beilinson--Flach element, interpolated along $\Gamma_K$. 
\end{remark}

\begin{theorem}[Skinner--Urban]
\label{thm_IMC_ord_split_families} 
Let us write $N_f=N^+_fN^-_f$ as in Remark~\ref{rem_what_is_known_split_def_indef_ord}.
 Suppose that 
\begin{enumerate}
      \item[\ref{item_SU1}] the residual representation $\overline{\rho}_\f$ is  irreducible and $p$-distinguished;
    \item[\mylabel{item_SU2}{{\bf (SU2)}}] $N^-_f$ is a product of odd number of primes;
      \item[\ref{item_SU3}] $\overline{\rho}_\f$ is ramified at all primes dividing $N^-_f$; 
    \item[\mylabel{item_SU4}{{\bf (SU4)}}] $\varepsilon_f=\mathds{1}$.
\end{enumerate}  Then Conjecture~\ref{IMC_ord_split_families} holds when $\chi=\mathds{1}$ and $\Gamma=\Gamma_K$.
\end{theorem}
\begin{proof}
 This  is \cite[Theorem 3.26]{skinnerurbanmainconj}.
\end{proof}


\section{Horizontal Euler system argument and results}
\label{sec_results_split}

 Let us fix a cuspidal eigen-newform $f$ of level $\Gamma_1(N_f)$ as before and suppose that it admits a $p$-ordinary stabilization $f^\alpha$ with $U_p$-eigenvalue $\alpha_f$. Let $\chi$ denote a ray class character of prime-to-$p$ conductor $\ff$, verifying the condition \ref{item_RIa} for some natural number $a$. 

Let $\psi\in \Sigma_{\rm crit}$ be a crystalline critical Hecke character for the eigenform $f$ in the sense of Definition~\ref{defn_criticality} and let $\theta_{\chi\psi}$ denote the eigenform defined as in Definition~\ref{defn_theta_series}(ii). The eigenform $\theta_{\chi\psi}$ is $p$-ordinary and $\alpha_{\chi\psi}=\chi^{-1}\psi_0(\p^c)$, where $\psi_0$ is given as in Definition~\ref{defn_theta_series}(i). We assume\footnote{This assumption can be eliminated in \S\ref{sec_results_split} using slightly different Beilinson--Flach elements, but since our sights are set on Iwasawa theoretic results, we will adopt it at this stage already so as not to extend the length of our article even further.} that the $p$-adic avatar $\psi_p$ of $\psi$ factors through $\Gamma_K$. 

Recall the Galois representation $X_{f,\chi\psi}^{\circ}:=R_f^*\otimes R^*_{\theta_{\chi\psi}}$, where the second factor is given as in Definition~\ref{def_Gal_rep_xi_general}. We set $X_{f,\chi\psi}:=X_{f,\chi\psi}^{\circ}\otimes \QQ_p$. We also put $W_{f,\chi\psi}:=X_{f,\chi\psi}^{\circ}\otimes \QQ_p/\ZZ_p$ and $W_{f,\chi\psi}^*(1):=\Hom(X_{f,\chi\psi}^\circ,\mu_{p^\infty})$. We define the $G_K$-representations $A_{f,\chi\psi}:=T_{f,\chi\psi}\otimes\QQ_p/\ZZ_p$ and $A_{f,\chi\psi}^*(1):=\Hom(T_{f,\chi\psi},\mu_{p^\infty})$. 

Recall also that we may identify $X_{f,\chi\psi}^{\circ}$ as a submodule of a homothetic copy of $T_{f,\chi\psi}$, with index that only depends on $a$ (c.f. Lemma~\ref{lemma_lattices_theta_pointwise}). Then the natural maps $W_{f,\chi\psi}\to A_{f,\chi\psi}$ and $W_{f,\chi\psi}^*\to A_{f,\chi\psi}^*$ are then necessarily isomorphisms.


\subsection{Bloch--Kato Conjectures for $f_{/K}\otimes \chi\psi$}
\label{subsec_BK_conjectures}
In this section, we explain a proof of the Bloch--Kato conjectures for $f_{/K}\otimes \chi\psi$ for a generic choice of a Hecke character $\psi$ as above. The main tool, together with the Beilinson--Flach Euler system, is the extension of the locally restricted Euler system machinery we develop in Appendix~\ref{appendix_sec_ES_main} to allow the treatment of residually reducible Galois representations.

\begin{defn}
\item[i)] We let $\Sigma$ denote the set of places of $K$ which divide $p\ff N_f\infty$. For each integer $m$ coprime to $p\ff N_f$, let us write $\Sigma_m$ for the set of  places of $K\QQ(m)$ that lie above $\Sigma$. We define $\Sigma_{\QQ}$ to be the set of places of $\QQ$ that lie below those in $K$ and with a slight abuse of notation, we also write $\Sigma_m$ for the set of places of $\QQ(m)$ that lie above the primes of $\Sigma_{\QQ}$. 
\item[ii)] We let $\BF^{f\otimes\chi\psi}_{\alpha_f,\alpha_{\chi\psi}}(m)\in H^1(G_{\QQ(m),\Sigma_m},X_{f,\chi\psi}^\circ)$ denote the image of $\mathbb{BF}^{f^{\alpha},\chi}_{m}$ under the natural morphism induced from ${\bf R}_{\chi}^*\,\widehat{\otimes}\,\LL(\Gamma_\cyc)\to R_{\theta_\chi\psi}^*$. When $m=1$, we set $\BF^{f\otimes\chi\psi}_{\alpha_f,\alpha_{\chi\psi}}:=\BF^{f\otimes\chi\psi}_{\alpha_f,\alpha_{\chi\psi}}(1) \in H^1(G_{\QQ,\Sigma},X_{f,\chi\psi}^\circ).$
\item[iii)] Given a morphism $u$ as in \S\ref{subsubsec_local_properties_CMgg}, let us put ${}^u\BF^{f\otimes\chi\psi}_{\alpha_f,\alpha_{\chi\psi}}\in H^1(G_{\QQ,\Sigma},T_{f,\chi\psi})$ denote the image of ${}^u\mathbb{BF}^{f^{\alpha},\chi}$.
\end{defn}

\begin{defn}
\label{defn_cokernel_u_specialized_chipsi}
For a choice of the morphism $u$ as in \S\ref{subsubsec_local_properties_CMgg}, let us define $\mathscr{C}(\psi):={\rm length}_{\cO}\left(\mathscr{C}(u)\otimes_{\chi\psi}\cO\right)$.
\end{defn}

\begin{defn}
\label{defn_aux_Selmer_groups}

\item[i)] We define the Selmer structures $\mathcal{F}_{{\rm BK}^{(1)}}$, $\mathcal{F}_{{\rm BK}^{(2)}}$ and  $\mathcal{F}_+$ on $X^\circ_{f,\chi\psi}$ with the local conditions
$$H^1_{\mathcal{F}_{?}}(\QQ(m)_v,X^\circ_{f,\chi\psi})=H^1_{\mathcal{F}_{+}}(\QQ(m)_v,X^\circ_{f,\chi\psi})=\ker\left(H^1(\QQ(m)_v,X^\circ_{f,\chi\psi}) \lra H^1(\QQ(m)_v^{\rm ur},X_{f,\chi\psi}) \right),\hbox{\,\,\,\,\,\,\,\,if  } v\nmid p\infty$$
if $v\mid p$,
\begin{align*}
   H^1_{\mathcal{F}_{{\rm BK}^{(1)}}}(\QQ(m)_v,X^\circ_{f,\chi\psi})&={\rm im}\left( H^1(\QQ(m)_v,F^+R_{f}^*\,\otimes\,R_{\theta_{\chi\psi}}^*)\stackrel{j_{v}^+}{\lra}H^1(\QQ(m)_v,X^\circ_{f,\chi\psi})\right),\\
H^1_{\mathcal{F}_{{\rm BK}^{(2)}}}(\QQ(m)_v,X^\circ_{f,\chi\psi})&={\rm im}\left( H^1(\QQ(m)_v,R_{f}^*\otimes F^+R_{\theta_{\chi\psi}}^*)\stackrel{j_{v}^+}{\lra}H^1(\QQ(m)_v,X^\circ_{f,\chi\psi})\right),\\
 H^1_{\mathcal{F}_{+}}(\QQ(m)_v,X^\circ_{f,\chi\psi})&=H^1_{\mathcal{F}_{{\rm BK}^{(1)}}}(\QQ(m)_v,X^\circ_{f,\chi\psi})+H^1_{\mathcal{F}_{{\rm BK}^{(2)}}}(\QQ(m)_v,X^\circ_{f,\chi\psi})\,.
\end{align*}

The dual Selmer structure on $W_{f,\chi\psi}^*(1)$ induced by local Tate duality will be denoted by $\calF^*_?$. We shall write $H^1_{\mathcal{F}_?}(\QQ(m),X^\circ_{f,\chi\psi})$ and $H^1_{\mathcal{F}_{?}^*}(\QQ(m),W_{f,\chi\psi}^*(1))$ for the associated Selmer groups for $?=+,{\rm BK}^{(1)}, {\rm BK}^{(2)}$.

\item[ii)] We similarly define the Selmer structures $\mathcal{F}_{?}$ on $T_{f,\chi\psi}$ with the local conditions
$$H^1_{\mathcal{F}_{?}}(K_v,T_{f,\chi\psi})=\ker\left(H^1(K_v,T_{f,\chi\psi}) \lra H^1(K_v^{\rm ur},T_{f,\chi\psi}\otimes\QQ_p) \right),\hbox{\,\,\,\,\,\,\,\,if  } v\nmid p\infty\qquad \hbox{ and }$$
\begin{align*}
    H^1_{\mathcal{F}_{{\rm BK}^{(1)}}}(K_v,T_{f,\chi\psi})&={\rm im}\left( H^1(K_v,F^+R_{f}^*\,\otimes\,\chi\psi)\stackrel{j_{v}^+}{\lra}H^1(K_v,T_{f,\chi\psi})\right), \hbox{ if  } v\mid p\,,\\
H^1_{\mathcal{F}_{{\rm BK}^{(2)}}}(K_\p,T_{f,\chi\psi})&=H^1(K_\p,T_{f,\chi\psi})\,,\,\,\,\,
H^1_{\mathcal{F}_{{\rm BK}^{(2)}}}(K_{\p^c},T_{f,\chi\psi})=0\,,\\
H^1_{\mathcal{F}_{+}}(K_\p,T_{f,\chi\psi})&=H^1(K_\p,T_{f,\chi\psi})\,,\,\,\,\,
H^1_{\mathcal{F}_{+}}(K_{\p^c},T_{f,\chi\psi})=H^1_{\mathcal{F}_{{\rm BK}^{(1)}}}(K_{\p^c},T_{f,\chi\psi})\,.
\end{align*}

The dual Selmer structures on $A_{f,\chi\psi}^*(1)$ and corresponding Selmer groups will be notated also  analogously.

\item[iii)] For $i=1,2$, let us put $H^1_{\mathcal{F}_{+/(i)}}(\QQ_p,X^\circ_{f,\chi\psi}):=H^1_{\mathcal{F}_{+}}(\QQ_p,X^\circ_{f,\chi\psi})/H^1_{\mathcal{F}_{{\rm BK}^{(i)}}}(\QQ_p,X^\circ_{f,\chi\psi})$ and define $\res_p^{+/(i)}$ as the composite map
 $$H^1_{\mathcal{F}_{+}}(\QQ,X^\circ_{f,\chi\psi})\stackrel{\res_p}{\lra} H^1_{\mathcal{F}_{+}}(\QQ_p,X^\circ_{f,\chi\psi})\lra H^1_{\mathcal{F}_{+/(i)}}(\QQ_p,X^\circ_{f,\chi\psi})\,.$$
 Let us write $\res_p^{(1)}$ for the composition of the arrows
 $$H^1_{\mathcal{F}_{+}}(\QQ,X^\circ_{f,\chi\psi})\xrightarrow{\res_p^{+/(1)}} H^1_{\mathcal{F}_{+/(1)}}(\QQ_p,X^\circ_{f,\chi\psi})\hookrightarrow H^1(\QQ_p,F^-R_f^*\otimes F^+R_{\theta_{\chi\psi}}^*)\,,$$
 and similarly define $\res_p^{(2)}$ as the composite map
 $$H^1_{\mathcal{F}_{+}}(\QQ,X^\circ_{f,\chi\psi})\xrightarrow{\res_p^{+/(2)}} H^1_{\mathcal{F}_{+/(2)}}(\QQ_p,X^\circ_{f,\chi\psi})\hookrightarrow H^1(\QQ_p,F^+R_f^*\otimes F^-R_{\theta_{\chi\psi}}^*)\,.$$
\end{defn}

Note that the Definition~\ref{defn_aux_Selmer_groups} crucially relies on the fact that $p$ splits in $K/\QQ$.

\begin{defn}
\label{defn_6_4_2021_09_22_11_39}
For each $i\in \{1,2\}$, we define strict Bloch-Kato Selmer structure $\mathcal{F}_{{\rm sBK}^{(i)}}$ on $T_{f,\chi\psi}$ by replacing the local conditions at $v\mid \mathbf{N}\ff N_f$ determined by $\mathcal{F}_{{\rm BK}^{(i)}}$ with the local conditions
$$H^1_{\mathcal{F}_{{\rm sBK}^{(i)}}}(K_v,T_{f,\chi\psi})=\ker\left(H^1(K_v,T_{f,\chi\psi}) \lra H^1(K_v^{\rm ur},T_{f,\chi\psi}) \right)\,.$$
We similarly define a pair of Selmer structures ${\mathcal{F}_{{\rm sBK}^{(i)}}}$ on $X_{f,\chi\psi}^\circ$. The dual Selmer structures ${\mathcal{F}_{{\rm sBK}^{(i)}}^*}$ on $A_{f,\chi\psi}^*(1)$ and $W_{f,\chi\psi}^*(1)$ are defined as above.

We finally set $\frak{X}^{(i)}_{\rm BK}(K,T_{f,\chi\psi}):=H^1_{\mathcal{F}_{{\rm BK}^{(i)}}}(K,A_{f,\chi\psi}^*(1))^\vee$ and $\frak{X}^{(i)}_{\rm sBK}(K,T_{f,\chi\psi}):=H^1_{\mathcal{F}_{{\rm sBK}^{(i)}}}(K,A_{f,\chi\psi}^*(1))^\vee$.
\end{defn}

\begin{remark}
\label{rem_generalities_about_selmer_complexes_comparison_etc}
\item[i)]
Let us consider the local conditions $\Delta^{(i)}$ (i=1,2) on $X_{f,\chi\psi}^\circ$ (in the sense of Nekov\'a\v{r}), which are given by unramified conditions for all primes in $\Sigma_m$ that are coprime to $p$ and by the Greenberg conditions with the choices of $G_{\QQ_p}$-morphisms
$$j_{\q}^+:\,F^+R_{f}^*\,\otimes\,R_{\theta_{\chi\psi}}^*  \stackrel{}{\longrightarrow} X_{f,\chi\psi}^\circ \qquad(i=1)$$
$$j_{\q}^+:\,R_{f}^*\,\otimes\,F^+R_{\theta_{\chi\psi}}^* \stackrel{}{\longrightarrow} X_{f,\chi\psi}^\circ\qquad (i=2)$$
for primes $\q \in \Sigma_m$ lying above $p$.  Shapiro's Lemma combined with the Lemma~\ref{lemma_lattices_theta_pointwise} give rise to a morphism
$$\mathscr{S}: \widetilde{{\bf R}\Gamma}_{\rm f}(G_{\QQ(m),\Sigma_m},X_{f,\chi\psi}^\circ;\Delta^{(i)})\lra \widetilde{{\bf R}\Gamma}_{\rm f}(G_{K\QQ(m),\Sigma_m},T_{f,\chi\psi};\Delta^{(i)})\,.$$
This in turn induces a map
\begin{equation}
\label{eqn_Shapiro_for_F_plus}
\mathscr{S}: H^1_{\calF_+}(\QQ,X_{f,\chi\psi}^{\circ}) \lra H^1_{\calF_+}(K,T_{f,\chi\psi})
\end{equation}
whose cokernel is annihilated by $p^a$, as well as commutative diagrams
\begin{equation}
\label{eqn_resp1_shapiro_chipsi}
\begin{aligned}
\xymatrix{H^1_{\calF_+}(K,T_{f,\chi\psi}) \ar[r]^(.47){\res_\p^{(1)}} & H^1(K_\p,F^-R_f^*(\chi\psi))\\
H^1_{\calF_+}(\QQ,X_{f,\chi\psi}^{\circ})\ar[r]^(.4){\res_p^{(1)}}\ar[u]^{\eqref{eqn_Shapiro_for_F_plus}}& H^1(\QQ_p,F^-R_f^*\otimes F^+R_{\theta_{\chi\psi}}^*)\ar[u]_{u_+^{(\p)}},
}
\end{aligned}
\end{equation}
\begin{equation}
\begin{aligned}
\label{eqn_resp2_shapiro_chipsi}
\xymatrix{H^1_{\calF_+}(K,T_{f,\chi\psi}) \ar[r]^(.47){\res_{\p^c}^{(2)}} & H^1(K_{\p^c},F^-R_f^*(\chi\psi))\\
H^1_{\calF_+}(\QQ,X_{f,\chi\psi}^{\circ})\ar[u]^{\eqref{eqn_Shapiro_for_F_plus}}\ar[r]^(.4){\res_p^{(2)}}& H^1(\QQ_p,F^+R_f^*\otimes F^-R_{\theta_{\chi\psi}}^*)\ar[u]_{u_-^{(\p^c)}}\,.
}
\end{aligned}
\end{equation}
where $\res_\p^{(1)}$ and $\res_{\p^c}^{(2)}$ are defined in the obvious manner, whereas the definition of the maps $u_+^{(\p)}$ and $u_-^{(\p^c)}$ are similar to those in \eqref{eqn_CM_Hida_Rep_extended_further}.
\item[ii)] Let us put

\begin{equation}
     \label{eqn_Nek_to_classical_discrete_selmer_A_1}
     \cA_1:=\left(F^-R_f\otimes R_{\theta_{\chi\psi}}\right)\otimes\mu_{p^\infty} \cong \left(F^-R_f\otimes {\rm Ind}_{K/\QQ}\,\chi^{-1}\psi^{-1}\right)\otimes\mu_{p^\infty}
  \end{equation}
 \begin{equation}
     \label{eqn_Nek_to_classical_discrete_selmer_A_2}
     \cA_2:=\left(R_f\otimes F^-R_{\theta_{\chi\psi}}\right)\otimes\mu_{p^\infty}\cong \left(R_f\otimes (\chi^c\psi^c)^{-1}\right)\otimes\mu_{p^\infty}\,.
 \end{equation}
By \cite[Lemma 9.6.3]{nekovar06}, we have the exact sequences

\begin{equation}
\label{eqn_nekovarsmainlongecactseq_3}
    0 \lra H^0(\QQ_p,\mathcal{A}_1)\lra \widetilde{H}^1_{\rm f}(G_{\QQ,\Sigma},W_{f,\chi\psi}^*(1);\Delta^{(1),\perp})\stackrel{}{\lra}H^1_{\mathcal{F}_{{\rm sBK}^{(1)}}}(\QQ,W_{f,\chi\psi}^*(1))\lra 0\,.
\end{equation}
\begin{equation}
\label{eqn_nekovarsmainlongecactseq_4}
    0 \lra H^0(\QQ_p,\mathcal{A}_2)\lra \widetilde{H}^1_{\rm f}(G_{\QQ,\Sigma},W_{f,\chi\psi}^*(1);\Delta^{(2),\perp})\stackrel{}{\lra}H^1_{\mathcal{F}_{{\rm sBK}^{(2)}}}(\QQ,W_{f,\chi\psi}^*(1))\lra 0\,.
\end{equation}
\end{remark}

\begin{remark}
We record in this remark how the Bloch-Kato Selmer groups and their strict variants compare, observing the natural appearance of Tamagawa numbers in this context. This comparison can be achieved through the long exact sequence
\begin{align}
\begin{aligned}
   0\rightarrow H^1_{\mathcal{F}_{{\rm sBK}^{(i)}}}(K,T_{f,\chi\psi}) \rightarrow H^1_{\mathcal{F}_{{\rm BK}^{(i)}}}&(K,T_{f,\chi\psi})\rightarrow H^0({\rm Err})\\
\label{alignstar_strictvsususalBK_seq}& \rightarrow H^1_{\mathcal{F}_{{\rm sBK}^{(i)}}^*}(K,A_{f,\chi\psi}^*(1))^\vee \rightarrow H^1_{\mathcal{F}_{{\rm BK}^{(i)}}^*}(K,A_{f,\chi\psi}^*(1))^\vee \rightarrow 0\,.
\end{aligned}
\end{align}
Here, ${\rm Err}$ is the complex defined in \cite[(6.2.3)]{nekovar06} and by the discussion in (7.6.9) in op. cit., it is quasi-isomorphic to the complex ${\displaystyle \bigoplus_{v\mid \ff N_f} {\rm Err}_v}$, where
$${\rm Err}_v={\rm Err}_v(T_{f,\chi\psi}):= {\rm Cone}\left(Z_v\xrightarrow{{\rm Frob}_v-1} Z_v\right)[-1]$$
with 
$$Z_v={\rm coker}\left(V_{f,\chi\psi}^{I_v}/T_{f,\chi\psi}^{I_v}\lra A_{f,\chi\psi}^{I_v}\right)\,.$$
Note then that 
$$H^0({\rm Err}_v)\stackrel{\sim}{\lra} H^1(I_v,T_{f,\chi\psi})_{\rm tor}^{{\rm Frob}_v=1}$$
$$H^1({\rm Err}_v)\stackrel{\sim}{\lra} H^1(G_v/I_v,A_{f,\chi\psi}^{I_v}/{\rm div})$$
and hence the quantities
\begin{align*}
\notag |H^0({\rm Err}_v)|=|H^1(I_v,T_{f,\chi\psi})_{\rm tor}^{{\rm Frob}_v=1}|&=|H^1(G_v/I_v,A_{f,\chi\psi}^{I_v}/{\rm div})|=|H^1({\rm Err}_v)|
\\
&=|H^1(I_v,T_{f,\chi\psi}^*(1))_{\rm tor}|=|H^0({\rm Err}_v(T_{f,\chi\psi}^*(1)))|
\end{align*}
all agree with the $p$-part of the $v$-local Tamagawa factor for $T_{f,\chi\psi}$ (and $T_{f,\chi\psi}^*(1)$) introduced by Fontaine and Perrin-Riou  (c.f. \cite{FPR94}, \S5.3).
\end{remark}
\subsubsection{Comparison of extended Selmer groups and Bloch--Kato Selmer groups} 
We record below a precise comparison between the extended Selmer groups and their classical counterparts (c.f. Lemma~\ref{lemma_NEK_global_duality_compare_BK}). We will extend this discussion in Lemma~\ref{lemma_compare_Iwasawa_Nek_Iwasawa_Greenberg_K_infty} to cover Iwasawa theoretic Selmer groups; we note that the comparison in Lemma~\ref{lemma_compare_Iwasawa_Nek_Iwasawa_Greenberg_K_infty} does play a crucial role in what follows.

By Nekov\'a\v{r}'s global duality \cite[(6.3.5)]{nekovar06} and thanks to our assumption that $R_f$ is residually irreducible, we have an exact sequence 
\begin{align}
\label{eqn_nekovar_Grothendieck_duality_sequence}
\begin{aligned}
0 \lra H^0({\rm Err}) \lra \widetilde{H}^1_{\rm f}(G_{K,\Sigma},A_{f,\chi\psi}^*(1); \Delta^{(i),\perp}) \lra& \widetilde{H}^2_{\rm f}(G_{K,\Sigma},T_{f,\chi\psi};\Delta^{(i)})^\vee  \\
&\lra H^1({\rm Err}) \lra  \widetilde{H}^2_{\rm f}(G_{K,\Sigma},A_{f,\chi\psi}^*(1);\Delta^{(i),\perp})\,.
\end{aligned}
\end{align}
\begin{defn}
\label{defn_tamagawa_like_error}

Let us set 
$$E_1(f,\chi\psi):=\left|{\rm im}(H^1({\rm Err}) \rightarrow  \widetilde{H}^2_{\rm f}(G_{K,\Sigma},A_{f,\chi}^*(1);\Delta^{(i),\perp}))\right|\,.$$
 We also put 
 $$E_2(f,\chi\psi):=|{\rm im}(H^0({\rm Err}) \rightarrow  H^1_{\mathcal{F}_{{\rm sBK}^{(i)}}^*}(K,A_{f,\chi\psi}^*(1))^\vee |$$
 $${\rm{Tam}}(f,\chi\psi):=\frac{E_1(f,\chi\psi)}{E_2(f,\chi\psi)}\,.$$
\end{defn}

We are now ready to state the explicit comparison between the extended Selmer group and its classical counterpart.

\begin{lemma}
\label{lemma_NEK_global_duality_compare_BK}
For each $i\in \{1,2\}$ we have 
$${\ord_p}\left(\frac{|\widetilde{H}^2_{\rm f}(G_{K,\Sigma},T_{f,\chi\psi};\Delta^{(i)})|}{|H^0(\QQ_p,\mathcal{A}_i)|\cdot {\rm{Tam}}(f,\chi\psi)}\right)={\ord_p}\left|H^1_{\mathcal{F}_{{\rm BK}^{(i)}}^*}(K,A_{f,\chi\psi}^*(1))\right|\,.$$ 
which, in the scenario when either side equals infinity, reduces to the assertion that the other side is also infinity.   
\end{lemma}
\begin{proof}
 It follows from the  exact sequence \eqref{eqn_nekovar_Grothendieck_duality_sequence} that
\begin{align}
\label{eqn_global_duality_lemma_1} |\widetilde{H}^2_{\rm f}(G_{K,\Sigma},T_{f,\chi\psi};\Delta^{(i)})|= | \widetilde{H}^1_{\rm f}(G_{K,\Sigma},A_{f,\chi\psi}^*(1);\Delta^{(i),\perp})|/E_1(f,\chi\psi)\,.
\end{align}
Moreover, the exact sequence \eqref{alignstar_strictvsususalBK_seq} yields
\begin{align}
\label{eqn_global_duality_lemma_2} 
|H^1_{\mathcal{F}_{{\rm BK}^{(i)}}^*}(K,A_{f,\chi\psi}^*(1)) |= | H^1_{\mathcal{F}_{{\rm sBK}^{(i)}}^*}(K,A_{f,\chi\psi}^*(1))^\vee|/E_2(f,\chi\psi)\,.
\end{align}
 The asserted identity in the lemma follows on combining \eqref{eqn_nekovarsmainlongecactseq_3}, \eqref{eqn_nekovarsmainlongecactseq_4}, \eqref{eqn_global_duality_lemma_1} and \eqref{eqn_global_duality_lemma_2}. 
\end{proof}

\subsubsection{Locally restricted Euler system bound}
We resume our main task. In Theorem~\ref{thm_Split_BK_Sigma12} below, we employ the Beilinson--Flach element Euler system and the locally restricted Euler system argument we developed in \S\ref{appendix_sec_ES_main} in the residually reducible scenario.

Recall from Definition~\ref{defn_6_4_2021_09_22_11_39} the dual $\frak{X}^{(i)}_{\rm BK}(K,T_{f,\chi\psi})$ ($i=1,2$) of the discrete Bloch--Kato Selmer group.
\begin{theorem}
\label{thm_Split_BK_Sigma12} 
Assume that the hypothesis \ref{item_tau} holds true.
\item[i)] Suppose that $\psi \in \Sigma^{(1)}_{\rm crit}$ is a $p$-crystalline character with infinity type $(\ell_1,\ell_2)$, whose associated $p$-adic Galois character factors through $\Gamma_K$. We then have  
\begin{align*}
    {\ord_p}\left(\frak{X}^{(1)}_{\rm BK}(K,T_{f,\chi\psi})\right)\leq \left[H^1(K_\p,F^-R_f^*(\chi\psi)):\cO\cdot \res_\p^{(1)}\left({}^u\BF^{f\otimes\chi\psi}_{\alpha_f,\alpha_{\chi\psi}}\right)\right]+t\,.
\end{align*}
for some constant $t$ that depends only on $a$ and the residual representation $\overline{X}_{f,\chi\psi}$. Moreover, one may take $t=0$ if $a=0$.
\item[ii)] Suppose that $\psi \in \Sigma^{(2)}_{\rm crit}$ is $p$-crystalline. We then have  
\begin{align*}
{\ord_p}\left(\frak{X}^{(2)}_{\rm BK}(K,T_{f,\chi\psi})\right)\leq \left[H^1(K_{\p^c},F^+R_f^*(\chi\psi)):\cO\cdot \res_{\p^c}^{(2)}\left({}^u\BF^{f\otimes\chi\psi}_{\alpha_f,\alpha_{\chi\psi}}\right)\right]+t\,.
\end{align*}
for some constant $t$ that depends only on $a$ and the residual representation $\overline{X}_{f,\chi\psi}$. Moreover, one may take $t=0$ if $a=0$.
\end{theorem}

We will need the following fact about Beilinson--Flach elements in the proof of Theorem~\ref{thm_Split_BK_Sigma12}.

\begin{proposition}
 \label{prop_local_conditions_for_BF}
For every positive integer $m$ coprime to $p\frak{f}N_f|D_K|$, we have
$$\BF^{f\otimes\chi\psi}_{\alpha_f,\alpha_{\chi\psi}}(m) \in H^1_{\mathcal{F}_{+}}(\QQ(m),X^\circ_{f,\chi\psi})\,.$$
\end{proposition}
\begin{proof}
This is immediate from Proposition~\ref{prop_local_conditions_for_BF_CMfamilies} (which is used to verify the required local conditions at primes above $p$) and \cite[Proposition II.1.1]{colmez98} (which is utilized to verify the local conditions away from $p$).
\end{proof}

\begin{proof}[Proof of Theorem~\ref{thm_Split_BK_Sigma12}]
Let us put ${\bf BF}:=\{\BF^{f\otimes\chi\psi}_{\alpha_f,\alpha_{\chi\psi}}(m)\}_{m}$. In view of Proposition~\ref{prop_local_conditions_for_BF}, the collection ${\bf BF}$ is an $\calF_+$-locally restricted Euler system (in the sense of Definition~\ref{defn_appendix_locrestES}). It follows from Theorem~\ref{thm_appendix_ES_main}(ii) combined with Remark~\ref{remark_appendix_simplify_defn_index} that 
$$\ord_p|H^1_{\mathcal{F}_{+}^*}(\QQ,W_{f,\chi\psi}^*(1))|\leq \ord_p \left[H^1_{\mathcal{F}_{+}}(\QQ,X^\circ_{f,\chi\psi}):\cO\cdot \BF^{f\otimes\chi\psi}_{\alpha_f,\alpha_{\chi\psi}}\right]+t$$
(where the index on the right is allowed to be $+\infty$, in which we observe the convention that $\ord_p\infty=\infty$) for some natural number $t$ that depends only on $a$ and the residual representation $\overline{X}_{f,\chi\psi}$.

We then have for $i=1,2$\,, 
\begin{equation}
    \label{eqn_global_duality_bound_BK_step_1}
    \ord_p\left|H^1_{\mathcal{F}_{{\rm BK}^{(i)}}^*}(\QQ,W_{f,\chi\psi}^*(1))\right|\leq \left[H^1_{\mathcal{F}_{+/(i)}}(\QQ_p,X^\circ_{f,\chi\psi}):\cO\cdot \res_p^{+/(i)}\left(\BF^{f\otimes\chi\psi}_{\alpha_f,\alpha_{\chi\psi}}\right)\right]+t\,,
\end{equation}
 by Poitou-Tate global duality. Part (i) (respectively, (ii)) now follows from the definition of the map $\res_p^{(1)}$ (respectively,  $\res_p^{(2)}$), the diagram~\eqref{eqn_resp1_shapiro_chipsi} (respectively, \eqref{eqn_resp2_shapiro_chipsi}) and the definition of ${}^u\BF^{f\otimes\chi\psi}_{\alpha_f,\alpha_{\chi\psi}}$.
\end{proof}

\begin{theorem}
 \label{thm_BK_type_result}
Suppose $f \in S_{k_f+2}(\Gamma_1(N_f),\epsilon_f)$ is a cuspidal eigen-newform and $\chi$ is a ray class character of $K$. Let $\psi \in \Sigma_{\rm crit}^{(i)}$ be an algebraic Hecke character whose associated $p$-adic Galois character factors through $\Gamma_K$. Suppose that $\overline{\rho}_f$ is absolutely irreducible and either $\chi\neq \chi^c$ or $\psi\neq \psi^c$ and assume the validity of the condition \ref{item_tau}.  If $L(f_K\otimes\chi^{-1}\psi^{-1},0)\neq 0$, then the Selmer group $\widetilde{H}^2_{\rm f}(G_{K,\Sigma},T_{f,\chi\psi};\Delta^{(i)})$ has finite cardinality.
\end{theorem}

\begin{proof}
Under the running hypotheses, \ref{item_HypX} holds true with $X=X_{f,\chi\psi}$ and therefore, Theorem~\ref{thm_appendix_ES_main}(i) applies with $X=X_{f,\chi\psi}$, $\calF=\calF_+$ and $\mathbf{c}=\mathbf{BF}$ as in the proof of Theorem~\ref{thm_Split_BK_Sigma12} to show that the dual Selmer group $H^1_{\mathcal{F}_{{\rm BK}^{(i)}}^*}(\QQ,W_{f,\chi\psi}^*(1))$ has finite cardinality so long as $\res_p^{+/(i)}\left(\BF^{f\otimes\chi\psi}_{\alpha_f,\alpha_{\chi\psi}}\right)$ is non-torsion. In that case, it follows from the proof of Lemma~\ref{lemma_NEK_global_duality_compare_BK} that $\widetilde{H}^2_{\rm f}(G_{K,\Sigma},T_{f,\chi\psi};\Delta^{(i)})$ has finite cardinality as well.

If $L(f_K\otimes\chi^{-1}\psi^{-1},0)\neq 0$, it follows from \cite[Theorem 10.2.2]{KLZ2} (reciprocity laws for Beilinson--Flach elements) and the interpolation formulae for Hida's $p$-adic $L$-function (as given by Theorem~\ref{thm:sigma1interpolationformula} in Appendix~\ref{appendix:corrige} and \cite[Theorem~2.3]{BLForum}) that $\res_p^{+/(i)}\left(\BF^{f\otimes\chi\psi}_{\alpha_f,\alpha_{\chi\psi}}\right)$ is non-torsion. Note that the interpolation factors in the interpolation formulae do not vanish when $\psi \in \Sigma_{\rm crit}^{(i)}$, since in that case the weight of $\theta_{\chi\psi}$ is never equal to $k_f+2$.
\end{proof}

\begin{remark}
One expects that depending on the global root number for $f_K\otimes \chi$, either for all but finitely many $\psi \in \Sigma_{\rm crit}^{(1)}$, or else all but finitely many $\psi \in \Sigma_{\rm crit}^{(2)}$, we have $L(f_K\otimes\chi^{-1}\psi^{-1},0)\neq 0$.
\end{remark}


\subsection{Iwasawa Theory}
\label{subsec_split_ord}
Let us fix a cuspidal eigen-newform $f\in S_{k_f+2}(\Gamma_1(N_f),\epsilon_f)$ (with $k_f>0$) without CM and let $\f$ denote the unique branch of the Hida family of tame level $N_f$ that admits $f^{\alpha}$ as a specialization at weight $k_f+2$. As above, we have also fixed a ray class character $\chi$ of the imaginary quadratic field $K$ with conductor $\ff$ and such that $\chi^c\neq \chi$, where ${\mathbf N}\ff$, $N_f$ and $p$ are pairwise coprime. We assume that $p=\p\p^c$ splits in $K/\QQ$ and that $|D_K|$ is coprime to $N_f$.

\begin{defn}
\label{defn_good_twists_for_hyp}
Let $a\in \ZZ_{\geq 0}$ be the least non-negative integer such that $\chi$ verifies \ref{item_RIa}. For each $i=1,2$, we define a subset $\mathcal{S}^{(i)}_f(a)$ of the space ${\rm Sp}(\LL(\Gamma_K))(\cO)$ on setting 
$$\mathcal{S}^{(i)}_f(a):=\{\psi_p: \psi\in \Sigma^{(2)}_{\rm crit} \hbox{ is } \cO\hbox{-valued, }  \psi_p \hbox{ factors through } \Gamma_K \hbox{ and } X_{f,\chi\psi}  \hbox{ verifies \ref{item_HypXa}}\}.$$ 

We similarly define a subset $\mathcal{S}^{(i)}(a)$ of the space ${\rm Sp}(\LL_{\rm wt}\,\widehat{\otimes}\, \LL(\Gamma_K))(\cO)$ on setting 
$$\mathcal{S}^{(i)}(a):=\{(\kappa,\psi_p): \psi\in \Sigma^{(i)}_{\rm crit}(\kappa) \hbox{ is } \cO\hbox{-valued, }  \psi_p \hbox{ factors through } \Gamma_K \hbox{ and } X_{\f(\kappa),\chi\psi}  \hbox{ verifies \ref{item_HypXa}}\}.$$ 
\end{defn}

\begin{theorem}
\label{thm_good_twists_are_dense} Suppose that \ref{item_tau} holds true and $\overline{\rho}_{\f}$ is $p$-distinguished. 
\item[i)]The set $\mathcal{S}^{(2)}_f(a,0)$ is dense in ${\rm Sp}(\LL(\Gamma_K))(\cO)$.
\item[ii)] The set $\mathcal{S}^{(i)}(a,0)$ is dense in ${\rm Sp}(\LL_{\rm wt}\,\widehat{\otimes}\, \LL(\Gamma_K))(\cO)$. 

The conclusions of this theorem hold with $a=0$ if $\chi^c\not\equiv \chi \mod \frak{m}$.
\end{theorem}
\begin{proof}
We only provide the proof of (ii) as the proof of (i) is very similar. We will also prove only that $\mathcal{S}^{(2)}(a,0)$ is dense in ${\rm Sp}(\LL_{\rm wt}\,\widehat{\otimes}\, \LL(\Gamma_K))(\cO)$ since a small alteration to the argument below allows one to handle the remaining case as well.

In this proof, we shall set $X=X_{\f(\kappa),\chi\psi}$ for $\psi\in \Sigma^{(2)}_{\rm crit}(\kappa)$ $p$-crystalline. As explained in Proposition~\ref{prop_appendix_example}, \ref{item_HypXa}(iii) always holds true. It also follows from Proposition~\ref{prop_appendix_example} that \ref{item_HypXa}(i) and \ref{item_HypXa}(iv) hold true, since $\chi\psi\not\equiv \chi^c\psi^c \mod \frak{m}^{a+1}$. Finally, since we assumed \ref{item_tau}, it follows from \cite[Theorem C.2.4]{BLInertOrdMC} that \ref{item_HypXa}(ii) is valid as well.
The set of $(\kappa,\psi)$ satisfying required conditions is therefore dense in ${\rm Sp}(\LL_{\rm wt}\,\widehat{\otimes}\, \LL(\Gamma_K))$. 
\end{proof}

Before stating our Iwasawa theoretic results, we will explain in the next few paragraphs the relation between the Iwasawa theoretic extended Selmer groups that appear in our statements and their classical counterparts (i.e., Greenberg's Selmer groups). We introduce some notation with that purpose in mind.

\begin{defn}
For $i=1,2$, we define 
$$\frak{X}^{(i)}_{\rm BK}(K_\infty,T_{f,\chi}):= \left(\varinjlim_\alpha H^1_{\mathcal{F}_{{\rm BK}^{(i)}}^*}(K_\alpha,A_{f,\chi\psi}^*(1))\right)^{\vee}\otimes_{\LL_{\cO}(\Gamma_K)}\LL_{\cO}(\Gamma_K)^\iota$$
where the direct limit is over all finite extensions $K_\alpha/K$ contained in $K_\infty$, where we recall that $\LL_{\cO}(\Gamma_K)^\iota$ is the free $\LL_{\cO}(\Gamma_K)$-module of rank one on which $G_K$ acts via the composition $G_K\twoheadrightarrow\Gamma_K\xrightarrow{\gamma\mapsto\gamma^{-1}}\Gamma_K\hookrightarrow \LL_{\cO}(\Gamma_K)$. Similarly, we put
$$\frak{X}^{(i)}_{\rm sBK}(K_\infty,T_{f,\chi}):= \left(\varinjlim_\alpha H^1_{\mathcal{F}_{{\rm sBK}^{(i)}}^*}(K_\alpha,A_{f,\chi\psi}^*(1))\right)^{\vee}\otimes_{\LL_{\cO}(\Gamma_K)}\LL_{\cO}(\Gamma_K)^\iota\,.$$

For any $\ZZ_p$-extension $K_\Gamma$ of $K$ with Galois group $\Gamma$ and any Hecke character $\psi\in \Sigma_{\rm crit}^{(i)}$, we likewise define 
$$\frak{X}^{(i)}_{\rm BK}(K_\Gamma,T_{f,\chi}):= \left(\varinjlim_\alpha H^1_{\mathcal{F}^*_{{\rm BK}^{(i)}}}(K_\alpha,A_{f,\chi\psi}^*(1))\right)^{\vee}\otimes_{\LL_{\cO}(\Gamma)}\LL_{\cO}(\Gamma)^\iota$$
$$\frak{X}^{(i)}_{\rm sBK}(K_\Gamma,T_{f,\chi}):= \left(\varinjlim_\alpha H^1_{\mathcal{F}^*_{{\rm sBK}^{(i)}}}(K_\alpha,A_{f,\chi\psi}^*(1))\right)^{\vee}\otimes_{\LL_{\cO}(\Gamma)}\LL_{\cO}(\Gamma)^\iota$$
where now the direct limit is over all finite extensions $K_\alpha/K$ contained in $K_\Gamma$. 
\end{defn}

We are now ready to the record the relation between the Iwasawa theoretic extended Selmer groups that appear in our statements and their classical counterparts;  this comparison is essentially due to Nekov\'a\v{r}.
\begin{lemma}
\label{lemma_compare_Iwasawa_Nek_Iwasawa_Greenberg_K_infty}
For each $i=1,2$, we have an injective morphism
$$\frak{X}^{(i)}_{\rm BK}(K_\infty,T_{f,\chi})\lra \widetilde{H}^2_{\rm f}(G_{K,\Sigma},\TT_{f,\chi};\Delta^{(i)})$$
of $\LL_{\cO}(\Gamma_K)$-modules with pseudo-null cokernel.
\end{lemma}

\begin{proof}
For a given finite extension $K_\alpha/K$ contained in $K_\infty$, we have the following exact sequence:
\begin{align}
\begin{aligned}
   0\rightarrow H^1_{\mathcal{F}_{{\rm sBK}^{(i)}}}(K_\alpha,T_{f,\chi}) \lra H^1_{\mathcal{F}_{{\rm BK}^{(i)}}}&(K_\alpha,T_{f,\chi})\lra H^0(K_\alpha,{\rm Err})\\
\label{alignstar_strictvsususalBK_seq_bis_alpha}& \lra H^1_{\mathcal{F}_{{\rm sBK}^{(i)}}^*}(K_\alpha,A_{f,\chi}^*(1))^\vee \lra H^1_{\mathcal{F}_{{\rm BK}^{(i)}}^*}(K_\alpha,A_{f,\chi\psi}^*(1))^\vee \rightarrow 0,
\end{aligned}
\end{align}
where $H^0(K_\alpha,{\rm Err})$ is given by the exactness of the sequence
\begin{equation}
    \label{eqn_err_alpha_exact}
 0\lra \bigoplus_{\substack{v\in \Sigma\\ v\nmid p}}H^1_{\rm ur}(K_{\alpha,v}, T_{f,\chi}) \lra    \bigoplus_{\substack{v\in \Sigma\\ v\nmid p}}H^1_{\rm f}(K_{\alpha,v}, T_{f,\chi})\lra H^0(K_\alpha,{\rm Err})\lra 0\,.
\end{equation}
Passing to inverse limit in \eqref{eqn_err_alpha_exact} in $\alpha$, we deduce, thanks to \cite[Prop. B.1.1]{rubin00}, that 
\begin{equation}
    \label{eqn_err_alpha_exact_limit}
 0\lra \varprojlim_\alpha\bigoplus_{\substack{v\in \Sigma\\ v\nmid p}}H^1_{\rm ur}(K_{\alpha,v}, T_{f,\chi}) \lra   \varprojlim_\alpha \bigoplus_{\substack{v\in \Sigma\\ v\nmid p}} H^1_{\rm f}(K_{\alpha,v}, T_{f,\chi})\lra \varprojlim_\alpha H^0(K_\alpha,{\rm Err})\lra 0
\end{equation}
is also exact. Moreover, it follows from \cite[Prop. B.3.4]{rubin00} (see also the proof of \cite{mr02}, Lemma 5.3.1) that 
$\varprojlim_\alpha\bigoplus_{\substack{v\in \Sigma\\ v\nmid p}}H^1_{\rm ur}(K_{\alpha,v}, T_{f,\chi})=\varprojlim_\alpha\bigoplus_{\substack{v\in \Sigma\\ v\nmid p}}H^1(K_{\alpha,v}, T_{f,\chi})$. This fact combined with \eqref{eqn_err_alpha_exact_limit} shows that $ \varprojlim_\alpha H^0(K_\alpha,{\rm Err})=0$.

We now pass to limit in \eqref{alignstar_strictvsususalBK_seq_bis_alpha} to obtain the exact sequence
\begin{align}
\begin{aligned}
   0\rightarrow \varprojlim_\alpha H^1_{\mathcal{F}_{{\rm sBK}^{(i)}}}(K_\alpha,T_{f,\chi}) \rightarrow \varprojlim_\alpha & H^1_{\mathcal{F}_{{\rm BK}^{(i)}}}(K_\alpha,T_{f,\chi})\rightarrow \varprojlim_\alpha H^0(K_\alpha,{\rm Err})=0\\
\label{alignstar_strictvsususalBK_seq_bis_alpha_limit}& \rightarrow \varprojlim_\alpha H^1_{\mathcal{F}_{{\rm sBK}^{(i)}}^*}(K_\alpha,A_{f,\chi}^*(1))^\vee \rightarrow \varprojlim_\alpha H^1_{\mathcal{F}_{{\rm BK}^{(i)}}^*}(K_\alpha,A_{f,\chi\psi}^*(1))^\vee \rightarrow 0
\end{aligned}
\end{align}
which reduces to the pair of isomorphisms
$$\varprojlim_\alpha H^1_{\mathcal{F}_{{\rm sBK}^{(i)}}}(K_\alpha,T_{f,\chi})\xrightarrow{\sim} \varprojlim_\alpha H^1_{\mathcal{F}_{{\rm BK}^{(i)}}}(K_\alpha,T_{f,\chi})\,,$$
\begin{equation}
\label{eqn_strict_vs_non_strict_limit_pontryagin_duals}
\frak{X}^{(i)}_{\rm sBK}(K_\Gamma,T_{f,\chi})^\iota=\varprojlim_\alpha H^1_{\mathcal{F}_{{\rm sBK}^{(i)}}^*}(K_\alpha,A_{f,\chi}^*(1))^\vee \xrightarrow{\sim} \varprojlim_\alpha H^1_{\mathcal{F}_{{\rm BK}^{(i)}}^*}(K_\alpha,A_{f,\chi\psi}^*(1))^\vee=\frak{X}^{(i)}_{\rm BK}(K_\Gamma,T_{f,\chi})^\iota\,.
\end{equation}
Here the equalities on the left and the right extremes follow from definitions and the fact that Pontryagin duality is an equivalence of categories ${\bf LCA}^{\rm op}\to {\bf LCA}$, where ${\bf LCA}$ is the category of locally compact abelian groups and ${\bf LCA}^{\rm op}$ is its opposite. 

On the other hand, Nekov\'a\v{r}'s ``Matlis duality'' \cite[(8.9.10)]{nekovar06} shows that we have an isomorphism
\begin{equation}
\label{eqn_Nekovar_Matlis}
    \widetilde{H}^1_{\rm f}(K_S/K_\infty,A_{f,\chi}^*(1);\Delta^{(i)})^{\vee}\otimes_{\LL_{\cO}(\Gamma_K)}\LL_{\cO}(\Gamma_K)^\iota\stackrel{\sim}{\lra} \widetilde{H}^2_{\rm f}(G_{K,\Sigma},\TT_{f,\chi};\Delta^{(i)})\,.
\end{equation}
Finally, \cite[Prop. 9.6.6]{nekovar06} shows that we have a natural injection
\begin{equation}
\label{eqn_Nekovar_Prop966}
   \frak{X}^{(i)}_{\rm sBK}(K_\Gamma,T_{f,\chi})^\iota\xrightarrow{\beta^\vee} \widetilde{H}^1_{\rm f}(K_S/K_\infty,A_{f,\chi}^*(1);\Delta^{(i)})^{\vee}
    \end{equation}
    with pseudo-null cokernel.
    
    Putting everything together, we conclude that the composition of the arrows
    \begin{align*}
        \frak{X}^{(i)}_{\rm BK}(K_\Gamma,T_{f,\chi})\xrightarrow[\sim] {\eqref{eqn_strict_vs_non_strict_limit_pontryagin_duals}^{-1}}
        \frak{X}^{(i)}_{\rm sBK}(K_\Gamma,T_{f,\chi})\xrightarrow[1-1]{\eqref{eqn_Nekovar_Prop966}} \widetilde{H}^1_{\rm f}(K_S/K_\infty,A_{f,\chi}^*(1);\Delta^{(i)})^{\vee}& \otimes_{\LL_{\cO}(\Gamma_K)}\LL_{\cO}(\Gamma_K)^\iota\\
    & \xrightarrow[\sim]{\eqref{eqn_Nekovar_Matlis}} \widetilde{H}^2_{\rm f}(G_{K,\Sigma},\TT_{f,\chi};\Delta^{(i)})
    \end{align*}
    is injective with pseudo-null cokernel, as required.
\end{proof}

We now state our first Iwasawa theoretic result.
\begin{theorem}
\label{thm_ordinary_Delta1_Selmer_torsion_via_horizontal_ES} 
Suppose $f \in S_{k_f+2}(\Gamma_1(N_f),\epsilon_f)$ is a $p$-ordinary cuspidal eigen-newform and $\chi$ is a ray class character of $K$ verifying \ref{item_RIa} for some natural number $a$. Let $\f$ be the unique branch of the Hida family that admits the slope-zero $p$-stabilized form $f^\alpha$ as a specialization in weight $k_f+2$. Suppose that \ref{item_tau} holds and $\overline{\rho}_{\f}$ is $p$-distinguished. 
\item[i)] All four $\LL_{\cO}(\Gamma_K)$-modules $\widetilde{H}^1_{\rm f}(G_{K,\Sigma},\TT_{f,\chi};\Delta^{(i)})$ and $\widetilde{H}^2_{\rm f}(G_{K,\Sigma},\TT_{f,\chi};\Delta^{(i)})$ ($i=1,2$) are torsion.
\item[ii)] All four $\LL_{\f}(\Gamma_K)$-modules $\widetilde{H}^1_{\rm f}(G_{K,\Sigma},\TT_{\f,\chi};\Delta^{(1)})$ and $\widetilde{H}^2_{\rm f}(G_{K,\Sigma},\TT_{\f,\chi};\Delta^{(i)})$ ($i=1,2$) are torsion.
\end{theorem}
\begin{proof}
The argument using the Poitou--Tate exact sequences in the proof of \cite[Theorem 3.8]{castellawanGL2families2018}  (adjusted to take into account the correction factors $\mathscr C(u)$ and the comparison in Lemma~\ref{lemma_compare_Iwasawa_Nek_Iwasawa_Greenberg_K_infty} of our extended Selmer groups to their classical variants that are used in op. cit.) applies to prove that the $\LL_{\cO}(\Gamma_K)$-module $\widetilde{H}^2_{\rm f}(G_{K,\Sigma},\TT_{f,\chi};\Delta^{(1)})$ is torsion if and only if the $\LL_{\cO}(\Gamma_K)$-module $\widetilde{H}^2_{\rm f}(G_{K,\Sigma},\TT_{f,\chi};\Delta^{(2)})$ is.  

Similarly, the $\LL_\f(\Gamma_K)$-module $\widetilde{H}^2_{\rm f}(G_{K,\Sigma},\TT_{\f,\chi};\Delta^{(1)})$ is torsion if and only if $\widetilde{H}^2_{\rm f}(G_{K,\Sigma},\TT_{\f,\chi};\Delta^{(2)})$ is. We remark that the dual Selmer groups ${\frak X}_{\emptyset,0}(\mathcal{K}_\infty,A_\f^\dagger)$ and   ${\frak X}_{\mathbf{Gr}}(\mathcal{K}_\infty,A_\f^\dagger)$ in op. cit. can be replaced by $\widetilde{H}^2_{\rm f}(G_{K,\Sigma},\TT_{\f,\chi};\Delta^{(2)})$ and $\widetilde{H}^2_{\rm f}(G_{K,\Sigma},\TT_{\f,\chi};\Delta^{(1)})$ respectively (c.f. Lemma~\ref{lemma_compare_Iwasawa_Nek_Iwasawa_Greenberg_K_infty}).

Moreover, it follows from global Euler--Poincar\'e characteristic formulae that $\widetilde{H}^2_{\rm f}(G_{K,\Sigma},\TT_{f,\chi};\Delta^{(i)})$ (resp., $\widetilde{H}^2_{\rm f}(G_{K,\Sigma},\TT_{\f,\chi};\Delta^{(i)})$) is torsion if and only if $\widetilde{H}^1_{\rm f}(G_{K,\Sigma},\TT_{f,\chi};\Delta^{(i)})$ (resp., $\widetilde{H}^1_{\rm f}(G_{K,\Sigma},\TT_{\f,\chi};\Delta^{(i)})$) is.
\item[i)] One may choose a $p$-crystalline character $\psi\in \Sigma^{(2)}_{\rm crit}$ with the property that the Euler product that defines the Rankin--Selberg $L$-series $L(f_{/K}\otimes\chi^{-1}\psi^{-1},s)$ absolutely converges at $s=0$, and therefore $L(f_{/K}\otimes\chi^{-1}\psi^{-1},0)\neq 0$ (a sufficient condition is that if $(\ell_1,\ell_2)$ is infinity-type of $\psi$, then $\ell=|\ell_1-\ell_2|\geq k_f+2$ and $\ell_0=\min(\ell_1,\ell_2)<-(k_f+\ell)/2$). It follows from Theorem~\ref{thm_BK_type_result} that the Selmer group $\widetilde{H}^2_{\rm f}(G_{K,\Sigma},T_{f,\chi\psi};\Delta^{(2)})$ has finite cardinality, and from the control theorem for Selmer complexes that $\widetilde{H}^2_{\rm f}(G_{K,\Sigma},\TT_{f,\chi};\Delta^{(2)})$ (and by the discussion above, also $\widetilde{H}^2_{\rm f}(G_{K,\Sigma},\TT_{f,\chi};\Delta^{(1)})$ as well as $\widetilde{H}^1_{\rm f}(G_{K,\Sigma},\TT_{f,\chi};\Delta^{(i)})$ for $i=1,2$) is torsion.
\item[ii)] Under the running hypotheses, there exists a crystalline specialization $\f(\kappa)$ of the Hida family $\f$ so that Part (i) applies with $f^\alpha=\f(\kappa)$. The required conclusions follow again from the control theorems for Selmer complexes.
\end{proof}

\begin{remark}
If $\chi=\chi^c$, so that $\chi$ is a character of $G_\QQ$, the conclusions of Theorem~\ref{thm_ordinary_Delta1_Selmer_torsion_via_horizontal_ES}  follow from well-known results in cyclotomic Iwasawa theory of Kato \cite{kato04} and Rohrlich \cite{rohrlich88Annalen}.
\end{remark}

\begin{remark}
Under the hypotheses of Theorem~\ref{thm_ordinary_Delta1_Selmer_torsion_via_horizontal_ES}, the argument in the proof of Theorem~\ref{thm_ordinary_Delta1_Selmer_torsion_via_horizontal_ES} may be altered slightly to prove the following twisted versions of the assertions therein: For $i=1,2$,
\begin{equation}
    \label{eqn_twisted_torsion_Gamma_K_1}
    \LL_{\cO}(\Gamma_K)\hbox{-modules } \widetilde{H}^1_{\rm f}(G_{K,\Sigma},\TT_{f,\chi}^\dagger;\Delta^{(i)})
    \hbox{ and } \widetilde{H}^2_{\rm f}(G_{K,\Sigma},\TT_{f,\chi}^\dagger;\Delta^{(i)}) \hbox{ are torsion\,;}
\end{equation}
\begin{equation}
    \label{eqn_twisted_torsion_families_Gamma_K_2}
    \LL_{\f}(\Gamma_K)\hbox{-modules } \widetilde{H}^1_{\rm f}(G_{K,\Sigma},\TT_{\f,\chi}^\dagger;\Delta^{(i)})
    \hbox{ and } \widetilde{H}^2_{\rm f}(G_{K,\Sigma},\TT_{\f,\chi}^\dagger;\Delta^{(i)}) \hbox{ are torsion.}
\end{equation}

\end{remark}



Suppose that $R$ is a UFD. Recall that an $R$-module $M$ is said to be pseudo-null if every $m\in M$ is annihilated by two coprime elements. We will need the following auxiliary result from commutative algebra. 

\begin{lemma}
\label{lemma_pseduo_null_quotients}
Suppose $R$ is a UFD and $N\supset M$ are finitely generated $R$-modules where $M$ is free. If $N/M$ contains a non-zero pseudo-null $R$-submodule, then so does $N$. 
\end{lemma}
\begin{proof}
By induction on the rank of $M$, we may assume without loss of generality that $M$  has rank one over $R$ and is given by $M=Rm$ for some $m\in N$. Suppose $N/M$ contains a non-zero pseudo-null $R$-submodule and $n+M \in N/M$ belongs to that submodule. It follows that there exists coprime elements $r_1,r_2\in R$ with $r_1n,r_2n\in M= Rm$. Thence, either $r_1n=0=r_2n$, or else $m$ is annihilated by some non-zero element of $R$. Since $M$ is free, the latter scenario is impossible. It follows that $n\in N$ is annihilated by two coprime elements of $R$ and therefore it contains a non-zero pseudo-null submodule.
\end{proof}


\begin{theorem}
\label{thm_ordinary_Deltai_MC_via_horizontal_ES}
Suppose $f \in S_{k_f+2}(\Gamma_1(N_f),\epsilon_f)$ is a cuspidal eigen-newform and $\chi$ is a ray class character of $K$ verifying \ref{item_RIa}. Let $\f$ be the unique branch of the Hida family that admits $f^\alpha$ as a specialization in weight $k_f+2$. Suppose that \ref{item_tau} holds and $\overline{\rho}_{\f}$ is $p$-distinguished. Fix a morphism $u$ as in \S\ref{subsubsec_local_properties_CMgg}.
\item[i)] For some $a_1,a_2\in \ZZ$ we have
\begin{align*}
    \Char_{\LL_{\cO}(\Gamma_K)}\left(\widetilde{H}^2_{\rm f}(G_{K,\Sigma},\TT_{f,\chi};\Delta^{(1)})\right)&\,\Big{|}\,\, \varpi^{a_1}{\LL_{\cO}(\Gamma_K)} L_p(f_{/K}\otimes\chi,\Sigma^{(1)}_{\rm crit})\cdot \mathscr{C}(u)\, ,\\
\Char_{\LL_{\cO}(\Gamma_K)}\left(\widetilde{H}^2_{\rm f}(G_{K,\Sigma},\TT_{f,\chi};\Delta^{(2)})\right){\otimes_\cO\cO_\Phi}&\,\Big{|}\,\, \varpi^{a_2}H_\chi{\Lambda_{\cO_\Phi}(\Gamma_K)} L_p(f_{/K}\otimes\chi,\Sigma^{(2)}_{\rm crit})\cdot \mathscr{C}(u)\,.
\end{align*}
Here, $H_\chi \in \LL_{
\Phi}(\Gamma_\ac)$ is a generator of the congruence module associated the CM form $\theta_\chi$.
\item[ii)] Suppose that $\LL_\f$ is regular. For some $b_1,b_2\in \ZZ$ we have
\begin{align*}
\Char_{\LL_\f(\Gamma_K)}\left(\widetilde{H}^2_{\rm f}(G_{K,\Sigma},\TT_{\f,\chi};\Delta^{(1)})\right)&\,\Big{|}\,\, \varpi^{b_1}H_{\f}{\LL_{\f}(\Gamma_K)}L_p(\f_{/K}\otimes\chi,{\mathbb \Sigma}^{(1)}_{\rm crit})\cdot \mathscr{C}(u)\, \\
\Char_{\LL_{\f}(\Gamma_K)}\left(\widetilde{H}^2_{\rm f}(G_{K,\Sigma},\TT_{\f,\chi};\Delta^{(2)})\right){\otimes_\cO\cO_\Phi}&\,\Big{|}\,\, \varpi^{b_2}H_\chi{\Lambda_{\f,\cO_\Phi}(\Gamma_K)} L_p(\f_{/K}\otimes\chi,{\mathbb \Sigma}^{(2)}_{\rm crit})\cdot \mathscr{C}(u)\,.
\end{align*}
Here, $H_\f \in \LL_{\f}$ is a generator of Hida's congruence ideal.
\end{theorem}

One may drop the correction terms $\mathscr{C}(u)$ from the statement of Theorem~\ref{thm_ordinary_Deltai_MC_via_horizontal_ES} if the hypothesis \ref{item_Ind} concerning the CM branches of Hida families holds. In particular, if \ref{item_RIa} is valid with $a=0$, then we can take $\mathscr C(u)=(1)$ thanks to Proposition~\ref{prop_uniqueness_of_the_lattice_in_residually_irred_case}.

\begin{proof}[Proof of Theorem~\ref{thm_ordinary_Deltai_MC_via_horizontal_ES}] Let us fix a $\ZZ_p$-basis $\{\gamma_1,\gamma_2\}$ of $\Gamma_K$. In this proof, we will use $\psi$ to denote the Galois character associated to $\psi$ as well. We put $\Gamma_1:=\gamma_1^{\ZZ_p}$ and $\Gamma_2:=\gamma_2^{\ZZ_p}$, so that $\Gamma_K=\Gamma_1\times\Gamma_2$. We will identify $\Gamma_1$ with the Galois group of the $\ZZ_p$-extension $K_{\Gamma_1}:=K_\infty^{\Gamma_2}$ of $K$, and similarly for $\Gamma_2$.

For $\psi\in {\rm Sp}(\LL(\Gamma_K))(\cO)$, we write $\pi_1^{\psi}: \LL_{\cO}(\Gamma_K)\rightarrow \LL_{\cO}(\Gamma_2)$ for the morphism induced by $\gamma_1\mapsto \psi(\gamma_1)$ and $\gamma_2\mapsto \psi(\gamma_2)\gamma_2$. We also put $\pi_2: \LL_{\cO}(\Gamma_2)\rightarrow \cO$ for the augmentation morphism induced by $\gamma_2\mapsto1$. These give rise to $G_K$-maps 
$$\TT_{f,\chi}\stackrel{\pi_1^{\psi}}{\lra} T_{f,\chi}\otimes{\rm Tw}_\psi\left(\LL_{\cO}(\Gamma_2)^\sharp\right)\stackrel{\pi_2}{\lra}T_{f,\chi\psi}$$
where ${\rm Tw}_\psi\left(\LL_{\cO}(\Gamma_2)^\sharp\right)$ is the free $\LL_{\cO}(\Gamma_2)$-module of rank one on which $G_K$ acts through $\Gamma_2 \rightarrow \LL_{\cO}(\Gamma_2)^\times$, $\gamma_2\mapsto \psi(\gamma_2)\gamma_2$. We set $\pi_\psi:= \pi_2\circ \pi_1^\psi$.
\item[i)] It follows from Theorem~\ref{thm_BK_type_result} combined with the generic non-vanishing argument we have utilized in the proof of Theorem~\ref{thm_ordinary_Delta1_Selmer_torsion_via_horizontal_ES}(i) and the proof of Theorem~\ref{thm_good_twists_are_dense}(i) that the subset $\mathcal{S}_f^{\prime}\subset \mathcal{S}_f^{(2)}(a)$ consisting of the $p$-adic Galois characters associated to $\cO$-valued $p$-crystalline characters $\psi \in \Sigma_{\rm crit}$ for which the $\cO$-module $\widetilde{H}^2_{\rm f}(G_{K,\Sigma},\TT_{f,\chi\psi};\Delta^{(2)})$ has finite cardinality is still dense in ${\rm Sp}(\LL(\Gamma_K))(\cO)$. 

Let us fix a generator 
$$h\in \Char_{\LL_{\cO}(\Gamma_K)}\left(\widetilde{H}^2_{\rm f}(G_{K,\Sigma},\TT_{f,\chi};\Delta^{(2)})\right)\stackrel{\rm Lemma~\ref{lemma_compare_Iwasawa_Nek_Iwasawa_Greenberg_K_infty}}{=}\Char_{\LL_{\cO}(\Gamma_K)}\left(\frak{X}^{(2)}_{\rm BK}(K_\infty, T_{f,\chi})\right).$$ 
For each $\psi \in \mathcal{S}_f^{\prime}$ we have $\pi_\psi(h)\neq 0$; in particular, $\pi_1^\psi(h)\in \LL_{\cO}(\Gamma_2)$ is non-zero. In particular, the height-one prime $\ker(\pi_1^\psi)=(\gamma_1-\psi(\gamma_1))$ is not in the support of the ${\LL_{\cO}(\Gamma_K)}$-module $\frak{X}^{(2)}_{\rm BK}(K_\infty, T_{f,\chi})$. Thence, $\frak{X}^{(2)}_{\rm BK}(K_\infty, T_{f,\chi})[\ker(\pi_1^\psi)]$ is contained in the maximal pseudo-null ${\LL_{\cO}(\Gamma_K)}$-submodule $M_{f,\chi}$ of $\frak{X}^{(2)}_{\rm BK}(K_\infty, T_{f,\chi})$ and we have 
\begin{align*} 
\LL_{\cO}(\Gamma_2)\cdot\pi_1^\psi(h)&=\Char_{\LL_{\cO}(\Gamma_2)}\left(\frak{X}^{(2)}_{\rm BK}(K_\infty, T_{f,\chi})\big{/}\ker(\pi_1^\psi)\frak{X}^{(2)}_{\rm BK}(K_\infty, T_{f,\chi})\right)\\
&=\Char_{\LL_{\cO}(\Gamma_2)}\left(\frak{X}^{(2)}_{\rm BK}(K_{\Gamma_2}, T_{f,\chi\psi})\right)\cdot \Char_{\LL_\cO(\Gamma_2)}(M_{f,\chi})^{-1}\,.
\end{align*}
The same argument also shows that the height one prime $\ker(\pi_2)=(\gamma_2-1)\subset\LL_{\cO}(\Gamma_2)$ is not in the support of $\frak{X}^{(2)}_{\rm BK}(K_{\Gamma_2}, T_{f,\chi\psi})$ and that we have
\begin{align} 
\label{eqn_char_ideal_reduce_pipsi}
\cO\cdot \psi(h)=\cO\cdot\pi_2\circ\pi_1^\psi(h)\,\,\big{|}\,\,{\rm Fitt}_{\cO}\left(\frak{X}^{(2)}_{\rm BK}(K, T_{f,\chi\psi})\right),
\end{align}
Combining \eqref{eqn_char_ideal_reduce_pipsi} with Theorem~\ref{thm_Split_BK_Sigma12}(ii), 
\begin{equation}
\label{eqn_eqn_char_ideal_reduce_pipsi_2}
\psi(h)\,\,\,\,\,\Big{|}\,\,\,\,\,\varpi^{c_1}\left[H^1(K_{\p^c},F^+R_f^*(\chi\psi)):\cO\cdot \res_{\p^c}^{(2)}\left({}^u\BF^{f\otimes\chi\psi}_{\alpha_f,\alpha_{\chi\psi}}\right)\right]
\end{equation}
for some $c_1\in \ZZ$ that doesn't depend on the choice of $\psi$.

Let us now fix a generator
$$b\in \Char_{\LL_{\cO}(\Gamma_K)}\left(H^1(K_{\p^c},F^+R_f^*(\chi)\otimes_{\ZZ_p}\LL(\Gamma_K)^{\sharp})\Big{/}\LL_{\cO}(\Gamma_K)\cdot \res_{\p^c}^{(2)}\left({}^u{\mathbb{BF}}_{f^{\alpha},\chi}\right)\right).$$ 
According to \cite[Lemma 8.11.3(ii)]{nekovar06}, we have the following scenarios for the sctructure of the $\LL_{\cO}(\Gamma_K)$-module $H^2(K_{\p^c},F^+R_f^*(\chi)\otimes_{\ZZ_p}\LL(\Gamma_K)^{\sharp})$: 
\begin{itemize}
    \item It has finite cardinality.
    \item It is a free $\cO$-module of rank one.
\end{itemize} 
Furthermore, the latter scenario would occur only if $\alpha_f\chi(\p^c)=1$, which is impossible since the complex conjugate of this expression has valuation $k_f+1>0$ in $\cO$. This shows that $H^2(K_{\p^c},F^+R_f^*(\chi)\otimes_{\ZZ_p}\LL(\Gamma_K)^{\sharp})$ has finite cardinality. Moreover, \cite[8.11.4(iii)]{nekovar06} tells us that $H^1(K_{\p^c},F^+R_f^*(\chi)\otimes_{\ZZ_p}\LL(\Gamma_K)^{\sharp})$ is torsion-free. This combined with Lemma~\ref{lemma_pseduo_null_quotients} shows that the quotient 
$$H^1(K_{\p^c},F^+R_f^*(\chi)\otimes_{\ZZ_p}\LL(\Gamma_K)^{\sharp})\Big{/}\LL_{\cO}(\Gamma_K)\cdot \res_{\p^c}^{(2)}\left({}^u{\mathbb{BF}}_{f^{\alpha},\chi}\right)$$ 
has no pseudo-null submodules. Combining all these facts, we infer that  
$$\Char_{\LL_{\cO}(\Gamma_2)}\left(\mathscr{H}(\Gamma_2)\right)=\pi_1^\psi(b)$$
for all $\psi \in \mathcal{S}_f^\prime$, where we have put 
$$\mathscr{H}(\Gamma_2):=H^1(K_{\p^c},F^+R_f^*(\chi)\otimes_{\ZZ_p}{\rm Tw}_\psi\left(\LL_{\cO}(\Gamma_2)^\sharp\right))\Big{/}\LL_{\cO}(\Gamma_2)\cdot \pi_1^{\psi}\left(\res_{\p^c}^{(2)}\left({}^u{\mathbb{BF}}_{f^{\alpha},\chi}\right)\right)$$ 
to ease notation. Thence,
\begin{align}\label{eqn_descend_b_viapipsi}
\notag \cO\cdot \psi(b)&=\pi_1\left(\Char_{\LL_{\cO}(\Gamma_2)}\left(\mathscr{H}(\Gamma_2)\right)\right)\\
\notag &={\rm Fitt}_{\cO}\left(\mathscr{H}(\Gamma_2)/(\gamma_2-1)\mathscr{H}(\Gamma_2) \right)\\
&= {\rm Fitt}_{\cO}\left(H^1(K_{\p^c},F^+R_f^*(\chi\psi))\big{/} \cO\cdot \res_{\p^c}^{(2)}\left({}^u\BF^{f\otimes\chi\psi}_{\alpha_f,\alpha_{\chi\psi}}\right)\right)\cdot {\rm Fitt}_{\cO}(\mathscr{H}_{\p^c}(\psi))^{-1}
\end{align}
for all $\psi \in \mathcal{S}_f^\prime$, where the second equality is a special case of \cite[Lemma 11.6.8]{nekovar06}, the $\cO$-module $\mathscr{H}_{\p^c}$ on the third line is a subquotient of $H^2(K_{\p^c},F^+R_f^*(\chi)\otimes_{\ZZ_p}\LL(\Gamma_K)^{\sharp})$ which is given by the exactness of the sequence
$$0\lra \frac{H^1(K_{\p^c},F^+R_f^*(\chi)\otimes_{\ZZ_p}{\rm Tw}_\psi\left(\LL_{\cO}(\Gamma_2)^\sharp\right))}{(\gamma_2-1)H^1(K_{\p^c},F^+R_f^*(\chi)\otimes_{\ZZ_p}{\rm Tw}_\psi\left(\LL_{\cO}(\Gamma_2)^\sharp\right))}\stackrel{\pi_1}{\lra} H^1(K_{\p^c},F^+R_f^*(\chi\psi))\lra \mathscr{H}_{\p^c}(\psi)\lra 0,$$ 
from which the third equality follows. It follows on combining \eqref{eqn_eqn_char_ideal_reduce_pipsi_2} with \eqref{eqn_descend_b_viapipsi} that 
\begin{equation}
\label{eqn_specialised_sought_after_divisibility}
\psi(h)\,\, \mid\,\, \varpi^{c_2} \psi(b)
\end{equation}
for all but finitely many $\psi \in \mathcal{S}_f^\prime$ (avoiding the finite set of characters $\xi\in \mathcal{S}_f^\prime$ for which the prime ideal $(\gamma_2-\xi(\gamma_2))$ is in the support of the torsion $\LL_{\cO}(\Gamma_2)$-module $H^2(K_{\p^c},F^+R_f^*(\chi)\otimes_{\ZZ_p}\LL(\Gamma_2)^{\sharp}$), so that the cardinality of $\mathscr{H}_{\p^c}(\psi)$ is bounded independently of the choice of $\psi$), for some $c_2\in \ZZ$ that does not depend on such $\psi$. Proposition A.0.1 of \cite{BLInertOrdMC} together with the density of $\mathcal{S}_f^\prime$ in ${\rm Sp}(\LL(\Gamma_K))(\cO)$ and \eqref{eqn_specialised_sought_after_divisibility} show that 
 $$ h \mid \varpi^{c_2} b,$$ 
hence also that
\begin{align}
\label{eqn_d_divides_b_2021_09_20}
\begin{aligned}
   &\Char_{\LL_{\cO}(\Gamma_K)}\left(\widetilde{H}^2_{\rm f}(G_{K,\Sigma},\TT_{f,\chi};\Delta^{(2)})\right)= \LL_{\cO}(\Gamma_K) h \qquad \hbox{divides} \\
&\qquad\qquad\qquad \qquad\qquad\qquad \LL_{\cO}(\Gamma_K)\varpi^{c_2} b=\varpi^{c_2} \Char_{\LL_{\cO}(\Gamma_K)}\left(\frac{H^1(K_{\p^c},F^+R_f^*(\chi)\otimes_{\ZZ_p}\LL(\Gamma_K)^{\sharp})}{\LL_{\cO}(\Gamma_K)\cdot \res_{\p^c}^{(2)}\left({}^u{\mathbb{BF}}_{f^{\alpha},\chi}\right)}\right)
\end{aligned}
\end{align} 
 for some $c_2\in \ZZ$. The proof that 
\begin{equation}\label{eqn_proved_divisibility_thm619}
    \Char_{\LL_{\cO}(\Gamma_K)}\left(\widetilde{H}^2_{\rm f}(G_{K,\Sigma},\TT_{f,\chi};\Delta^{(2)})\right){\otimes_\cO\cO_\Phi}\,\Big{|}\,\, \varpi^{a_2}H_\chi L_p(f_{/K}\otimes\chi,\Sigma^{(2)}_{\rm crit})\cdot \mathscr{C}(u)\cdot{\Lambda_{\cO_\Phi}(\Gamma_K)}
\end{equation}
for some $a_2\in \ZZ$ follows from  \eqref{eqn_d_divides_b_2021_09_20} and Proposition~\ref{prop_reciprocity_law}(i). 

 The argument using Poitou--Tate exact sequences in the proof of \cite[Theorem 3.8]{castellawanGL2families2018}, with the divisibility \eqref{eqn_proved_divisibility_thm619} and   Proposition~\ref{prop_reciprocity_law} as  input, taking into account the correction factor $\mathscr{C}(u)$ and the comparison in Lemma~\ref{lemma_compare_Iwasawa_Nek_Iwasawa_Greenberg_K_infty} of our extended Selmer groups to their classical variants that are used in op. cit.,  proves the divisibility 
$$\Char_{\LL_{\cO}(\Gamma_K)}\left(\widetilde{H}^2_{\rm f}(G_{K,\Sigma},\TT_{f,\chi};\Delta^{(1)})\right)\,\Big{|}\,\, \varpi^{a_1} L_p(f_{/K}\otimes\chi,\Sigma^{(1)}_{\rm crit})\cdot \mathscr{C}(u)\cdot \LL_{\cO}(\Gamma_K)$$
for some $a_1\in\ZZ$, concluding the proof of (i).


\item[ii)] This part is proved in a similar way, except that we will verify the first divisibility (rather than the second) 
$$\LL_\f(\Gamma_K)\cdot\mathbb{h}=:\Char_{\LL_\f(\Gamma_K)}\left(\widetilde{H}^2_{\rm f}(G_{K,\Sigma},\TT_{\f,\chi};\Delta^{(1)})\right)\,\Big{|}\,\, \varpi^{b_1}H_{\f}L_p(\f_{/K}\otimes\chi,{\mathbb \Sigma}^{(1)}_{\rm crit})\cdot \mathscr{C}(u)\cdot\LL_\f(\Gamma_K)$$
for some $b_1\in \ZZ$. 

 It follows from Theorem~\ref{thm_ordinary_Delta1_Selmer_torsion_via_horizontal_ES}(ii) that the height-one prime $\ker(\kappa) \subset \LL_{\f}(\Gamma_K)$ (where $\kappa \in {\rm Sp}(\LL(\Gamma_{\rm wt}))(\cO)$ is any crystalline arithmetic specialization) is not in the support of $\widetilde{H}^2_{\rm f}(G_{K,\Sigma},\TT_{\f,\chi};\Delta^{(1)})$ and 
 $$\LL_{\cO}(\Gamma_K)\cdot\kappa(\mathbb{h})\,\, \Big{|}\,\,\Char_{\LL_\cO(\Gamma_K)}\left(\widetilde{H}^2_{\rm f}(G_{K,\Sigma},\TT_{\f(\kappa),\chi};\Delta^{(1)})\right)=\Char_{\LL_\cO(\Gamma_K)}\left(\frak{X}^{(1)}_{\rm BK}(K_\infty,T_{\f(\kappa),\chi})\right)\,.$$
Similar line of argument shows 
$$\cO\cdot \mathbb{h}(\kappa,\psi) \,\,\Big{|}\,\,{\rm Fitt}_{\cO}\left(\frak{X}^{(1)}_{\rm BK}(K,T_{\f(\kappa),\chi\psi})\right)$$
for every $(\kappa,\psi) \in \mathcal{S}^{(1)}(a)$. This combined with Theorem~\ref{thm_Split_BK_Sigma12}(i) yields the divisibility
\begin{equation}
\label{eqn_descent_from_weight_space_1}
    \cO\cdot \mathbb{h}(\kappa,\psi) \,\,\Big{|}\,\, \varpi^{t} \,{\rm Fitt}_{\cO}\left(H^1(K_\p,F^-R_{\f(\kappa)}^*(\chi\psi))\big{/}\cO\cdot \pi_{\psi}\circ\kappa\left(\res_{\p^c}^{(1)}\left({}^u{\mathbb{BF}}_{\f,\chi}\right)\right)\right)
\end{equation}
for some $t\in \ZZ$ that doesn't depend on the choice of $(\kappa,\psi)\in \mathcal{S}^{(1)}(a)$. 

Let us now choose any generator 
$$\mathbb{b} \in \Char_{\LL_\f\,\widehat{\otimes}\,\LL(\Gamma_K)}\left(H^1(K_{\p},F^-R_{\f}^*(\chi)\otimes\LL(\Gamma_K)^\sharp)\Big{/}(\LL_\f\,\widehat{\otimes}\,\LL(\Gamma_K))\cdot\res_{\p}\left({}^u\mathbb{BF}_{\f,\chi}\right) \right).$$ 
Since the $G_{K_{\p}}$-representation $F^-R_{\f}^*(\chi)$ is unramified, it follows that $H^1(K_{\p},F^-R_{\f}^*(\chi)\otimes\LL(\Gamma_K)^\sharp)$ is a free  $\LL_{\f}(\Gamma_K)$-module of rank one and the natural morphism
$$H^1(K_{\p},F^-R_{\f}^*(\chi)\otimes\LL(\Gamma_K)^\sharp)\stackrel{\pi_\psi\circ \kappa}{\lra} H^1(K_{\p},F^-R_{\f(\kappa)}^*(\chi\psi))$$
is surjective. Thence,
\begin{align*}
\Char_{\LL_\f\,\widehat{\otimes}\,\LL(\Gamma_K)}&\left(H^1(K_{\p},F^-R_{\f}^*(\chi)\otimes\LL(\Gamma_K)^\sharp)\Big{/}(\LL_\f\,\widehat{\otimes}\,\LL(\Gamma_K))\cdot\res_{\p}\left({}^u\mathbb{BF}_{\f,\chi}\right) \right)\\
&={\rm Fitt}_{\LL_\f\,\widehat{\otimes}\,\LL(\Gamma_K)}\left(H^1(K_{\p},F^-R_{\f}^*(\chi)\otimes\LL(\Gamma_K)^\sharp)\Big{/}(\LL_\f\,\widehat{\otimes}\,\LL(\Gamma_K))\cdot\res_{\p}\left({}^u\mathbb{BF}_{\f,\chi}\right) \right)
\end{align*}
and $\mathbb{b}(\kappa,\psi)\in \cO$ generates the ideal ${\rm Fitt}_{\cO}\left(H^1(K_\p,F^-R_{\f(\kappa)}^*(\chi\psi))\big{/}\cO\cdot \pi_{\psi}\circ\kappa\left(\res_{\p^c}^{(1)}\left({}^u{\mathbb{BF}}_{\f,\chi}\right)\right)\right)$. This combined with \eqref{eqn_descent_from_weight_space_1} yields
$$\mathbb{h}(\kappa,\psi) \mid \varpi^t \mathbb{b}(\kappa,\psi)$$
for all $(\kappa,\psi)\in \mathcal{S}^{(1)}(a) $ and for some $t\in \ZZ$ that doesn't depend on the choice of $(\kappa,\psi)$. The proof of Part (ii) now follows on combining \cite[Proposition A.0.1]{BLInertOrdMC} and Proposition~\ref{prop_reciprocity_law}(ii).
\end{proof}


\subsection{Descent}
\label{subsec_descent_to_Gamma}
Throughout this section, we will assume that $\Gamma\in {\rm Gr}(\Gamma_K)$ is transversal to $\Gamma_\cyc$, in the sense that the associated $\ZZ_p$-extension $K_\Gamma/K$ is linearly disjoint over $K$ with the cyclotomic $\ZZ_p$-extension $K^\cyc/K$.
\subsubsection{Nekov\'a\v{r}'s descent formalism}
\label{subsubsec_Nek_descent}
We recall Nekov\'a\v{r}'s descent formalism that  developed in \cite{nekovar06} and make some of the constructions therein explicit. We will later apply this discussion with $X=\TT_{f,\chi}^{(\Gamma)}$, $R=\LL_{\cO}(\Gamma)$ and $\TT_{\f,\chi}^{(\Gamma)}$, $R=\LL_\f(\Gamma)$, where $\Gamma$ is transversal to $\Gamma_\cyc$ in the sense explained above.

\begin{defn}
\label{defn_general_derived_char_ideal}
\item[i)]For $R$ and $X$ as in Definition~\ref{defn_selmer_complex_general_ord}, let us set $R_\cyc:=R\,\widehat{\otimes}\LL(\Gamma_\cyc)$ and $X_\cyc:=X\,\widehat{\otimes}\LL(\Gamma_\cyc)^\sharp$. A choice of Greenberg local conditions $\Delta_X$ as in \cite[\S6.1]{nekovar06} give rise to local conditions $\Delta_{X_\cyc}$ for $X_\cyc$ and allows one to consider the Selmer complex
$$\widetilde{{\bf R}\Gamma}_{\rm f}(G_{K,\Sigma},X_\cyc;\Delta_{X_\cyc})\in D_{\rm ft} (_{R_\cyc}{\rm Mod})\,.$$
\item[ii)] Let ${\rm ht}_1(R)$ denote the set of height-$1$ primes of $R$. For any $P\in {\rm ht}_1(R)$ of $R$, let us write $P_\cyc\subset R_\cyc$ for the prime ideal of $R_\cyc$ generated by $P$ and $\gamma_\cyc-1$. 
\item[iii)] Suppose $M$ is a torsion $R_\cyc$-module of finite-type. Following \cite[Definition 11.6.6]{nekovar06}, we define
$$a_P(M_{P_\cyc}):={\rm length}_{R_P}\left(\varinjlim_r \frac{ (\gamma_\cyc-1)^{r-1}M_{P_\cyc}}{(\gamma_\cyc-1)^{r}M_{P_\cyc}} \right)\,.$$
\item[iv)] For each $P\in {\rm ht}_1(R)$, we let ${\rm Tam}_v(X,P)$ denote the Tamagawa factor at a prime $v\in \Sigma$ coprime to $p$, which is given as in \cite[Definition~7.6.10]{nekovar06}.
\item[v)] Let us set $X^*(1):={\rm Hom}(X,R)(1)$, which is a free $R$-module of finite rank endowed with a continuous action of $G_{K,\Sigma}$ in the obvious manner. We let $\Delta_X^*$ denote the dual local conditions for $X^*(1)$ (in the sense that the formalism in \cite[\S10.3.1]{nekovar06} applies). We then have an $R$-adic (cyclotomic) height pairing
$$\frak{h}_{X}^{\rm Nek}: \widetilde{H}_{\rm f}^1(G_{K,\Sigma},X;\Delta_X)\otimes \widetilde{H}_{\rm f}^1(G_{K,\Sigma},X^*(1);\Delta_X^*)\lra R$$
given as \S11.1.4 of op. cit., with $\Gamma=\Gamma_\cyc$. 
\item[vi)] If $R$ is a Krull domain, we define the $R$-adic regulator ${\rm Reg}_X$ on setting
$${\rm Reg}_X:=\Char_R\left({\rm coker}(\widetilde{H}_{\rm f}^1(G_{K,\Sigma},X;\Delta_X)\stackrel{{\rm adj}(\frak{h}_{X}^{\rm Nek})}{\lra}\widetilde{H}_{\rm f}^1(G_{K,\Sigma},X^*(1);\Delta_X^*))\right)$$
where ${\rm adj}$ denotes adjunction. Note that ${\rm Reg}_X$ is non-zero if and only if $\frak{h}_{X}^{\rm Nek}$ is non-degenerate.
\end{defn}
\begin{defn}
\label{def_connecting_a_P_to_leading_term}
\item[i)]
Suppose $R$ is a Krull domain so that $R_P$ is a discrete valuation ring for every $P\in {\rm ht}_1(R)$ and $(R_\cyc)_{P_\cyc}$ is a regular ring. Let $M$ be a finitely generated torsion $R_\cyc$-module. Let $r(M_{P_\cyc})$ denote the largest integer such that 
$$\Char_{(R_\cyc)_{P_\cyc}}(M_{P_\cyc})\in (\gamma_\cyc-1)^{r(M_{P_\cyc})}(R_\cyc)_{P_\cyc}\,.$$
We set 
$$\partial_\cyc\, \Char_{(R_\cyc)_{P_\cyc}}(M_{P_\cyc}):=\mathds{1}_\cyc\left((\gamma_\cyc-1)^{-r(M_{P_\cyc})}\Char_{(R_\cyc)_{P_\cyc}}(M_{P_\cyc})\right)\subset R_P\,$$
where $\mathds{1}_\cyc:R_\cyc\rightarrow R$ is the augmentation map.
\item[ii)] 
Suppose in addition that every primes in ${\rm ht}_1(R)$ is principal; this requirement is equivalent to asking that $R$ is a unique factorization domain. In that case, one may similarly define an integer $r(M)$ with 
$$\Char_{R_\cyc}(M)\in (\gamma_\cyc-1)^{r(M)}R_\cyc$$
and set
$$\partial_\cyc\,\Char_{R_\cyc}(M):=\mathds{1}_\cyc\left((\gamma_\cyc-1)^{-r(M)}\Char_{R_\cyc}(M)\right) \in R\,.$$
Observe in this particular scenario that we have 
\begin{equation}\label{eqn_char_Rcyc_localized_vs_char_Rcyc}
    \Char_{(R_\cyc)_{P_\cyc}}(M_{P_\cyc})=\left(\Char_{R_\cyc}(M)\right)_{P_\cyc}
\end{equation}
\begin{equation}\label{eqn_derived_char_Rcyc_localized_vs_char_Rcyc}
\partial_\cyc\, \Char_{(R_\cyc)_{P_\cyc}}(M_{P_\cyc})=\left(\partial_\cyc\,\Char_{R_\cyc}(M)\right)_P\,.    
\end{equation}
\end{defn}
\begin{lemma}[Nekov\'a\v{r}]
\label{lemma_nekovar_a_vs_partialcyc}
Suppose $R$ is a Krull domain. Then, 
$$a_P(M_{P_\cyc})=\ord_P\left(\partial_\cyc\, \Char_{(R_\cyc)_{P_\cyc}}(M_{P_\cyc})\right)\,.$$
If $R$ is in addition a unique factorization domain, then
$$a_P(M_{P_\cyc})=\left(\partial_\cyc\,\Char_{R_\cyc}(M)\right)_P\,.$$
\end{lemma}
\begin{proof}
The first part is \cite[Lemma 11.6.8]{nekovar06}. The second part follows from this combined with \eqref{eqn_derived_char_Rcyc_localized_vs_char_Rcyc}.
\end{proof}
\begin{proposition}[Nekov\'a\v{r}]
\label{prop_big_BSD}
With the notation of Definition~\ref{defn_general_derived_char_ideal}, suppose that the ring $R$ is a Krull domain and $\widetilde{H}^1_{\rm f}(G_{K,\Sigma},X_\cyc;\Delta_{X_\cyc})$ and $\widetilde{H}^2_{\rm f}(G_{K,\Sigma},X_\cyc;\Delta_{X_\cyc})$ are torsion as $R_\cyc$-modules. Assume also that  
$$\widetilde{H}^0_{\rm f}(G_{K,\Sigma},X;\Delta_{X})=\widetilde{H}^3_{\rm f}(G_{K,\Sigma},X;\Delta_{X})=0$$
$$\widetilde{H}^0_{\rm f}(G_{K,\Sigma},X_\cyc;\Delta_{X_\cyc})=\widetilde{H}^3_{\rm f}(G_{K,\Sigma},X_\cyc;\Delta_{X_\cyc})=0$$
and that the $R$-adic height pairing $\frak{h}^{\rm Nek}_X$ is non-degenerate $($in alternative wording, we assume that ${\rm Reg}_X$ is non-zero$)$. Let $P\in {\rm ht}_1(R)$ be any height-one prime such that ${\rm Tam}_v(X,P)=0$ for every $v\in \Sigma$ prime to $p$.  Then
\begin{align*}
    \ord_P\left(\partial_{\cyc}\Char_{(R_\cyc)_{P_\cyc}}\left(\widetilde{H}^2_{\rm f}(G_{K,\Sigma},X_\cyc;\Delta_{X_\cyc})_{P_{\cyc}}\right)\right)=& \,{\rm ord}_{P}\left({\rm Reg}(X)\right)\\
    & +{\rm length}_{R_P}\left( \left(\widetilde{H}^2_{\rm f}(G_{K,\Sigma},X;\Delta_{X})_{P}\right)_{R_P{\rm -tor}}\right)\,.
\end{align*}
If in addition $R$ is a Krull unique factorization domain, we have
\begin{align*}
    \ord_P\left(\partial_{\cyc}\Char_{R_\cyc}\left(\widetilde{H}^2_{\rm f}(G_{K,\Sigma},X_\cyc;\Delta_{X_\cyc})\right)\right)= \,{\rm ord}_{P}&\left({\rm Reg}(X)\right)\\
    & +\ord_P\left({\Char}_{R}\left(\widetilde{H}^2_{\rm f}(G_{K,\Sigma},X;\Delta_{X})_{R-{\rm tor}}\right)\right)\,.
\end{align*}
\end{proposition}
\begin{proof}
The first assertion is \cite[11.7.11]{nekovar06}, combined with Lemma~\ref{lemma_nekovar_a_vs_partialcyc}. The second statement follows from the first, using \eqref{eqn_char_Rcyc_localized_vs_char_Rcyc} and \eqref{eqn_derived_char_Rcyc_localized_vs_char_Rcyc}.
\end{proof}

\subsubsection{Descent to $\Gamma\in {\rm Gr}(\Gamma_K)$}
\label{subsubsec_descent_to_Gamma}
We now restrict our attention to the setting where $X=\TT_{f,\chi}^{(\Gamma)}$, $R=\LL_{\cO}(\Gamma)$ and $\TT_{\f,\chi}^{(\Gamma)}$, $R=\LL_\f(\Gamma)$, assuming that $\Gamma\in {\rm Gr}(\Gamma_K)$ is transversal to $\Gamma_\cyc$ in the sense explained above.

\begin{defn}
\label{defn_lif_gammacyc}
Suppose $\gamma_\cyc\in \Gamma_\cyc$ is a fixed topological generator as in \S\ref{subsubsec_Nek_descent}. We define its lift $\gamma_\cyc^{(\Gamma)}\in \Gamma_K$ parallel to $\Gamma$ as the unique element which maps to $\gamma_\cyc$ and ${\rm id}$ under the natural surjections $\Gamma_K\twoheadrightarrow \Gamma_\cyc$ and $\Gamma_K\twoheadrightarrow \Gamma$, respectively.
\end{defn}
The choice of the lifting $\gamma_\cyc^{(\Gamma)}\in \Gamma_K$ induces a split exact sequence
$$\xymatrix@R=.05cm{
1\ar[r]& \Gamma_\cyc \ar[r]& \Gamma_K \ar[r]& \Gamma\ar[r]& 1\\
&\gamma_\cyc\ar@{|->}[r]& \gamma_\cyc^{(\Gamma)}&&
}$$
of topological groups, as well as a commutative diagram $$\xymatrix@C=.4cm@R=.3cm
{\TT_{f,\chi}\ar[rr]^(.4){\sim}\ar[dr]_{\pi_\Gamma}&&\TT_{f,\chi}^{(\Gamma)}\,\widehat{\otimes}\,\LL(\Gamma_\cyc)^
\sharp\ar[dl]^{\frak{a}_{\cyc}}
\\
& \TT_{f,\chi}^{(\Gamma)}&}$$
(and a similar diagram when $\TT_{f,\chi}$ and $\TT_{f,\chi}^{(\Gamma)}$ are replaced by $\TT_{\f,\chi}$ and $\TT_{\f,\chi}^{(\Gamma)}$, respectively) where $\pi_\Gamma$ induced from the natural projection $\Gamma_K\twoheadrightarrow \Gamma$ and the map $\frak{a}_{\cyc}$ from the augmentation map $\gamma_\cyc\mapsto 1$. Through these identifications, one may apply the formalism in \S\ref{subsubsec_Nek_descent} to deduce the following consequence of Theorem~\ref{thm_ordinary_Deltai_MC_via_horizontal_ES}.

\begin{corollary}
\label{cor_main_conj_along_Gamma}
Suppose $f \in S_{k_f+2}(\Gamma_1(N_f),\epsilon_f)$ is a cuspidal eigen-newform and $\chi$ is a ray class character of $K$ verifying \ref{item_RIa} for some natural number $a$. Let $\f$ be the unique branch of the Hida family that admits $f$ as a specialization in weight $k_f+2$. Assume that \ref{item_tau} holds and $\overline{\rho}_{\f}$ is $p$-distinguished. Let us choose a morphism $u$ as in \S\ref{subsubsec_local_properties_CMgg} and $\Gamma\in {\rm Gr}(\Gamma_K)$.
\item[i)] The following the divisibilities in Conjecture~\ref{IMC_split_definite_indefinite_ord} hold true:
\begin{align*}
    \Char_{\LL_{\cO}(\Gamma)}\left(\widetilde{H}^2_{\rm f}(G_{K,\Sigma},\TT_{f,\chi}^{(\Gamma)};\Delta^{(1)})\right) \otimes_{\ZZ_p}\QQ_p&\,\Big{|}\,\,  \mathscr{C}(u)\cdot{\LL_{\cO}(\Gamma)} L_p(f_{/K}\otimes\chi,\Sigma^{(1)}_{\rm crit}){\big \vert}_\Gamma \otimes_{\ZZ_p}\QQ_p\, ,\\
    \Char_{\LL_{\cO}(\Gamma)}\left(\widetilde{H}^2_{\rm f}(G_{K,\Sigma},\TT_{f,\chi}^{(\Gamma)};\Delta^{(2)})\right){\otimes_{\cO}\Phi}&\,\Big{|}\,\, \mathscr{C}(u)\cdot H_\chi{\Lambda_\Phi(\Gamma)} L_p(f_{/K}\otimes\chi,\Sigma^{(2)}_{\rm crit})\vert_{\Gamma}\,.
\end{align*}
\item[ii)] Assume in addition that $\LL_\f$ is regular. The following the divisibilities in Conjecture~\ref{IMC_ord_split_families} hold true:
\begin{align*}
    \Char_{\LL_\f(\Gamma)}\left(\widetilde{H}^2_{\rm f}(G_{K,\Sigma},\TT_{\f,\chi}^{(\Gamma)};\Delta^{(1)})\right)\otimes_{\ZZ_p}\QQ_p&\,\Big{|}\,\,\mathscr{C}(u)\cdot H_{\f}{\LL_{\f}(\Gamma_K)}L_p(\f_{/K}\otimes\chi,{\mathbb \Sigma}^{(1)}_{\rm crit}){\big \vert}_\Gamma\otimes_{\ZZ_p}\QQ_p\, ,\\
\Char_{\LL_{\f}(\Gamma)}\left(\widetilde{H}^2_{\rm f}(G_{K,\Sigma},\TT_{\f,\chi}^{(\Gamma)};\Delta^{(2)})\right){\otimes_{\cO}\Phi}&\,\Big{|}\,\, \mathscr{C}(u)\cdot H_\chi{\Lambda_{\f,\Phi}(\Gamma)} L_p(\f_{/K}\otimes\chi,{\mathbb \Sigma}^{(2)}_{\rm crit}){\big \vert}_\Gamma\,.
\end{align*}
\end{corollary}

One may eliminate the correction terms $\mathscr{C}(u)$ from the statement of Corollary~\ref{cor_main_conj_along_Gamma} if the hypothesis \ref{item_Ind} concerning the CM branches of Hida families holds. In particular, if \ref{item_RIa} is valid with $a=0$, then we can take $\mathscr C(u)=(1)$ thanks to Proposition~\ref{prop_uniqueness_of_the_lattice_in_residually_irred_case}. 

\begin{proof}[Proof of Corollary~\ref{cor_main_conj_along_Gamma}]
\item[i)] Let us first assume that $\Gamma\in {\rm Gr}(\Gamma_K)$ is transversal to $\Gamma_\cyc$. 

We first begin with noting that 
$$\widetilde{H}^0_{\rm f}(G_{K,\Sigma},\TT_{\f,\chi}^{(\Gamma)};\Delta^{(i)})=\widetilde{H}^3_{\rm f}(G_{K,\Sigma},\TT_{\f,\chi}^{(\Gamma)};\Delta^{(i)})=0$$
$$\widetilde{H}^0_{\rm f}(G_{K,\Sigma},\TT_{\f,\chi};\Delta^{(i)})=\widetilde{H}^3_{\rm f}(G_{K,\Sigma},\TT_{\f,\chi};\Delta^{(i)})=0$$
because we have assumed \ref{item_tau}.

If $L_p(f_{/K}\otimes\chi,\Sigma^{(1)}_{\rm crit}){\big \vert}_\Gamma=0$, there is nothing to prove and we therefore assume that $L_p(f_{/K}\otimes\chi,\Sigma^{(1)}_{\rm crit}){\big \vert}_\Gamma\neq 0$. It follows from Theorem~\ref{thm_ordinary_Deltai_MC_via_horizontal_ES}(i) that 
$\gamma_\cyc^{(\Gamma)}-1$ does not divide $\Char_{\LL_\cO(\Gamma_K)}\left(\widetilde{H}^2_{\rm f}(G_{K,\Sigma},\TT_{f,\chi};\Delta^{(1)})\right)$, which in turn shows that the $\LL_{\cO}(\Gamma)$-module
$$\widetilde{H}^2_{\rm f}(G_{K,\Sigma},\TT_{f,\chi};\Delta^{(1)})/(\gamma_\cyc^{(\Gamma)}-1)\widetilde{H}^2_{\rm f}(G_{K,\Sigma},\TT_{f,\chi};\Delta^{(1)})\stackrel{\sim}{\lra}\widetilde{H}^2_{\rm f}(G_{K,\Sigma},\TT_{\f,\chi}^{(\Gamma)};\Delta^{(1)})$$
is torsion. By global Euler-Poincar\'e characteristic formulae, it follows that $\widetilde{H}^1_{\rm f}(G_{K,\Sigma},\TT_{\f,\chi}^{(\Gamma)};\Delta^{(1)})$ is $\LL_{\cO}(\Gamma)$-torsion as well. This discussion shows that 
\begin{equation}
\label{eqn_Nek_for_Gamma_1}
\widetilde{H}^2_{\rm f}(G_{K,\Sigma},\TT_{f,\chi}^{(\Gamma)};\Delta^{(1)})_{\LL_{\cO}(\Gamma)-{\rm tor}}=\widetilde{H}^2_{\rm f}(G_{K,\Sigma},\TT_{f,\chi}^{(\Gamma)};\Delta^{(1)})\,,
\end{equation}
\begin{align}
\notag \partial_{\cyc}\Char_{\LL_{\cO}(\Gamma_K)}\left(\widetilde{H}^2_{\rm f}(G_{K,\Sigma},\TT_{f,\chi};\Delta^{(1)})\right)&=\Char_{\LL_{\cO}(\Gamma_K)}\left(\widetilde{H}^2_{\rm f}(G_{K,\Sigma},\TT_{f,\chi};\Delta^{(1)})\right)\big{\vert}_{\Gamma}\\
\label{eqn_Nek_for_Gamma_2}
&\qquad \qquad\Big{\vert}\,\,\,  \mathscr{C}(u)\cdot {\LL_{\cO}(\Gamma_K)}L_p(f_{/K}\otimes\chi,\Sigma^{(1)}_{\rm crit}){\big \vert}_\Gamma\,,
\end{align}
\begin{equation}
\label{eqn_Nek_for_Gamma_3}
{\rm Reg}(\TT_{f,\chi}^{\Gamma})=\LL_{\cO}(\Gamma)\,.
\end{equation}
Moreover, it follows from \cite[Corollary 8.9.7.4]{nekovar06} applied with $T=\TT_f$ that ${\rm Tam}_v(X,P)=0$ for every height-one prime of $\LL_{\cO}(\Gamma)$. Combining this with \eqref{eqn_Nek_for_Gamma_1}, \eqref{eqn_Nek_for_Gamma_2}, \eqref{eqn_Nek_for_Gamma_3} and Proposition~\ref{prop_big_BSD} for every height one prime $P$ of $\LL_{\cO}(\Gamma)$ which does not contain $p$, we deduce that
$$\Char_{\LL_{\cO}(\Gamma)}\left(\widetilde{H}^2_{\rm f}(G_{K,\Sigma},\TT_{f,\chi}^{(\Gamma)};\Delta^{(1)})\right) \,\Big{|}\,\,  \mathscr{C}(u)\cdot{\LL_{\cO}(\Gamma)} L_p(f_{/K}\otimes\chi,\Sigma^{(1)}_{\rm crit}){\big \vert}_\Gamma \otimes_{\ZZ_p}\QQ_p\,,$$
as required.

The setting that concerns the local conditions $\Delta^{(2)}$ may be handled in an identical manner.

In the remaining case when $\Gamma=\Gamma_\cyc$, we may simply replace the role of $\Gamma_\cyc$ in \S\ref{subsubsec_Nek_descent} by any other $\Gamma^\prime\in {\rm Gr}(\Gamma_K)$ and argue as above.
\item[ii)] The argument in the first part applies verbatim.
\end{proof}


\subsection{Anticyclotomic Iwasawa Theory}
\label{subsec_split_indefinite_definite_ord}
We shall now concentrate in the  particular case when $\Gamma=\Gamma_\ac$ and refine Corollary~\ref{cor_main_conj_along_Gamma}. We assume throughout \S\ref{subsec_split_indefinite_definite_ord} that $\chi^c=\chi^{-1}$ (i.e., the character $\chi$ is anticyclotomic or a ring class character) and we work with the conjugate self-dual twists $\TT_{f,\chi}^\dagger=\TT_{f,\chi}^\ac\,\widehat{\otimes}\,\LL(\Gamma_\cyc)^\sharp$ and $\TT_{\f,\chi}^\dagger=\TT_{\f,\chi}^\ac\,\widehat{\otimes}\,\LL(\Gamma_\cyc)^\sharp$ of $\TT_{f,\chi}$ and $\TT_{\f,\chi}$, respectively. Let us put $N_f=N_f^+N_f^-$ so that $N_f^+$ (resp., $N_f^-$) is only divisible by primes that are split (resp., inert) in $K/\QQ$.

The following are the main results of \cite{chidahsiehcrelle, hsiehnonvanishing}.

\begin{theorem}[Chida--Hsieh, Hsieh]
\label{thm_Chida_Hsieh} Suppose $N_f^-$ is square-free.
\item[i)] If $N_f^-$ is the product of an odd number of primes, then  $L_p(f_{/K}\otimes\chi,\Sigma^{(1)}_{\rm cc})\big{\vert}_{\Gamma_\ac}$ is non-trivial. 
\item[ii)] If $N_f^-$ is a product of even number of primes, then $L_p(f_{/K}\otimes\chi, \Sigma^{(2)}_{\rm cc}){\big \vert}_{\Gamma_\ac}$ is non-trivial.
\end{theorem}

\begin{theorem}
\label{thm_ordinary_Delta1_Delta2_anticyclo_Selmer_torsion_via_horizontal_ES} 
Suppose $f \in S_{k_f+2}(\Gamma_0(N_f))$ is a cuspidal eigen-newform and $\chi$ is a ring class character of $K$. We assume that $N_f^-$ is square-free. Let $\f$ be the unique branch of the Hida family that admits $f$ as a specialization in weight $k_f+2$. Suppose that $\overline{\rho}_{\f}$ is  $p$-distinguished.
\item[i)] Assume that one of the following holds:
\begin{itemize}
    \item[(A)] $\chi^2\neq \mathds{1}$ and the hypothesis \ref{item_tau} holds true.
    \item[(B)] { \ref{item_tau}, \ref{item_HSS}, \ref{item_HnEZ}} and {\ref{item_HDist}} are valid.
    \item[(C)] $\chi=\mathds{1}$, $k_f\equiv 0$ (mod $p-1$) and the hypotheses \ref{item_SU1}--\ref{item_SU3} are valid.
\end{itemize}  
If $N_f^-$ is the product of an odd (resp., even) number of primes, then both $\LL_{\cO}(\Gamma_\ac)$-modules $\widetilde{H}^1_{\rm f}(G_{K,\Sigma},\TT_{f,\chi}^\ac;\Delta^{(1)})$ and $\widetilde{H}^2_{\rm f}(G_{K,\Sigma},\TT_{f,\chi}^\ac;\Delta^{(1)})$ (resp., $\widetilde{H}^1_{\rm f}(G_{K,\Sigma},\TT_{f,\chi}^\ac;\Delta^{(2)})$ and $\widetilde{H}^2_{\rm f}(G_{K,\Sigma},\TT_{f,\chi}^\ac;\Delta^{(2)})$) are torsion.
\item[ii)] Assume either (A) or (B) above, or else one of the following conditions holds:
\begin{itemize}
    \item[(C')] $\chi=\mathds{1}$ and the hypotheses \ref{item_SU1}--\ref{item_SU3} are valid.
\end{itemize} 
If $N_f^-$ is the product of an odd (resp., even) number of primes, then both $\LL_{\f}(\Gamma_\ac)$-modules $\widetilde{H}^1_{\rm f}(G_{K,\Sigma},\TT_{\f,\chi}^\ac;\Delta^{(1)})$ and $\widetilde{H}^2_{\rm f}(G_{K,\Sigma},\TT_{\f,\chi}^\ac;\Delta^{(1)})$ (resp., $\widetilde{H}^1_{\rm f}(G_{K,\Sigma},\TT_{\f,\chi}^\ac;\Delta^{(2)})$ and $\widetilde{H}^2_{\rm f}(G_{K,\Sigma},\TT_{\f,\chi}^\ac;\Delta^{(2)})$) are torsion.
\end{theorem}
\begin{proof}
\item[i)] In the situation of (B), the assertions were already proved in \cite{BLForum}. When (C) holds, the required result is a direct consequence of \cite[Theorem 3.36]{skinnerurbanmainconj} combined with control theorems for Selmer complexes~\cite[Corollary 8.10.2]{nekovar06} and Theorem~\ref{thm_Chida_Hsieh}(i).

Thanks to global Euler--Poincar\'e characteristic formulae, we only need to prove that the $\LL_{\cO}(\Gamma_\ac)$-module $\widetilde{H}^2_{\rm f}(G_{K,\Sigma},\TT_{f,\chi}^\ac;\Delta^{(1)})$ (resp., $\widetilde{H}^2_{\rm f}(G_{K,\Sigma},\TT_{f,\chi}^\ac;\Delta^{(2)})$) is torsion if $N_f^-$ is the product of an odd (resp., even) number of primes. By the control theorem for Selmer complexes, we are reduced to checking that the module
$\widetilde{H}^2_{\rm f}(G_{K,\Sigma},T_{f,\chi\psi}^\dagger;\Delta^{(1)})$ (resp., $\widetilde{H}^2_{\rm f}(G_{K,\Sigma},T_{f,\chi\psi}^\dagger;\Delta^{(2)})$) has finite cardinality for some Hecke character $\mathds{1}\neq \psi \in \Sigma^{(1)}_{\rm cc}$ (resp., $\psi \in \Sigma^{(2)}_{\rm cc}$) whose associated $p$-adic Galois character $\psi_p$ factors through $\Gamma_\ac$.

In the situation of (A), the existence of such $\psi$ follows as a consequence of Theorem~\ref{thm_BK_type_result} combined with Theorem~\ref{thm_Chida_Hsieh}(i) when $N_f^-$ is the product of an odd number of primes (resp.,  Theorem~\ref{thm_Chida_Hsieh}(ii) $N_f^-$ is the product of an even number of primes). 

\item[ii)] The proof of this portion reduces to (i) thanks to the control theorem for Selmer complexes~\cite[Proposition 8.10.1]{nekovar06}.
\end{proof}
\begin{corollary}
\label{cor_thm_ordinary_Delta1_Delta2_anticyclo_Selmer_torsion_via_horizontal_ES}
Suppose $f \in S_{k_f+2}(\Gamma_0(N_f))$ is a cuspidal eigen-newform and $\chi$ is a ring class character of $K$. We assume that $N_f^-$ is square-free. Let $\f$ be the unique branch of the Hida family that admits $f$ as a specialization in weight $k_f+2$. Suppose that $\overline{\rho}_{\f}$ is $p$-distinguished.
\item[i)]  Assume that one of the following holds:
\begin{itemize}
    \item[(A)] $\chi^2\neq \mathds{1}$ and the hypothesis \ref{item_tau} holds true.
    \item[(B)] {{ \ref{item_tau}, \ref{item_HSS}, \ref{item_HnEZ}} and \rm{\ref{item_HDist}}} are valid.
    \item[(C)] $\chi=\mathds{1}$, $k_f\equiv 0$ (mod $p-1$) and the hypotheses \ref{item_SU1}--\ref{item_SU3} are valid.
\end{itemize}   
Then, ${\rm rank}_{\LL_{\cO}(\Gamma_\ac)}\left(\widetilde{H}^1_{\rm f}(G_{K,\Sigma},\TT_{f,\chi}^\ac;\Delta^{(1)})\right)\leq 1\,.$
\item[ii)] Assume either (A) or (B) above, or else one of the following conditions holds:
\begin{itemize}
     \item[(C')] $\chi=\mathds{1}$ and the hypotheses \ref{item_SU1}--\ref{item_SU3} are valid.
\end{itemize}  
Then ${\rm rank}_{\LL_{\f}(\Gamma_\ac)}\left(\widetilde{H}^1_{\rm f}(G_{K,\Sigma},\TT_{\f,\chi}^\ac;\Delta^{(1)})\right)\leq 1\,.$
\end{corollary}

\begin{proof}
\item[i)] In the setting of (B), this is \cite[Theorem 3.5]{BLForum}. In the situation of (A) when $N_f^-$ has odd number of prime factors and in the setting of (C), the required bound on the rank follows as an immediate consequence of Theorem~\ref{thm_ordinary_Delta1_Delta2_anticyclo_Selmer_torsion_via_horizontal_ES}(i). 

In the situation of (A) when $N_f^-$ has even number of prime factors,
it follows from Theorem~\ref{thm_ordinary_Delta1_Delta2_anticyclo_Selmer_torsion_via_horizontal_ES}(i) that $\widetilde{H}^1_{\rm f}(G_{K,\Sigma},\TT_{f,\chi}^\ac;\Delta^{(2)})$ is torsion. Moreover, the exactness of the sequence
$$0\lra \widetilde{H}^1_{\rm f}(G_{K,\Sigma},\TT_{f,\chi}^\ac;\Delta^{(2)})\lra H^1_{\calF_+}(K,\TT_{f,\chi}^\ac) \lra H^1(K_{\p^c},F^+\TT_{f,\chi}^\ac)$$ 
(where $H^1_{\calF_+}(K,\TT_{f,\chi}^\ac)$ is the Selmer group given as in \cite[Definition 3.1]{BLForum}) together with the fact that the $\LL_{\cO}(\Gamma_\ac)$-module $ H^1(K_{\p^c},F^+\TT_{f,\chi}^\ac)$ is of rank one, shows that $\LL_{\cO}(\Gamma_\ac)$-module $H^1_{\calF_+}(K,\TT_{f,\chi}^\ac)$ has rank at most one. Since $H^1_{\calF_+}(K,\TT_{f,\chi}^\ac)$ contains $\widetilde{H}^1_{\rm f}(G_{K,\Sigma},\TT_{f,\chi}^\ac;\Delta^{(1)})$, the asserted bound on the rank of the latter follows.
\item[ii)] The proof of this portion reduces to (i) thanks to the control theorem for Selmer complexes~\cite[Proposition 8.10.1]{nekovar06}.
\end{proof}
\begin{defn}
\label{defn_lift_gammacyc_ac}
Suppose $\gamma_\cyc\in \Gamma_\cyc$ is a fixed topological generator as in \S\ref{subsubsec_Nek_descent}. We define its lift (which we will denote by the same symbol) $\gamma_\cyc\in \Gamma_K$ parallel to $\Gamma_\ac$ as the unique element which maps to $\gamma_\cyc$ and ${\rm id}$ under the natural surjections $\Gamma_K\twoheadrightarrow \Gamma_\cyc$ and $\Gamma_K\twoheadrightarrow \Gamma_\ac$, respectively.
\end{defn}

\begin{defn}
\label{def_mock_leading_term}
Suppose $R$ is a Krull unique factorization domain and $J\subset R_\cyc$ is an ideal. We set
$$\partial_\cyc^* J =\begin{cases}
\mathds{1}_\cyc(J) & \hbox{ if } (\gamma_\cyc-1)\nmid J\\
\mathds{1}_\cyc((\gamma_\cyc-1)^{-1}J) & \hbox{ if } (\gamma_\cyc-1)\mid J\,.
\end{cases}$$
and call it the mock leading term for $J$. In particular, for a torsion $R_\cyc$-module $M$ we have
$$\partial_\cyc^*\Char_{R_\cyc}(M)=\begin{cases}
\partial_\cyc\Char_{R_\cyc}(M) & \hbox{ if } r(M)\leq 1\\
0 & \hbox{ if } r(M)>1\,.
\end{cases}$$
\end{defn}
Theorem~\ref{thm_big_BSD_step1} below presents algebraic formulations of $\LL_{\cO}(\Gamma_\ac)$-adic and  $\LL_\f(\Gamma_\ac)$-adic BSD formulae. Their analytic counterparts (which involves suitable $p$-adic $L$-functions) have been established in certain sub-cases; see {\cite[Theorem 1.1]{BLForum}.} 
\begin{theorem}[$\LL(\Gamma_\ac)$-adic and $\LL_\f\,\widehat{\otimes}\LL(\Gamma_\ac)$-adic BSD-type formula]
\label{thm_big_BSD_step1} Suppose $f \in S_{k_f+2}(\Gamma_0(N_f))$ is a cuspidal eigen-newform and $\chi$ is a ring class character of $K$. We assume that $N_f^-$ is square-free. Let $\f$ be the unique branch of the Hida family that admits $f^\alpha$ as a specialization in weight $k_f+2$. Suppose that $\overline{\rho}_{\f}$ is $p$-distinguished. 
\item[i)]  Assume that one of the following holds:
\begin{itemize}
    \item[(A)] $\chi^2\neq \mathds{1}$ and the hypothesis \ref{item_tau} holds true.
    \item[(B)] {{\ref{item_tau}, \ref{item_HSS}, \ref{item_HnEZ}} and \rm{\ref{item_HDist}}} are valid.
    \item[(C)] $\chi=\mathds{1}$, $k_f\equiv 0$ (mod $p-1$) and the hypotheses \ref{item_SU1}--\ref{item_SU3} are valid.
\end{itemize}   
Then,
\begin{align*}\partial_{\cyc}^{*}\Char_{\LL_\cO(\Gamma_K)}\left(\widetilde{H}^2_{\rm f}(G_{K,\Sigma},\TT_{f,\chi}^\dagger;\Delta^{(1)})\right)={\rm Reg}_{\TT_{f,\chi}^\ac}\cdot \Char_{\LL_\cO(\Gamma_\ac)}\left(\widetilde{H}^2_{\rm f}(G_{K,\Sigma},\TT_{f,\chi}^\ac;\Delta^{(1)})_{{\rm tor}}\right)\,.
    \end{align*}
\item[ii)] Assume either (A) or (B) above, or else one of the following conditions holds:
\begin{itemize}
    \item[(C')] $\chi=\mathds{1}$ and the hypotheses \ref{item_SU1}--\ref{item_SU3} are valid.
\end{itemize}  
Assume in addition that $\LL_\f$ is regular. Then,
\begin{align*}\partial_{\cyc}^{*}\Char_{\LL_\f(\Gamma_K)}\left(\widetilde{H}^2_{\rm f}(G_{K,\Sigma},\TT_{\f,\chi}^\dagger;\Delta^{(1)})\right)={\rm Reg}_{\TT_{\f,\chi}^\ac}\cdot \Char_{\LL_\f(\Gamma_\ac)}\left( \widetilde{H}^2_{\rm f}(G_{K,\Sigma},\TT_{\f,\chi}^\ac;\Delta^{(1)})_{{\rm tor}}\right)\,.
    \end{align*}
\end{theorem}

\begin{proof}
We only explain the proof of (ii) since the proof of (i) is very similar.

The asserted equality follows from Proposition~\ref{prop_big_BSD} once we verify that the following properties hold true.
\begin{enumerate}
    \item ${\rm Tam}_v(\TT_{\f,\chi}^{\ac},P)=0$ for every height-one prime $P$ of $\LL_\f(\Gamma_\ac)$.
    \item The $\LL_\f(\Gamma_K)$-modules $\widetilde{H}^1_{\rm f}(G_{K,\Sigma},\TT_{\f,\chi}^\dagger;\Delta^{(1)})$ and $\widetilde{H}^2_{\rm f}(G_{K,\Sigma},\TT_{\f,\chi}^\dagger;\Delta^{(1)})$ are torsion.
    \item $r(\widetilde{H}^2_{\rm f}(G_{K,\Sigma},\TT_{\f,\chi}^\dagger;\Delta^{(1)}))>1$ if and only if ${\rm Reg}_{\TT_{\f,\chi}^\ac}=0$. 
\end{enumerate}
Property (1) follows from \cite[Corollary 8.9.7.4]{nekovar06} applied with $T=\TT_{\f,\chi}^\dagger$ and $\Gamma=\Gamma_\ac$. Property (2) is 
\begin{itemize}
    \item the twisted variant \eqref{eqn_twisted_torsion_families_Gamma_K_2} of Theorem~\ref{thm_ordinary_Delta1_Selmer_torsion_via_horizontal_ES}(ii) when (A) holds;
    \item a consequence of \cite[Theorem 3.19]{BLForum} (via control theorems for Selmer complexes) in the situation of (B);
    \item a consequence of \cite[Theorem 17.4(i)]{kato04} (via control theorems for Selmer complexes) in the situation of (C)  (without requiring \ref{item_SU1}--\ref{item_SU3}.
\end{itemize} 
We now explain how to use \cite[Proposition 11.7.6(vii)]{nekovar06} to prove that $r(\widetilde{H}^2_{\rm f}(G_{K,\Sigma},\TT_{\f,\chi}^\ac;\Delta^{(1)})\leq 1$ if ${\rm Reg}_{\TT_{\f,\chi}^\ac}\neq 0$ (which is the content of the property (3)). 

Since $\LL_\f(\Gamma_\ac)$ is an integral domain, it follows that the set denoted by $Q$ in op. cit. consists only of the zero ideal $(0)$ of $\LL_\f(\Gamma_\ac)$. In this situation, the ideal $(\gamma_\cyc-1) \in \LL_\f(\Gamma_K)$ (which makes sense thanks to Definition~\ref{defn_lift_gammacyc_ac}) corresponds to $\overline{\frak{q}}$ in the notation of Nekov\'a\v{r}. 

It follows from  \cite[Proposition 11.7.6(vii)]{nekovar06} that 
$${\rm length}_{\LL_\f(\Gamma_K)_{(\gamma_\cyc-1)}}\left(\widetilde{H}^2_{\rm f}(G_{K,\Sigma},\TT_{\f,\chi}^\dagger;\Delta^{(1)})_{(\gamma_\cyc-1)}\right)\geq {\rm length}_{\LL_\f(\Gamma_\ac)_{(0)}}\left(\widetilde{H}^1_{\rm f}(G_{K,\Sigma},\TT_{\f,\chi}^\ac;\Delta^{(1)})_{(0)}\right)$$
with equality if and only if $({\rm Reg}_{\TT_{\f,\chi}^\ac})_{(0)}\neq 0$ (which is equivalent to asking ${\rm Reg}_{\TT_{\f,\chi}^\ac}\neq 0$, since $\LL_\f(\Gamma_\ac)$ is an integral domain). 
Thence,
\begin{align*}
    r(\widetilde{H}^2_{\rm f}(G_{K,\Sigma},\TT_{\f,\chi}^\dagger;\Delta^{(1)}))&={\rm length}_{\LL_\f(\Gamma_K)_{(\gamma_\cyc-1)}}\left(\widetilde{H}^2_{\rm f}(G_{K,\Sigma},\TT_{\f,\chi}^\dagger;\Delta^{(1)})_{(\gamma_\cyc-1)}\right)\\
    &\geq {\rm length}_{\LL_\f(\Gamma_\ac)_{(0)}}\left(\widetilde{H}^1_{\rm f}(G_{K,\Sigma},\TT_{\f,\chi}^\ac;\Delta^{(1)})_{(0)}\right)\\
    &={\rm rank}_{\LL_\f(\Gamma_\ac)}\left(\widetilde{H}^1_{\rm f}(G_{K,\Sigma},\TT_{\f,\chi}^\ac;\Delta^{(1)})\right)
\end{align*}
with equality if and only if $({\rm Reg}_{\TT_{\f,\chi}^\ac})\neq 0$.  Since ${\rm rank}_{\LL_\f(\Gamma_\ac)}\left(\widetilde{H}^1_{\rm f}(G_{K,\Sigma},\TT_{\f,\chi}^\ac;\Delta^{(1)}\right)\leq 1$, it follows from Corollary~\ref{cor_thm_ordinary_Delta1_Delta2_anticyclo_Selmer_torsion_via_horizontal_ES}(ii) that this inequality is strict if $r(\widetilde{H}^2_{\rm f}(G_{K,\Sigma},\TT_{\f,\chi}^\dagger;\Delta^{(1)}))>1$, and in turn also that ${\rm Reg}_{\TT_{\f,\chi}^\ac}=0$ in that case. Conversely, if  ${\rm Reg}_{\TT_{\f,\chi}^\ac}=0$,  then the inequality above is strict and we have 
$$r(\widetilde{H}^2_{\rm f}(G_{K,\Sigma},\TT_{\f,\chi}^\dagger;\Delta^{(1)}))\geq 1+ {\rm rank}_{\LL_\f(\Gamma_\ac)}\left(\widetilde{H}^1_{\rm f}(G_{K,\Sigma},\TT_{\f,\chi}^\ac;\Delta^{(1)}\right)\,.$$
To complete the verification of (3), it therefore remains to prove that ${\rm rank}_{\LL_\f(\Gamma_\ac)}\left(\widetilde{H}^1_{\rm f}(G_{K,\Sigma},\TT_{\f,\chi}^\ac;\Delta^{(1)}\right)\geq1$. This can be easily proved using the formula of Bertolini--Darmon--Prasanna in families, c.f. \cite{bertolinidarmonprasanna13,CastellapadicvariationofHeegnerpoints}.

This concludes the proof of Theorem~\ref{thm_big_BSD_step1}(ii).
\end{proof}

\begin{defn}
\label{def_mock_leading_term_2}
Given an element $\mathbb{A}=\sum_{0\geq } A_n\cdot(\gamma_\cyc-1)^n\in \LL_\f(\Gamma_\ac)\,\widehat{\otimes}\LL(\Gamma_\cyc)$ $($so that we have $A_n\in \LL_\f(\Gamma_\ac)$$)$, we set 
$$\partial_\cyc^* \mathbb{A}:=\begin{cases}
A_0\,,& \hbox{ if } \mathds{1}_{\cyc}(\mathbb{A})\neq 0,\\
A_1\,,& \hbox{ if } \mathds{1}_{\cyc}(\mathbb{A})=0
\end{cases}$$
and call it the mock leading term of the power series $\mathbb{A}$.
\end{defn}
\begin{corollary}
\label{cor_big_BSD_step2} 
 Suppose $f \in S_{k_f+2}(\Gamma_0(N_f))$ is a cuspidal eigen-newform and $\chi$ is a ring class character of $K$. We assume that $N_f^-$ is square-free. Let $\f$ be the unique branch of the Hida family that admits $f$ as a specialization in weight $k_f+2$. Suppose that $\overline{\rho}_{\f}$ is $p$-distinguished. Let us fix a morphism $u$ as in \S\ref{subsubsec_local_properties_CMgg}.
\item[i)] Suppose Conjecture~\ref{IMC_split_definite_indefinite_ord_twisted}{\rm (i)} holds true\footnote{This is equivalent, thanks to \cite[Theorem 3.19(i)]{BLForum}, to the requirement that Conjecture~\ref{IMC_split_definite_indefinite_ord_twisted}{\rm (ii)} holds true, granted the validity of \ref{item_Ind}.} with $\Gamma=\Gamma_K$.
In the situation of Theorem~\ref{thm_big_BSD_step1}(i) we have
$$\LL_{L}(\Gamma_\ac)\cdot \partial_\cyc^* \left({\rm Tw}\left(L_p(f_{/K}\otimes\chi,  {\Sigma}^{(1)})\right)\right)={\rm Reg}_{\TT_{f,\chi}^\ac}\cdot \Char_{\LL_{\cO}(\Gamma_\ac)}\left( \widetilde{H}^2_{\rm f}(G_{K,\Sigma},\TT_{f,\chi}^\ac;\Delta^{(1)})_{{\rm tor}}\right)\otimes\QQ_p\,.$$
Moreover we have (resp., the opposite containment)
$$\LL_{L}(\Gamma_\ac)\cdot \mathscr{C}(u)\cdot \partial_\cyc^* \left({\rm Tw}\left(L_p(f_{/K}\otimes\chi,  {\Sigma}^{(1)})\right)\right)\subset {\rm Reg}_{\TT_{f,\chi}^\ac}\cdot \Char_{\LL_{\cO}(\Gamma_\ac)}\left(\widetilde{H}^2_{\rm f}(G_{K,\Sigma},\TT_{f,\chi}^\ac;\Delta^{(1)})_{{\rm tor}}\right) \otimes\QQ_p$$
if the containment (resp., the opposite containment)
$$\LL_{L}(\Gamma_\ac)\cdot \mathscr{C}(u)\cdot {\rm Tw}\left(L_p(f_{/K}\otimes\chi,  {\Sigma}^{(1)})\right)\subset \Char_{\LL_{\cO}(\Gamma_K)}\left(\widetilde{H}_{\rm f}^2(G_{K,\Sigma},\TT_{f,\chi}^\dagger;\Delta^{(1)}) \right)\otimes_{\Zp}\Qp$$
in Conjecture~\ref{IMC_split_definite_indefinite_ord}{\rm (i)} holds  true.
\item[ii)] 
Suppose Conjecture~\ref{IMC_ord_split_families}{\rm (i)} holds true
with $\Gamma=\Gamma_K$. In the situation of Theorem~\ref{thm_big_BSD_step1}(ii) we have
$$\left(\LL_\f(\Gamma_\ac)\otimes_{\ZZ_p}\QQ_p\right) \partial_\cyc^* \left(H_\f \ {\rm Tw}_{\f}\left(L_p(\f_{/K}\otimes\chi,  \mathbb{\Sigma}^{(1)})\right)\right)={\rm Reg}_{\TT_{\f,\chi}^\ac} \Char_{\LL_\f(\Gamma_\ac)}\left( \widetilde{H}^2_{\rm f}(G_{K,\Sigma},\TT_{\f,\chi}^\ac;\Delta^{(1)})_{{\rm tor}}\right)\otimes\QQ_p.$$
Moreover we have (resp., the opposite containment)
$$\left(\LL_\f(\Gamma_\ac)\otimes_{\ZZ_p}\QQ_p\right)\,\mathscr{C}(u)\, \partial_\cyc^* \left(H_\f {\rm Tw}_{\f}\left( L_p(\f_{/K}\otimes\chi,  \mathbb{\Sigma}^{(1)})\right)\right)\subset {\rm Reg}_{\TT_{\f,\chi}^\ac} \Char_{\LL_\f(\Gamma_\ac)}\left( \widetilde{H}^2_{\rm f}(G_{K,\Sigma},\TT_{\f,\chi}^\ac;\Delta^{(1)})_{{\rm tor}}\right)\otimes\QQ_p$$
if the containment (resp., the opposite containment)
$$\left(\LL_\f(\Gamma_K)\otimes_{\ZZ_p}\QQ_p\right)\, \mathscr{C}(u)\, H_\f {\rm Tw}_{\f}\left(L_p(\f_{/K}\otimes\chi,  \mathbb{\Sigma}^{(1)})\right)\subset \Char_{\LL_\f(\Gamma_K)}\left(\widetilde{H}_{\rm f}^2(G_{K,\Sigma},\TT_{\f,\chi}^\dagger;\Delta^{(1)}) \right)\otimes_{\Zp}\Qp$$
in Conjecture~\ref{IMC_ord_split_families}{\rm (i)} holds  true.
\end{corollary}

One may drop the correction terms $\mathscr{C}(u)$ from the statement of Corollary~\ref{cor_big_BSD_step2} if the hypothesis \ref{item_Ind} concerning the CM branches of Hida families holds. In particular, if \ref{item_RIa} is valid with $a=0$ (which translates in the setting of Corollary~\ref{cor_big_BSD_step2} to the requirement that $\chi^2\not\equiv \mathds 1 \mod \m_\cO$), then we can take $\mathscr C(u)=(1)$ thanks to Proposition~\ref{prop_uniqueness_of_the_lattice_in_residually_irred_case}. 


\subsubsection{Anticyclotomic Iwasawa Theory in the definite case}
\label{subsubsec_split_definite_ord}
Throughout in \S\ref{subsubsec_split_definite_ord}, we assume that $N_f^-$ is a square-free product of an odd number of primes. We will explain how to obtain unconditional versions of Corollary~\ref{cor_big_BSD_step2} in this set up.  

 Let ${\rm ver}_{\rm ac}: \Gamma_\ac\to\Gamma_K$ be the Verschiebung map, given by $\gamma\mapsto \widehat{\gamma}^{c-1}$, where $\widehat{\gamma}\in \Gamma_K$ is any lift of $\gamma\in \Gamma_\ac$ and $c$ is the generator of $\Gal(K/\QQ)$.

\begin{defn}
\label{defn_C_ac_u}
Given a morphism $u$ as in \S\ref{subsubsec_local_properties_CMgg}, we define the ideal $\mathscr{C}^{\ac}(u)\subset \LL_{\cO}(\Gamma)$ on setting
$$\mathscr{C}^{\ac}(u):=\ker\left(\LL_{\cO}(\Gamma_\ac)\xrightarrow{{\rm ver}_{\rm ac}}\LL_{\cO}(\Gamma_K)\twoheadrightarrow  \LL_{\cO}(\Gamma_\p)\twoheadrightarrow \LL_{\cO}(\Gamma_\p)/\mathscr{C}(u)\right)\,.$$
\end{defn}
Note that $\mathscr{C}^{\ac}(u)=(1)$ if $\mathscr{C}(u)=(1)$.
\begin{theorem}
\label{thm_ordinary_definite_ac} 
Suppose $f \in S_{k_f+2}(\Gamma_0(N_f))$ is a cuspidal eigen-newform and $\chi$ is a ring class character of $K$. Let $\f$ be the unique branch of the Hida family that admits $f^\alpha$ as a specialization in weight $k_f+2$. Suppose that $\overline{\rho}_{\f}$ is $p$-distinguished. 
\item[i)] Assume at least one of the following holds:
\begin{itemize}
    \item[(A)] $\chi^2\neq \mathds{1}$ and the hypothesis \ref{item_tau} holds true.
    \item[(B)] {\ref{item_tau}, \ref{item_HSS}, \ref{item_HnEZ}} and \rm{\ref{item_HDist}} are valid.
    \item[(C)] $\chi=\mathds{1}$, $k_f\equiv 0$ (mod $p-1$) and the hypotheses  \ref{item_SU2}--\ref{item_SU3} are valid.
\end{itemize} 
Then,
$$\Char_{\LL_{\cO}(\Gamma_\ac)}\left( \widetilde{H}^2_{\rm f}(G_{K,\Sigma},\TT_{f,\chi}^\ac;\Delta^{(1)})\right){\otimes_\cO\Phi}\\,\, \Big{|}\,\,\, \LL_{{\Phi}}(\Gamma_\ac)\cdot \mathscr{C}^{\ac}(u)\, L_p(f_{/K}\otimes\chi,  {\Sigma}^{(1)}_{\rm cc})\big{\vert}_{\Gamma_\ac}\,.$$
One may take $\mathscr{C}^{\ac}(u)=(1)$ in the situation of (B) or (C). Moreover, the asserted divisibility is an equality in the latter case.
\item[ii)] Assume that either (A), (B) or (C') below holds:
\begin{itemize}
    \item[(C')] $\chi=\mathds{1}$ and the hypothesis \ref{item_SU2}--\ref{item_SU3} are valid.
\end{itemize} 
Then,
$$\Char_{\LL_{\f}(\Gamma_\ac)}\left( \widetilde{H}^2_{\rm f}(G_{K,\Sigma},\TT_{\f,\chi}^\ac;\Delta^{(1)})\right)\otimes\QQ_p\,\, \Big{|}\,\,\, \LL_{L}(\Gamma_\ac)\cdot\mathscr{C}^{\ac}(u)\, H_{\f}L_p(\f_{/K}\otimes\chi,  \mathbb{\Sigma}^{(1)}_{\rm cc})\big{\vert}_{\Gamma_\ac}\,.$$
One may take $\mathscr{C}^{\ac}(u)=(1)$ in the situation of (B) or (C'). Moreover, the asserted divisibility is an equality in the latter case.
\end{theorem}

One may drop the correction terms $\mathscr{C}(u)$ from the statement of Theorem~\ref{thm_ordinary_definite_ac} if the hypothesis \ref{item_Ind} concerning the CM branches of Hida families holds. In particular, if \ref{item_RIa} is valid with $a=0$ (which translates in the setting of Theorem~\ref{thm_ordinary_definite_ac} to the requirement that $\chi^2\not\equiv \mathds 1 \mod \m_\cO$), then we can take $\mathscr C(u)=(1)$ thanks to Proposition~\ref{prop_uniqueness_of_the_lattice_in_residually_irred_case}. 

\begin{proof}[Proof of Theorem~\ref{thm_ordinary_definite_ac}] Note in this scenario we have ${\rm Reg}_{\TT_{\f,\chi}^\ac}=\LL_{\cO}(\Gamma_\ac)$ (since both Selmer groups involved in the definition of this regulator are trivial) and we have 
$$\partial_\cyc^* \left({\rm Tw}\left(L_p(f_{/K}\otimes\chi,  {\Sigma}^{(1)})\right)\right)= L_p(f_{/K}\otimes\chi,  {\Sigma}^{(1)}_{\rm cc})\big{\vert}_{\Gamma_\ac}\,.$$

\item[i)] In the setting of (B), the required divisibility is \cite[Theorem 1.1(i)]{BLForum}, whereas in the situation of (C), the asserted equality follows from \cite[Theorem 3.36]{skinnerurbanmainconj} combined with Corollary~\ref{cor_big_BSD_step2}(i) (which we apply in light of our observations at the start of this proof). In the remaining case when (A) holds true, this is a restatement of Corollary~\ref{cor_main_conj_along_Gamma}(i) applied with $\Gamma=\Gamma_\ac$.
\item[ii)] The proof of this portion proceeds in a manner identical to Part (i).
\end{proof}


\subsubsection{Anticyclotomic Iwasawa Theory in the indefinite case}
\label{subsubsec_split_indefinite_ord} Throughout in \S\ref{subsubsec_split_definite_ord}, we assume that $N_f^-$ is a square-free product of an even number of primes. We will explain how to obtain unconditional versions of Corollary~\ref{cor_big_BSD_step2} in this set up. 

\begin{theorem}
\label{thm_ordinary_Delta1_Selmer_torsion_via_horizontal_ES_indefinite}
Suppose $f \in S_{k_f+2}(\Gamma_0(N_f))$ is a cuspidal eigen-newform and $\chi$ is a ring class character of $K$. Let $\f$ be the unique branch of the Hida family that admits $f^\alpha$ as a specialization in weight $k_f+2$. Suppose that $\overline{\rho}_{\f}$ is $p$-distinguished, as well as that $\LL_\f$ is a regular ring. Let us choose a morphism $u$ as in \S\ref{subsubsec_local_properties_CMgg} and define the correction factor $\mathscr{C}^{\ac}(u)$ as in Definition~\ref{defn_C_ac_u}.
\item[i)] Assume at least one of the following holds:
\begin{itemize}
    \item[(A)] $\chi^2\neq \mathds{1}$ and the hypothesis \ref{item_tau} holds true.
    \item[(B)] \ref{item_tau}, \ref{item_HSS}, \ref{item_HnEZ} and \ref{item_HDist} are valid.
\end{itemize} 
Then,
\begin{itemize}
    \item[i.1)]  Both $\LL_{\cO}(\Gamma_\ac)$-modules $\widetilde{H}^1_{\rm f}(G_{K,\Sigma},\TT_{f,\chi}^\ac;\Delta^{(1)})$ and $\widetilde{H}^2_{\rm f}(G_{K,\Sigma},\TT_{f,\chi}^\ac;\Delta^{(1)})$ have rank one.
    \item[i.2)] We have a containment
    $$ \LL_{L}(\Gamma_\ac)\cdot \mathscr{C}^{\ac}(u)\,\partial_\cyc^* \left({\rm Tw}\left(L_p(f_{/K}\otimes\chi,  {\Sigma}^{(1)})\right){\big{\vert}_{\Gamma_K}}\right)\,\, \subset \,\,{\rm Reg}_{\TT_{f,\chi}^\ac}\cdot\Char_{\LL_{\cO}(\Gamma_\ac)}\left( \widetilde{H}^2_{\rm f}(G_{K,\Sigma},\TT_{f,\chi}^\ac;\Delta^{(1)})_{\rm tor}\right)\otimes\QQ_p$$
      where we can take $\mathscr{C}^{\ac}(u)=(1)$ in the setting of (B).
\end{itemize}

\item[ii)] Assume that one of the hypothesis (A)or  (B) holds.
Then,
\begin{itemize}
    \item[ii.1)] Both $\LL_{\f}(\Gamma_\ac)$-modules  $\widetilde{H}^1_{\rm f}(G_{K,\Sigma},\TT_{\f,\chi}^\ac;\Delta^{(1)})$ and $\widetilde{H}^2_{\rm f}(G_{K,\Sigma},\TT_{\f,\chi}^\ac;\Delta^{(1)})$ have rank one.
    \item[ii.2)] We have a containment
    $$(\LL_{\f}(\Gamma_\ac)\otimes_{\ZZ_p}\QQ_p) \mathscr{C}^{\ac}(u)\,H_{\f}\, \partial_\cyc^* \left({\rm Tw}_\f\left(L_p(\f_{/K}\otimes\chi,  \mathbb{\Sigma}^{(1)})\right){\big{\vert}_{\Gamma_K}}\right)\,\subset\, {\rm Reg}_{\TT_{\f,\chi}^\ac}\cdot\Char_{\LL_{\f}(\Gamma_\ac)}\left( \widetilde{H}^2_{\rm f}(G_{K,\Sigma},\TT_{\f,\chi}^\ac;\Delta^{(1)})_{\rm tor}\right)\otimes\QQ_p$$
  where we can take $\mathscr{C}^{\ac}(u)=(1)$ in the setting of (B).
\end{itemize}
\end{theorem}

One may drop the correction factors $\mathscr{C}^{\ac}(u)$ from the statement of Theorem~\ref{thm_ordinary_Delta1_Selmer_torsion_via_horizontal_ES_indefinite} if the hypothesis \ref{item_Ind} concerning the CM branches of Hida families holds. In particular, if \ref{item_RIa} is valid with $a=0$ (which translates in the setting of Theorem~\ref{thm_ordinary_Delta1_Selmer_torsion_via_horizontal_ES_indefinite} to the requirement that $\chi^2\not\equiv \mathds 1 \mod \m_\cO$), then we can take $\mathscr C(u)=(1)$ thanks to Proposition~\ref{prop_uniqueness_of_the_lattice_in_residually_irred_case}. 

\begin{proof}[Proof of Theorem~\ref{thm_ordinary_Delta1_Selmer_torsion_via_horizontal_ES_indefinite}] It follows from the global Euler--Poincar\'e characteristic formulae that $${\rm rank}\,\widetilde{H}^1_{\rm f}(G_{K,\Sigma},\TT_{f,\chi}^\ac;\Delta^{(1)})={\rm rank }\,\widetilde{H}^2_{\rm f}(G_{K,\Sigma},\TT_{f,\chi}^\ac;\Delta^{(1)})$$ 
$${\rm rank}\,\widetilde{H}^1_{\rm f}(G_{K,\Sigma},\TT_{\f,\chi}^\ac;\Delta^{(1)})={\rm rank}\,\widetilde{H}^2_{\rm f}(G_{K,\Sigma},\TT_{\f,\chi}^\ac;\Delta^{(1)}).$$
\item[i)] In the setting of (B), this is \cite[Theorem 1.1(ii)]{BLForum}. 

We explain the proof of i.1) assuming (A). It follows from Corollary~\ref{cor_thm_ordinary_Delta1_Delta2_anticyclo_Selmer_torsion_via_horizontal_ES}(i) that  $${\rm rank}_{\LL_{\cO}(\Gamma_\ac)}\left(\widetilde{H}^1_{\rm f}(G_{K,\Sigma},\TT_{f,\chi}^\ac;\Delta^{(1)})\right)\leq 1\,.$$
To prove that ${\rm rank}_{\LL_{\cO}(\Gamma_\ac)}\left(\widetilde{H}^1_{\rm f}(G_{K,\Sigma},\TT_{f,\chi}^\ac;\Delta^{(1)})\right)= 1$, we need to exhibit a non-zero element of the torsion-free module $\widetilde{H}^1_{\rm f}(G_{K,\Sigma},\TT_{f,\chi}^\ac;\Delta^{(1)})$. The required element is the image ${}^u{\mathbb{BF}}_{f^{\alpha},\chi}^{\ac}$ of the Beilinson--Flach element ${}^u{\mathbb{BF}}_{f^{\alpha},\chi}^\dagger \in H^1_{\calF_+}(K,\TT_{f,\chi}^\dagger)$. The element ${}^u{\mathbb{BF}}_{f^{\alpha},\chi}^{\ac}$ is indeed non-trivial thanks to the twisted version Proposition~\ref{prop_reciprocity_law} (which is also proved in \cite{KLZ2}); see also the proof of Theorem~3.13(ii) in \cite{BLForum}. It also follows from \cite[Theorem 3.11]{BLForum} and the fact that the Coleman map denoted by ${\rm Col}^{(1,\chi)}$ in op. cit. is injective that 
$$\res_{\p}^{(1)}\left({}^u{\mathbb{BF}}_{f^{\alpha},\chi}^{\ac}\right)=0,$$ 
so that ${}^u{\mathbb{BF}}_{f^{\alpha},\chi}^{\ac} \in \widetilde{H}^1_{\rm f}(G_{K,\Sigma},\TT_{f,\chi}^\ac;\Delta^{(1)})$, as required.

The proof of i.2) in the setting of (A) follows on combining  Theorem~\ref{thm_ordinary_Deltai_MC_via_horizontal_ES}(i) with Corollary~\ref{cor_big_BSD_step2}(i).

\item[ii)] The proof of this portion proceeds in a manner identical to Part (i). 
\end{proof}


\appendix


\section{Locally restricted Euler systems}
\label{appendix_sec_ES_main}
Our goal in this appendix is to produce bounds on Selmer groups in terms of locally restricted Euler systems with minimal set of hypotheses.

\subsection{Set up}
\label{appendix_subsec_setup}
Suppose $p>2$ is a prime. Let $E$ be a finite extension of $\Qp$ and $\cO\subset E$ its valuation ring, $\frak{m} \subset \cO$ the maximal ideal and $\varpi\in \frak{m}$ is a fixed uniformizer. Let $S$ denote a finite set of places of $\QQ$ which contains the archimedean place and the prime $p$. Suppose $X$ is a finite dimensional $E$-vector space endowed with a continuous action of $G_{\QQ,S}$ and let $X_\circ\subset X$ denote a $G_{\QQ,S}$-stable $\cO$-lattice. Let us put $\overline{X}=X_\circ/\frak{m}X_\circ$ and call it (by slight abuse) the residual representation of $X$. We will assume that $\overline{X}$ is not the trivial representation. 

Let us also set $X^*(1):=\Hom(X,\QQ_p)(1)$ to denote the Tate dual of $X$; similarly define the lattice $X_\circ^*(1):=\Hom(X_\circ,\ZZ_p)(1)$ and $\overline{X}^*(1):=\Hom(\overline{X},\mu_p)$. Let us also set $W=X\otimes\QQ_p/\ZZ_p$ and similarly define the divisible group $W^*(1)$. Finally, let us define the extension $\Omega/\QQ$ as the fixed field of $\ker\left( G_{\QQ}\lra {\rm Aut}(X)\oplus {\rm Aut}(\mu_{p^\infty})\right)\,.$

Throughout this appendix, $M$ will always denote a power of $p$ and we will put $W_M:=X_\circ/MX_\circ$ and similarly define $W_M^*(1)$.

We will consider the following hypotheses on $X$:
\begin{enumerate}
   \item[\mylabel{item_HypXcirc}{{\bf Hyp($X_\circ$)}}] The following conditions simultaneously hold true:
   \begin{itemize}
    \item[i)]The residual representation $\overline{X}$ is irreducible.
    \item[ii)] There exists $\tau\in G_{\QQ(\mu_{p^\infty})}$ such that $X_\circ/(\tau-1)X_\circ$ is a free $\cO$-module of rank one.
    \item[iii)] $H^1(\Omega/\QQ,\overline{X})=0=H^1(\Omega/\QQ,\overline{X}^*(1))$.
\end{itemize}
      \end{enumerate}

\begin{enumerate}
   \item[\mylabel{item_HypX}{{\bf Hyp($X$)}}]  The following conditions simultaneously hold true: \begin{itemize}
    \item[i)]The representation ${X}$ is irreducible and $\overline{X}^{G_\QQ}=0=(\overline{X}^*(1))^{G_\QQ}$.
    \item[ii)] There exists $\tau\in G_{\QQ(\mu_{p^\infty})}$ such that $X/(\tau-1)X$ is an $E$-vector space of dimension one.
    \item[iii)] $H^1(\Omega/\QQ,\overline{X})=0=H^1(\Omega/\QQ,\overline{X}^*(1))$.
\end{itemize}
      \end{enumerate}

\begin{enumerate}
   \item[\mylabel{item_HypXa}{{\bf Hyp($X,a$)}}] The following conditions hold true for a natural number $a$: \begin{itemize}
    \item[i)]The representation ${X}$ is irreducible and $\overline{X}^{G_\QQ}=0=(\overline{X}^*(1))^{G_\QQ}$.
    \item[ii)] There exists $\tau\in G_{\QQ(\mu_{p^\infty})}$ such that $X_\circ/(\tau-1)X_\circ$ is a free $\cO$-module of rank one.
    \item[iii)] $H^1(\Omega/\QQ,\overline{X})=0=H^1(\Omega/\QQ,\overline{X}^*(1))$.
    \item[iv)] Suppose $\tau$ is as in (ii). The ideal $\frak{m}^a$ annihilates the maximal $G_{\QQ,S}$-stable submodule of $(\tau-1)W$ and of $(\tau-1)W^*(1)$.
\end{itemize}
      \end{enumerate}

\begin{enumerate}
   \item[\mylabel{item_HypXab}{\bf  Hyp($X,a,b$)}] The following conditions hold true for a pair of natural numbers $a,b$: \begin{itemize}
    \item[i)]The representation ${X}$ is irreducible and $\overline{X}^{G_\QQ}=0=(\overline{X}^*(1))^{G_\QQ}$.
    \item[ii)] There exists $\tau\in G_{\QQ(\mu_{p^\infty})}$ such that $X/(\tau-1)X$ is an $E$-vector space of dimension one.
    \item[iii)] $H^1(\Omega/\QQ,\overline{X})=0=H^1(\Omega/\QQ,\overline{X}^*(1))$.
    \item[iv)] Suppose $\tau$ is as in (ii). The ideal $\frak{m}^a$ annihilates the maximal $G_{\QQ,S}$-stable submodule of $(\tau-1)W$ and of $(\tau-1)W^*(1)$.
    \item[v)]  Suppose $\tau$ is as in (ii). The length of $(X_\circ/(\tau-1)X_\circ)_{\rm tor}$ is bounded by $b$.
\end{itemize}
      \end{enumerate}

\begin{remark}
Note that \ref{item_HypXa} is equivalent to \ref{item_HypXab}$_{\vert b=0}$\,; whereas both of the equivalent hypotheses \ref{item_HypXa}$_{\vert a=0}$ and \ref{item_HypXab}$_{\vert a=b=0}$\, follow from \ref{item_HypXcirc}. Any one of these hypotheses also implies \ref{item_HypX}.
\end{remark}

\begin{defn}
\item[i)] Let $N$ denote the product of non-archimedean primes in $S$. Let $\cR$ denote the set of square-free products $p_1\cdots p_s$ with $p_i$ coprime to $N$. For each integer $m=p_1\cdots p_s\in \cR$, let us write $\QQ(m)=\QQ(p_1)\cdots\QQ(p_s^s)$, where we write $\QQ(\ell)$ for the maximal $p$-extension in $\QQ(\mu_\ell)$ for a prime number $\ell$. Let us write $S_m$ for the primes of $\QQ(m)$ lying above the primes in $S$, as well all primes ramified in $\QQ(m)/\QQ$. For $m\in \cR$, let us put $\Delta_m:=\Gal(\QQ(m)/\QQ)$.
\item[ii)] A Selmer structure $\mathcal{F}$ for the pair $(X_\circ,\cR)$ is a choice
$$H^1_{\mathcal{F}}(\QQ(m)_v,X_\circ)\subset H^1(\QQ(m)_v,X_\circ)$$
for every $v\in S_m$. Local Tate duality then determines the dual Selmer structure $\mathcal{F}^*$ on $X_\circ^*(1)$. For any quotient $Y$ of $X_\circ$, we shall denote by $\mathcal{F}$ the Selmer structure on $(Y,\cR)$ induced from $X_\circ\twoheadrightarrow Y$. In this scenario, we define the Selmer groups
$$H^1_{\mathcal{F}}(\QQ(m),Z):=\ker\left(H^1(G_{\QQ,S},Z)\lra \oplus_{v\in S_m}\frac{H^1(\QQ(m)_v,Z)}{H^1_{\mathcal{F}}(\QQ(m)_v,Z)} \right)$$
for $Z=X_\circ, W_M$. We similarly define the dual Selmer groups 
$H^1_{\mathcal{F}^*}(\QQ(m),Z)$ for $Z=X_\circ^{*}(1), W_M^*(1)$.
\item[iii)] The canonical Selmer structure $\mathcal{F}_{\rm can}$ for the pair $(X_\circ,\cR)$ is given on setting for every $m\in \cR$
$$H^1_{\mathcal{F}_{\rm can}}(\QQ(m)_v,X_\circ)=\ker\left(H^1(\QQ(m)_v,X_\circ) \lra H^1(\QQ(m)_v^{\rm ur},X) \right)$$
for every non-archimedean $v \in S_m$ prime to $p$, and
$$H^1_{\mathcal{F}_{\rm can}}(\QQ(m)_v,X_\circ)=H^1(\QQ(m)_v,X_\circ)$$
for every $v$ above $p$.
\item[iv)] Given a Selmer structure $\mathcal{F}$ for the pair $(X_\circ,\cR)$, we write $\mathcal{F}< {\mathcal{F}_{\rm can}}$ to indicate that
$$H^1_{\mathcal{F}}(\QQ(m)_v,X_\circ)=H^1_{\mathcal{F}_{\rm can}}(\QQ(m)_v,X_\circ)$$
for every non-archimedean $v \in S_m$ prime to $p$, and
$$H^1_{\mathcal{F}}(\QQ(m)_v,X_\circ)\subseteq H^1(\QQ(m)_v,X_\circ)\,.$$
\item[v)] A Selmer structure $\mathcal{F}< {\mathcal{F}_{\rm can}}$ is said to be Cartesian if for every $M$ we have:
\begin{itemize}
    \item For every $m\in\cR$ and a $p$-primary integer $M$, the restriction map induces an isomorphism 
    $$H^1_{\mathcal{F}}(\QQ,W_M)\stackrel{\sim}{\lra} H^1_{\mathcal{F}}(\QQ(m),W_M)^{\Delta_m}.$$
    \item For  every $m\in\cR$ and $p$-primary integers $M^\prime,M$ with $M^\prime\mid M$ we have
    $$H^1_{\mathcal{F}}(\QQ(m)_v,W_{M^\prime}) =\ker\left(H^1(\QQ(m)_v,W_{M^\prime})\stackrel{[M/M^\prime]}{\lra}\frac{H^1(\QQ(m)_v,W_{M})}{H^1_{\mathcal{F}}(\QQ(m)_v,W_{M})} \right)$$
    for every $v$ dividing $p$. 
\end{itemize}
\end{defn}
\begin{defn}[Locally restricted Euler system] \label{defn_appendix_locrestES}
\item[i)] Suppose that $\mathbf{c}=\{c_m\}_{m\in \cR}$ is an Euler system for the triple $(X_\circ,\cup_m\QQ(m),N)$ in the sense of \cite[\S II]{rubin00}. We say that $\bf{c}$ is $\mathcal{F}$-locally restricted if for every $m\in \cR$ we have 
$$c_m\in H^1_{\mathcal{F}}(\QQ(m),X_\circ)$$
for a Cartesian Selmer structure $\mathcal{F}<\mathcal{F}_{\rm can}$ on $(X_\circ,\cR)$.
\item[ii)] Suppose that $\mathbf{c}=\{c_m\}_{m\in \cR}$ is an $\mathcal{F}$-locally restricted Euler system for the triple $(X_\circ,\cup_m\QQ(m),N)$. We define 
$${\rm ind}_{\cO}(\mathbf{c},\mathcal{F})\in \ZZ_{\geq 0}\cup \{\infty\}$$
as the largest integer $n$ such that $c_{\QQ}\in H^1_{\mathcal{F}}(\QQ,X_\circ)$ is divisible by $\frak{m}^n$ in $H^1_{\mathcal{F}}(\QQ,X_\circ)/H^1_{\mathcal{F}}(\QQ,X_\circ)_{\rm tor}$ (with the obvious convention that ${\rm ind}_{\cO}(\mathbf{c},\mathcal{F})=\infty$ if $c_{\QQ}\in H^1_{\mathcal{F}}(\QQ,X_\circ)_{\rm tor}$).
\end{defn}
Note that an $\mathcal{F}_{\rm can}$-locally restricted Euler system is none other than an Euler system in the usual sense.  
\begin{remark}
\label{remark_appendix_simplify_defn_index}Suppose that $\mathbf{c}=\{c_m\}_{m\in \cR}$ is an $\mathcal{F}$-locally restricted Euler system for $(X_\circ,\cup_m\QQ(m),N)$.
\item[i)] As in \cite{rubin00}, let us write ${\rm ind}_{\cO}(\mathbf{c})$ to denote the largest integer $n$ such that $c_{\QQ}\in H^1(G_{\QQ,S},X_\circ)$ is divisible by $\frak{m}^n$ in $H^1(G_{\QQ,S},X_\circ)/H^1(G_{\QQ,S},X_\circ)_{\rm tor}$. The exactness of the sequence
\begin{align*}
    {\rm Tor}_1^{\cO}\left(\frac{H^1(G_{\QQ,S},X_\circ)}{H^1(G_{\QQ,S},X_\circ)_{\rm tor}+H^1_{\mathcal{F}}(\QQ,X_\circ)},\cO/\frak{m}^n\right)\lra& \frac{H^1_{\mathcal{F}}(\QQ,X_\circ)}{H^1_{\mathcal{F}}(\QQ,X_\circ)_{\rm tor}+\frak{m}^nH^1_{\mathcal{F}}(\QQ,X_\circ)} \\
     &\lra\frac{H^1(G_{\QQ,S},X_\circ)}{H^1(G_{\QQ,S},X_\circ)_{\rm tor}+\frak{m}^nH^1_{\mathcal{F}}(\QQ,X_\circ)}
\end{align*}
for any natural number $n$ shows that 
\begin{itemize}
    \item ${\rm ind}_{\cO}(\mathbf{c},\mathcal{F}) \leq {\rm ind}_{\cO}(\mathbf{c})$;
    \item if we have 
    $${\rm Tor}_1^{\cO}\left(\frac{H^1(G_{\QQ,S},X_\circ)}{H^1(G_{\QQ,S},X_\circ)_{\rm tor}+H^1_{\mathcal{F}}(\QQ,X_\circ)},\cO/\frak{m}^n\right)=\left(\frac{H^1(G_{\QQ,S},X_\circ)}{H^1(G_{\QQ,S},X_\circ)_{\rm tor}+H^1_{\mathcal{F}}(\QQ,X_\circ)}\right)[\frak{m}^n]=0,$$
    then ${\rm ind}_{\cO}(\mathbf{c},\mathcal{F}) = {\rm ind}_{\cO}(\mathbf{c})$.
\end{itemize}
\item[ii)] Suppose in addition that \ref{item_HypX} holds true (in the main body of the current article, our main results will always have this assumption). Then $H^1(G_{\QQ,S},X_\circ)_{\rm tor}=\{0\}$ and our conclusions (i) simplify to 
\begin{itemize}
    \item ${\rm ind}_{\cO}(\mathbf{c},\mathcal{F}) \leq {\rm ind}_{\cO}(\mathbf{c})$;
    \item if ${H^1(G_{\QQ,S},X_\circ)}/{H^1_{\mathcal{F}}(\QQ,X_\circ)}$ is torsion-free, then ${\rm ind}_{\cO}(\mathbf{c},\mathcal{F}) = {\rm ind}_{\cO}(\mathbf{c})$.
\end{itemize}
Notice in addition that we have an injection
$${H^1(G_{\QQ,S},X_\circ)}/{H^1_{\mathcal{F}}(\QQ,X_\circ)}\hookrightarrow {H^1(\QQ_p,X_\circ)}/{H^1_{\mathcal{F}}(\QQ_p,X_\circ)},$$
so we conclude in this scenario that if the quotient  ${H^1(\QQ_p,X_\circ)}/{H^1_{\mathcal{F}}(\QQ_p,X_\circ)}$
is torsion free, it follows that ${\rm ind}_{\cO}(\mathbf{c},\mathcal{F}) = {\rm ind}_{\cO}(\mathbf{c})$.
\item[iii)] We continue to assume the validity of \ref{item_HypX}. Then, 
$$\frak{m}^{{\rm ind}_{\cO}(\mathbf{c},\mathcal{F})}\mid {\rm Fitt}_{\cO}\left(H^1_{\mathcal{F}}(\QQ,X_\circ)/\cO\cdot c_\QQ\right)\,,$$
with the convention that if $c_\QQ=0$, we read this divisibility as $0=0$.

This divisibility is an equality if $H^1_{\mathcal{F}}(\QQ,X_\circ)$ is (necessarily free) of rank one (in the main body of the current article, this will always be the case whenever we apply the machinery in this appendix).
\end{remark}

\begin{defn}
\item[i)] Given a Selmer structure $\mathcal{F}<\mathcal{F}_{\rm can}$ on $(X_\circ,\cR)$ and $m\in \cR$, we define the Selmer structure $\mathcal{F}_m < \mathcal{F}$ on $X_\circ$ by the requirement that
$$H^1_{\mathcal{F}_m}(\QQ_\ell,X_\circ)=
\begin{cases}
H^1_{\mathcal{F}}(\QQ_\ell, X_\circ) & \hbox{ if } \ell \nmid m\\
0& \hbox{ if } \ell \mid m\,.
\end{cases}$$
We also define the Selmer structure $\mathcal{F}^m$ on $X_\circ$ by the requirement that
$$H^1_{\mathcal{F}^m}(\QQ_\ell,X_\circ)=
\begin{cases}
H^1_{\mathcal{F}}(\QQ_\ell, X_\circ) & \hbox{ if } \ell \nmid m\\
H^1(\QQ_\ell, X_\circ)& \hbox{ if } \ell \mid m\,.
\end{cases}$$
We then write $\mathcal{F}_m^*$ to denote the Selmer structure on $X_\circ^*(1)$ dual to $\mathcal{F}^m$, so that we have
$$H^1_{\mathcal{F}_m^*}(\QQ_\ell,X_\circ^*(1))=
\begin{cases}
H^1_{\mathcal{F}^*}(\QQ_\ell, X_\circ^*(1)) & \hbox{ if } \ell \nmid m\\
0& \hbox{ if } \ell \mid m\,.
\end{cases}$$
\item[ii)] Given $m\in \cR$, let us write $\nu(m)$ for the number of prime divisors of $m$. We define the integer $\chi(\overline{X})$ on setting 
$$\chi(\overline{X})={\min}_{m\in \cR}\left\{\nu(m): H^1_{\mathcal{F}^*_{m}}(\QQ_\ell,\overline{X}^*(1))=0 \right\}\,.$$ 
\end{defn}
\begin{remark}
The proof of \cite[Lemma V.2.3]{rubin00} shows that $\chi(\overline{X})$ is well-defined.
\end{remark}

\begin{defn}
\label{defn_Kolyvagin_derivatives}
Let us fix $M$ to be a power of $p$. Let us write $\cR_M\subset \cR$ for the set defined as in \cite[Definition IV.1.1]{rubin00} with $F=\QQ$.

Given an $\mathcal{F}$-locally restricted Euler system $\mathbf{c}=\{c_m\}_{m\in \cR}$ for $(X_\circ,\cup_m\QQ(m),N)$ and $n\in \cR_M$, we let $\kappa_{r,M}\in H^1(G_{\QQ,S},W_M)$ denote the Kolyvagin's derivative class given as in \cite[Definition IV.4.10]{rubin00}.
\end{defn}
\begin{proposition}
\label{prop_locallyrestrictedES_Kolyvagin_derivative}
Let $M$ be a fixed power of $p$.  Let $\mathbf{c}=\{c_m\}_{m\in \cR}$ be an $\mathcal{F}$-locally restricted Euler system for $(X_\circ,\cup_m\QQ(m),N)$. Assume that $\overline{X}^{G_\QQ}=\{0\}$. Then for any $r\in \cR_M$ we have
$$\kappa_{r,M}\in H^1_{\mathcal{F}^r}(\QQ,W_M)\,.$$
\end{proposition}
\begin{proof}
This is proved in a manner identical to \cite[Theorem 3.25]{kbbesrankr}. The key property is that $\mathcal{F}$ is Cartesian. 
\end{proof}
\subsection{Euler system bounds}
Suppose that $\mathbf{c}=\{c_m\}_{m\in \cR}$ is an $\mathcal{F}$-locally restricted Euler system for $(X_\circ,\cup_m\QQ(m),N)$ for a Cartesian Selmer structure $\mathcal{F}<\mathcal{F}_{\rm can}$, in the sense of Definition~\ref{defn_appendix_locrestES}(i).
\begin{theorem}
\label{thm_appendix_ES_main}
\item[i)] Suppose that \ref{item_HypX} holds true and ${\rm ind}_{\cO}(\bf{c})<\infty$. Then 
$$\left|H^1_{\mathcal{F}^*}(\QQ,W^*(1))\right|<\infty\,.$$
\item[ii)] Suppose that Hyp($X$,a,b) holds true for some natural numbers $a$ and $b$. Then 
$${\rm length}_{\rm \cO}\left(H^1_{\mathcal{F}^*}(\QQ,W^*(1))\right)\leq t+{\rm ind}_{\cO}(\bf{c})$$
where $t=t(a,b,\chi(\overline{X}))\in \ZZ$ is a constant that only depends on $a,b$ and $\chi(\overline{X})$.
\item[iii)] Suppose that Hyp($X$,a) holds true for some natural number $a$. Then 
$${\rm length}_{\rm \cO}\left(H^1_{\mathcal{F}^*}(\QQ,W^*(1))\right)\leq {t}+{\rm ind}_{\cO}(\bf{c})$$
where $t\in \ZZ$ is a constant that only depends on $a$ and $\chi(\overline{X})$.
\item[iv)] Suppose that \ref{item_HypXcirc} holds true for some natural numbers $a$ and $b$. Then 
$${\rm length}_{\rm \cO}\left(H^1_{\mathcal{F}^*}(\QQ,W^*(1))\right)\leq {\rm ind}_{\cO}(\bf{c})\,.$$
\end{theorem}
\begin{remark}
\item[i)] 
When $\mathcal{F}=\mathcal{F}_{\rm can}$, Theorem~\ref{thm_appendix_ES_main}(i) is \cite[Theorem II.2.3]{rubin00} and Theorem~\ref{thm_appendix_ES_main}(iv) is Theorem II.2.2 in op. cit. A careful study of the proofs of these results in \cite[\S V]{rubin00} yields Theorem~\ref{thm_appendix_ES_main}(ii) and Theorem~\ref{thm_appendix_ES_main}(iii) when $\mathcal{F}=\mathcal{F}_{\rm can}$.
\item[ii)] 
Theorem~\ref{thm_appendix_ES_main}(iv) should be compared to the main results of \cite{kbbstark,kbbesrankr,kbbstick} where the Selmer structure $\mathcal{F}$ is denoted by $\mathcal{F}_{\mathcal{L}}$ in op. cit. and corresponds to a rank-one direct summand $\mathcal{L}$ of $H^1(\QQ_p,X_\circ^*(1))$ complementary to the Bloch--Kato submodule. Note that the notion of a locally restricted Euler system is already present in \cite{kbbstick}.
\end{remark}

We will follow the argument in \cite[\S V]{rubin00} very closely for the proof of this theorem, explaining wherever it is required to refine the arguments in op. cit. (by keeping track of the $p$-local properties of the Euler system and Kolyvagin derivatives). For Theorem~\ref{thm_appendix_ES_main}(i), we will modify the argument of Rubin in \cite[\S V.3]{rubin00} suitably. For Theorem~\ref{thm_appendix_ES_main}(ii), we will improve on the argument in \S V.2 of op. cit. Proofs of (iii) and (iv) will easily follow from the proof of (ii). 

We first introduce some notions and establish various auxiliary results which we will rely on in our proof of Theorem~\ref{thm_appendix_ES_main}.

\begin{defn}
For an $\cO$-module $Z$ and $z\in Z$, let us define ${\rm order}(z,Z)={\inf}\{n\geq 0: {\frak{m}^n}z=0\} \in \NN \cup \infty$.
\end{defn}
Note that if $\phi:Z_1\ra Z_2$ is any homomorphism of $\cO$-modules and $z\in Z_1$, then ${\rm order}(z,Z_1)\geq {\rm order}(\phi(z),Z_2)$. If $\phi$ is injective, then ${\rm order}(z,Z_1)={\rm order}(\phi(z),Z_2)$.
\begin{defn}
Suppose $F$ is an algebraic number field and $\kappa\in H^1(\QQ,\ast)$. We write $(\kappa)_F \in H^1(F,\ast)$ for the image of $\kappa$ under the restriction map.
\end{defn}
\begin{defn}
Suppose $\mathcal{F}$ is a Cartesian Selmer structure on $X_\circ$. Let us put $H^1_{\mathcal{F}}(\QQ,X)\subset H^1(G_{\QQ,S},X)$ for the subspace generated by  $H^1_{\mathcal{F}}(\QQ,X_\circ)$ and $H^1_{\mathcal{F}}(\QQ,W)$ for the image of  $H^1_{\mathcal{F}}(\QQ,X)$ under the natural map induced by $X\lra W$. 
As a consequence of our assumption that $\mathcal{F}$ is Cartesian, it follows that 
$$\iota_M^{-1}(H^1_{\mathcal{F}}(\QQ,W))=H^1_{\mathcal{F}}(\QQ,W_M)\,.$$
\end{defn}

\begin{lemma}
\label{lemma_aux_to_choose_k_optimally}
Fix a power $M$ of $p$. Let $\mathcal{F}$ be a Cartesian Selmer structure on $X_\circ$. Assume that {\rm Hyp}$(X)$ holds true. For any positive integer $n$ we have a natural isomorphism 
$$H^1_{\mathcal{F}^*_n}(\QQ,W_M^*(1))^\vee/\frak{m}H^1_{\mathcal{F}^*_n}(\QQ,W_M^*(1))^\vee\stackrel{\sim}{\lra} H^1_{\mathcal{F}^*_n}(\QQ,\overline{X}^*(1))^\vee.$$
\end{lemma}

\begin{proof}
Since $\mathcal{F}$ is a Cartesian Selmer structure on $X_\circ$, so is $\mathcal{F}^n$. Since $(\overline{X}^*(1))^{G_\QQ}=0$ by assumption, it follows from \cite[Lemma 3.5.3]{mr02} that 
\begin{equation}\label{eqn_MR_3_5_3}
    H^1_{\mathcal{F}^*_n}(\QQ,\overline{X}^*(1))\stackrel{\sim}{\lra} H^1_{\mathcal{F}^*_n}(\QQ,W_M^*(1))[\frak{m}]\,.
\end{equation}
The required isomorphism follows on dualizing \eqref{eqn_MR_3_5_3}.
\end{proof}
\begin{corollary}
\label{cor_lemma_aux_to_choose_k_optimally}
We retain the notation in Lemma~\ref{lemma_aux_to_choose_k_optimally}. For every positive integer $n$, one may choose a set $C_n\subset H^1_{\mathcal{F}^*_n}(\QQ,W_M^*(1))$ of size at most $\chi(\overline{X})$ and that generates $H^1_{\mathcal{F}^*_n}(\QQ,W_M^*(1))$.
\end{corollary}

We assume until the end that we have ${\rm ind}_{\cO}({\bf c})<\infty$ for the $\calF$-locally restricted Euler system we are given, since otherwise there is nothing to prove. Let us write $\iota_M$ for natural maps induced from the injection $W_M\hookrightarrow W$. For example, when $\overline{X}^{G_\QQ}=0$, we have an injection
\begin{equation}
    \label{eqn_cartesian_conseq_1_1}
    \iota_M:\,H^1(G_{\QQ,S},W_M)\hookrightarrow H^1(G_{\QQ,S},W)\,.
\end{equation}
\begin{lemma}\label{lemma_rubin_5_1_1}
For any $M$ with ${\rm ord}_{\frak{m}}M\geq {\rm ind}_{\cO}({\bf c},\mathcal{F})$ and assuming the truth of \ref{item_HypX} we have
    $${\rm order}\left(\kappa_{1,M},H^1_{\mathcal{F}}(\QQ,W_M)\right)={\rm order}\left(\iota_M(\kappa_{1,M}),H^1_{\mathcal{F}}(\QQ,W)\right)= {\rm ord}_{\frak{m}}M-{\rm ind}_{\cO}({\bf c},\mathcal{F})\,.$$
\end{lemma}
\begin{proof}
The first equality is immediate from \eqref{eqn_cartesian_conseq_1_1}, whereas the second equality is \cite[Lemma 5.1.1]{rubin00}, on noting that
$${\rm order}\left(\iota_M(\kappa_{1,M}),H^1_{\mathcal{F}}(\QQ,W)\right)={\rm order}\left(\iota_M(\kappa_{1,M}),H^1(G_{\QQ,S},W)\right)\,.$$
\end{proof}
\begin{lemma}
\label{lemma_rubin_lemma_3_1}
Fix a power $M$ of $p$. Suppose $F$ is a finite extension of $\QQ$ such that $G_F$ acts trivially on $W_M$ and $W_M^*(1)$. Assume that \ref{item_HypXab} holds true. If 
$$\kappa\in H^1(G_{\QQ,S},W_M),\,\,\,\,\,\eta\in H^1(G_{\QQ,S},W_M^*(1))$$
then there exists an element $\gamma\in G_F$ such that
\begin{itemize}
    \item[1)] ${\rm order}(\kappa(\gamma\tau),W_M/(\tau-1)W_M)\geq {\rm order}((\kappa)_F,H^1(G_{F,S},W_M))-a$,
    \item[2)] ${\rm order}(\eta(\gamma\tau),W_M^*(1)/(\tau-1)W_M^*(1))\geq {\rm order}((\eta)_F,H^1(G_{F,S},W_M^*(1)))-a$\,.
\end{itemize}
\end{lemma}

\begin{proof}
This is Lemma V.3.1 in \cite{rubin00}. Since we assume throughout $p>2$, the lower bounds here are slightly improved as compared to lower bounds in op. cit. 
\end{proof}

\begin{lemma}
\label{lemma_Rubin_5_2_3}
Fix a power $M$ of $p$. Let $\mathcal{F}$ be a Cartesian Selmer structure on $X_\circ$. Let $\bf{c}$ denote an $\mathcal{F}$-locally restricted Euler system and let $\{\kappa_{r,M}\}_{r\in \cR_M}$ denote Kolyvagin's derivative classes. We assume that \ref{item_HypXab} holds true.

Suppose $C\subset H^1_{\mathcal{F}^*}(\QQ,W_M^*(1))$ is any set of cardinality $k$. There exists a set of rational primes $\fP=\{q_1,\cdots,q_k\} \subset \cR_M$ such that for $1\leq i\leq k$,
\begin{itemize}
    \item[i)] ${\rm Fr}_{q_i}$ is conjugate to the image of $\tau$ in $\Gal(\QQ(W_M)/\QQ)$,
     \item[ii)] setting $n_j=q_1\cdots q_j$ (with $n_0=1$ and $n_k=:n$), we have
     $${\rm order}\left(\res_{q_i}(\kappa_{n_{j-1},M}), H^1_{\rm f}(\QQ_{q_i}, W_M) \right)\geq {\rm order}\left((\kappa_{n_{j-1},M})_\Omega,H^1(\Omega,W_M) \right)-a,$$
      \item[iii)] $\frak{m}^a\left(C\cap H^1_{\mathcal{F}^*_n}(\QQ,W_M^*(1))\right)=0$.
\end{itemize}
\end{lemma}

\begin{proof}
The argument to prove \cite[Lemma V.2.3]{rubin00} works with small alterations in our set up. Let us write $C=\{\eta_1,\cdots,\eta_k\}$. As in op.cit., we shall explain how to choose the primes $q_i\in \cR_M$ inductively so as to satisfy (i), (ii) and 
\begin{equation}\label{eqn_Rubin_Lemma_5_2_3_2}
   \res_{q_i}(\eta)\in H^1_{\rm f}(\QQ_{q_i},W_M^*(1))\,\,\,\, \hbox{for every } \eta \in C\,; 
\end{equation}
\begin{equation}\label{eqn_Rubin_Lemma_5_2_3_3}
    {\rm order}\left(\res_{q_i}(\eta_i),H^1_{\rm f}(\QQ_{q_i},W_M^*(1))\right)\geq  {\rm order}\left((\eta_i)_{\Omega}, H^1(\Omega,W_M^*(1)) \right)-a\,.
\end{equation}
Suppose $1\leq i \leq k$ and $q_1,\cdots,q_{i-1} \in \cR_M$ are chosen to satisfy (i), (ii), \eqref{eqn_Rubin_Lemma_5_2_3_2} and \eqref{eqn_Rubin_Lemma_5_2_3_3}. Let us denote by $N$ the product of non-archimedean primes in $S$. Let us denote the fixed field of $\ker\left( G_{\QQ}\rightarrow {\rm Aut}(W_M)\oplus {\rm Aut}(\mu_{p^M})\right)$ by $F$; note that $F$ is contained in $\Omega$. We shall apply Lemma~\ref{lemma_rubin_lemma_3_1} with this $F$, $\kappa=\kappa_{n_{i-1},M}$ and $\eta=\eta_i$, to find an element $\gamma\in G_F$. On noting $(\kappa_{n_{i-1},M})_F\in H^1(F,W_M)=\Hom(G_F,W_M)$ and $(\eta_i)_F\in H^1(F,W_M)=\Hom(G_F,W_M)=\Hom(G_F,W_M^*(1))$, we may consider the fixed field $F^\prime$ of  
$$\xymatrix{G_F\ar[rrr]^(.4){(\kappa_{n_{i-1},M})_F\oplus (\eta_i)_F}&&& W_M\oplus W_M^*(1)\,.}$$
We choose $q_i\in \cR_M$ as a prime which doesn't divide $Nn_{j-1}$ and whose arithmetic Frobenius in $F^\prime/\QQ$ (for some choice of a prime of $F^\prime$ above $q_i$) equals $\gamma\tau$. Note that Tchebotarev density theorem guarantees the existence of such $q_i$.

Properties (i) and \eqref{eqn_Rubin_Lemma_5_2_3_2} are immediate by the choice of $q_i$. To verify property (ii) and \eqref{eqn_Rubin_Lemma_5_2_3_3}, we note by \cite[Lemma 1.4.7(i)]{rubin00} that evaluation of cocycles at the (arithmetic) Frobenius ${\rm Fr}_{q_i}$ at $q_i$ induces an isomorphism
$$H^1_{\rm f}(\QQ_{q_i},Z)\cong Z/({\rm Fr}_{q_i}-1)Z=Z/(\tau-1)Z$$
for $Z=W_M,W_M^*(1)$. This in turn shows that
\begin{align}
  \notag  {\rm order}\left(\res_{q_i}(\kappa_{n_{j-1},M}), H^1_{\rm f}(\QQ_{q_i}, W_M) \right)&= {\rm order}\left(\res_{q_i}(\kappa_{n_{j-1},M}),W_M/(\tau-1)W_M \right)\\
  \notag &\geq \notag {\rm order}\left((\kappa_{n_{j-1},M})_F, H^1(F,W_M)\right)-a\\
   \label{align_Rubin_5_2_3_1} &\geq {\rm order}\left((\kappa_{n_{j-1},M})_\Omega,H^1(\Omega,W_M) \right)-a,
\end{align}
where the first inequality is thanks to the choice of $q_i$; and also that
\begin{align}
  \notag  {\rm order}\left(\res_{q_i}(\eta_i), H^1_{\rm f}(\QQ_{q_i}, W_M^*(1)) \right)&= {\rm order}\left(\res_{q_i}(\eta_i),W_M^*(1)/(\tau-1)W_M^*(1) \right)\\
  \notag &\geq \notag {\rm order}\left((\eta_i)_F, H^1(F,W_M^*(1))\right)-a\\
   \label{align_Rubin_5_2_3_2} &\geq {\rm order}\left((\eta_i)_\Omega,H^1(\Omega,W_M^*(1)) \right)-a,
\end{align}
where the first inequality is again immediate by the choice of $q_i$. Note that \eqref{align_Rubin_5_2_3_1} is the sought after property (ii) whereas \eqref{align_Rubin_5_2_3_2} is the property \eqref{eqn_Rubin_Lemma_5_2_3_3}.

To conclude the proof of Lemma~\ref{lemma_Rubin_5_2_3}, we need to check property (iii) with the chosen set of primes $\frak{P}=\{q_1,\cdots,q_k\}$. We explain how to use \eqref{eqn_Rubin_Lemma_5_2_3_3} to verify this property. Suppose $\eta=\eta_i \in C\cap H^1_{\mathcal{F}^*_n}(\QQ,W_M^*(1))$, so that $\res_q(\eta_i)=0$ for every $q\in \frak{P}$. In particular, $\res_{q_i}(\eta_i)=0$ and \eqref{eqn_Rubin_Lemma_5_2_3_3} shows that \begin{align}
   \notag \frak{m}^a\eta\in \ker&\left(H^1(\QQ,W_M^*(1))\lra H^1(\Omega,W_M^*(1))\right)\\
   \label{eqn_rubin_5_2_3_4} &=H^1(\Omega/K,W_M^*(1))\\
  \label{eqn_rubin_5_2_3_5}  &=0,
\end{align}
where \eqref{eqn_rubin_5_2_3_4} follows from the inflation-restriction sequence and \eqref{eqn_rubin_5_2_3_5} is part of our hypothesis \ref{item_HypXab}(iii). This concludes the proof of property (iii).
\end{proof}

\begin{defn}
\label{defn_Rubin_5_2_4}
Let $M$ be a power of $p$ and suppose $n\in \cR_M$. We then have an exact sequence
\begin{equation}
    \label{eqn_Rubin_5_2_4}
    0\lra H^1_{\mathcal{F}}(\QQ,W_M)\lra H^1_{\mathcal{F}^n}(\QQ,W_M) \stackrel{\res_{n,W_M}^s}{\lra} \bigoplus_{q\mid n} H^1_{\rm s}(\QQ_q,W_M)
\end{equation}
where we recall that $H^1_{\rm s}(\QQ_q,W_M):=H^1(\QQ_q,W_M)/H^1_{\rm f}(\QQ_q,W_M)$ is the singular quotient and $\res_{n,W_M}^s:=\oplus_{q\mid n}\, \res_q^s$.
\end{defn}

\begin{lemma}
\label{lemma_Rubin_5_2_5}
Suppose $I=\frak{m}^s$ is a non-zero ideal of $\cO$ and $k$ is a positive integer. Suppose $M$ is a power $p$ satisfying ${\rm ord}_{\frak{m}}M\geq s+{\rm ind}_{\cO}(\bf{c},\mathcal{F})$. Let $\mathcal{F}$ be a Cartesian Selmer structure on $X_\circ$. Let $\bf{c}$ denote an $\mathcal{F}$-locally restricted Euler system and let $\{\kappa_{r,M}\}_{r\in \cR_M}$ denote Kolyvagin's derivative classes. We assume that \ref{item_HypXab} holds true.

Suppose that there exists a finite set of primes $\frak{P}=\{q_1,\cdots,q_k\} \subset \cR_M$ such that for $1\leq i\leq k$,
\begin{itemize}
    \item[a)] ${\rm Fr}_{q_i}$ is conjugate to the image of $\tau$ in $\Gal(\QQ(W_M)/\QQ)$,
     \item[b)] setting $n_j=q_1\cdots q_j$ (with $n_0=1$ and $n_k=:n$), we have
     $${\rm order}\left(\res_{q_i}(\kappa_{n_{j-1},M}), H^1_{\rm f}(\QQ_{q_i}, W_M) \right)\geq {\rm order}\left((\kappa_{n_{j-1},M})_\Omega,H^1(\Omega,W_M) \right)-a\,.$$
\end{itemize}
Then,
$${\rm length}_{\cO}\left({\rm coker}\left(\res_{n,W_I}^s\right)\right)\leq {\rm ind}_{\cO}({\bf c},\mathcal{F})+ \frac{k^2+3k}{2}a+ {(k^2+4k)}b      \,.$$
\end{lemma}
\begin{proof}
We shall modify the proof of \cite[Lemma 5.2.5]{rubin00} suitably (under the assumption that $W^{G_\QQ}=0$, which is ensured by the running hypothesis \ref{item_HypXab}) to apply in our set up.

Thanks to \ref{item_HypXab}(i) and \ref{item_HypXab}(i) and (iii), the maps
$$H^1_{\mathcal{F}^n}(\QQ,W_M)\lra H^1(\QQ,W_M) \stackrel{\iota_M}{\lra} H^1(\QQ,W)\lra H^1(\Omega,W)$$
$$H^1(\Omega,W_M)\stackrel{\iota_M}{\lra} H^1(\Omega,W)$$
are all injective for any $n\in \cR_M$. For each $1\leq i\leq k$, we set as in op. cit. 
\begin{align*}
    \frak{d}_i={\rm order}\left(\kappa_{n_i,M},H^1_{\cF^n}(\QQ,W_M)\right)&={\rm order}\left(\kappa_{n_i,M},H^1(\QQ,W_M)\right)={\rm order}\left(\iota_M(\kappa_{n_i,M}),H^1(\QQ,W)\right) \\
    &={\rm order}\left(\iota_M(\kappa_{n_i,M})_{\Omega},H^1(\Omega,W)\right)={\rm order}\left((\kappa_{n_i,M})_{\Omega},H^1(\Omega,W_M)\right)\,.
\end{align*}
By Lemma~\ref{lemma_rubin_5_1_1} and the choice of $M$, 
\begin{equation}
    \label{eqn_Rubin_5_2_5_1}
    \frak{d}_0=\ord_{\frak{m}}M-{\rm ind}_{\cO}({\bf{c}},\mathcal{F})\geq s\,.
\end{equation}
For each $i\geq 1$, 
\begin{align}
    \notag \frak{d}_i \geq {\rm order}&\left(\res_{q_i}^s(\kappa_{n_i,M}),H^1_{\rm s}(\QQ_{q_i},W_M)\right)\\
  \label{eqn_Rubin_5_2_5_2}  &\geq {\rm order}\left(\res_{q_i}(\kappa_{n_{i-1},M}),H^1_{\rm s}(\QQ_{q_i},W_M)\right)-2b\geq \frak{d}_{i-1}-a-2b
\end{align}
where the first inequality is tautological, the second is Theorem 4.5.4 combined with Corollary A.2.6 of op. cit., the final inequality is the assumption (b) of the lemma. Combining \eqref{eqn_Rubin_5_2_5_1} and \eqref{eqn_Rubin_5_2_5_2}, we see that $\frak{d}_i\geq s-i(a+2b)$ for every $i$.

It follows from \cite[Lemma 3.5.3]{mr02} that we have an isomorphism
$$H^1_{\calF^{n_i}}(\QQ,W_I)\stackrel{\sim}{\lra} H^1_{\calF^{n_i}}(\QQ,W_M)[I]$$
induced from the injection $\iota_{I,M}: W_I\hookrightarrow W_M$. Combining this with Proposition~\ref{prop_locallyrestrictedES_Kolyvagin_derivative}, we may find $\overline{\kappa}_i\in H^1_{\calF^{n_i}}(\QQ,W_I)$ such that
$$\cO\cdot\iota_{I,M}(\overline{\kappa}_i)=\frak{m}^{\frak{d}_i-s+i(a+2b)}\kappa_{n_i,M}\,.$$
For every $i\leq k$, we set 
$$A^{(i)}:={\rm span}_{\cO}\{\overline{\kappa}_1,\cdots, \overline{\kappa}_i\} \subset H^1_{\calF^{n_i}}(\QQ,W_I)\subset H^1_{\calF^{n}}(\QQ,W_I)$$
and let $A:=A^{(k)}$. The singular projection map $\res_{q_i}^s$ induces surjections
$$\res_{n}^s(A^{(i)})/\res_{n}^s(A^{(i-1)})\twoheadrightarrow \cO\cdot \res_{q_i}^s(\overline{\kappa}_i)\subset H^1_{\rm s} (\QQ_{q_i},W_I)\,.$$
Thence, 
\begin{align*}
    {\rm length}_{\cO}\left(\res_{n}^s(A^{(i)})/\res_{n}^s(A^{(i-1)}) \right) &\geq {\rm order}\left(\res_{q_i}^s(\overline{\kappa}_i),H^1_{\rm s} (\QQ_{q_i},W_I) \right)\\
    &\geq {\rm order}\left(\res_{q_i}^s(\overline{\kappa}_{n_i,M}),H^1_{\rm s} (\QQ_{q_i},W_M) \right)-(\frak{d}_i-s+i(a+2b))\\
    &\geq s+\frak{d}_{i-1}-\frak{d}_i + (i+1)(a+2b)\,.
\end{align*}
where the second inequality follows from the definition of $\overline{\kappa}_i$ and the third from \eqref{eqn_Rubin_5_2_5_2}. The filtration
$$\res_{n}^s(A)=\res_{n}^s(A^{(k)})\supset \cdots \supset \res_{n}^s(A^{(1)})\supset \res_{n}^s(A^{(0)})=0$$
shows using \eqref{eqn_Rubin_5_2_5_1} and the trivial estimate $\frak{d}_k\leq \ord_{\frak{m}}M$ that
\begin{align}
  \notag  {\rm length}_{\cO}\left( \res_n^s\left(H^1_{\calF^n}(\QQ,W_I)\right)\right)&\geq  {\rm length}_{\cO}\left(\res_n^s(A)\right)\\
   \notag &\geq \sum_{i=1}^k (s+\frak{d}_{i-1}-\frak{d}_i + (i+1)(a+2b))\\
    \notag&=ks+\frak{d}_{0}-\frak{d}_{k}+ \frac{k(k+3)}{2}(a+2b)  \\
    \notag &=ks+\ord_{\frak{m}}M-{\rm ind}_{\cO}({\bf{c}},\mathcal{F})-\frak{d}_{k}+ \frac{k(k+3)}{2}(a+2b)\\
\label{align_Rubin_5_2_5_3}    &\geq ks -{\rm ind}_{\cO}({\bf{c}},\mathcal{F})+ \frac{k(k+3)}{2}(a+2b)\,.
\end{align}
For every $q\in \cR_M$ we have $H^1_{\rm s}(\QQ_q,W_I)=W_I^{{\rm Fr}_q=1}$ by \cite[Lemma 1.4.7(i)]{rubin00}. Hence,
$${\rm length}_{\cO}\left(\oplus_{q\in \frak{P}}H^1_{\rm s}(\QQ_q,W_I)\right)=k\cdot{\rm length}_{\cO}\left(W_I^{\tau=1} \right)=k\cdot{\rm length}_{\cO}\left( W_I/(\tau-1)W_I \right)\leq ks +kb $$
where the final inequality follows as a consequence of Proposition A.2.5 in op. cit. Combining this with \eqref{align_Rubin_5_2_5_3}, we conclude that
$${\rm length}_{\cO}\left({\rm coker}\left(H^1_{\calF^n}(\QQ,W_I)\stackrel{\res_n^s}{\lra}\oplus_{q\mid n} H^1_{\rm s}(\QQ_q,W_I) \right) \right)\leq {\rm ind}_{\cO}({\bf{c}},\mathcal{F}) + \frac{k^2+3k}{2}a+ {(k^2+4k)}b$$
as claimed.
\end{proof}

We are now ready to proceed with the proof of Theorem~\ref{thm_appendix_ES_main}.

\begin{proof}[Proof of Theorem~\ref{thm_appendix_ES_main}] 
We assume in all cases that ${\rm ind}_{\cO}({\bf c})<\infty$ since otherwise there is nothing to prove. 
\item[i)] In this part of the proof, we will largely rely on the argument of Rubin in \cite[\S V.3]{rubin00}. We indicate what portions of the argument that need to be altered to suit our purposes. Recall that we assume the truth of \ref{item_HypXab} in this portion.

Suppose $\eta\in H^1_{\mathcal{F}^*}(\QQ,W_M)$ is an arbitrary element. On applying Lemma~\ref{lemma_rubin_lemma_3_1} with the fixed field $F$ of $\ker\left( G_{\QQ}\lra {\rm Aut}(W_M)\oplus {\rm Aut}(\mu_{p^M})\right)$, the element $\eta \in  H^1_{\mathcal{F}^*}(\QQ,W_M)$ we have fixed and $\kappa=\kappa_{1,M}$, we can choose $\gamma\in G_F$ so that
   \begin{align}
     \notag  {\rm order}(\kappa_{1,M}(\gamma\tau),W_M/(\tau-1)W_M)&\geq {\rm order}((\kappa_{1,M})_F,H^1(G_{F,S},W_M))-a\\
  \label{eqn_rubin_3_1_aux_1}     &\geq {\rm order}((\kappa_{1,M})_\Omega,H^1(\Omega,W))-a
   \end{align}
   \begin{align}
      \notag {\rm order}(\eta(\gamma\tau),W_M^*(1)/(\tau-1)W_M^*(1))&\geq {\rm order}((\eta)_F,H^1(G_{F,S},W_M^*(1)))-a\\
  \label{eqn_rubin_3_1_aux_2}     &\geq {\rm order}((\eta)_\Omega,H^1(\Omega,W^*(1)))-a
   \end{align}
It follows from \ref{item_HypXab}(i) and \ref{item_HypXab}(iii) that the composition
$$ H^1_{\mathcal{F}}(\QQ,W)\lra H^1(G_{\QQ,S},W)\lra H^1(\Omega,W)$$
is injective, thence
$${\rm order}((\kappa_{1,M})_\Omega,H^1(\Omega,W))={\rm order}(\kappa_{1,M},H^1_{\mathcal{F}}(\QQ,W))\,.$$ 
This together with Property (1) above combined with Lemma~\ref{lemma_rubin_5_1_1} yields
\begin{equation}
    \label{eqn_Rubin_eqn_5_16_1}
    {\rm order}(\kappa(\gamma\tau),W_M/(\tau-1)W_M)\geq {\rm ord}_{\frak{m}}M-{\rm ind}_{\cO}({\bf c},\mathcal{F})-a\,.
\end{equation}
Similarly, 
\begin{equation}
    \label{eqn_Rubin_eqn_5_16_2}
    {\rm order}(\eta(\gamma\tau),W_M^*(1)/(\tau-1)W_M^*(1))\geq {\rm order}(\eta,H^1(G_{\QQ,S},W_M^*(1)))-a\,.
\end{equation}
Let $q$ denote the rational prime chosen as in the paragraph following (5.18) in \cite{rubin00}. The argument in the portion from (5.18) until the last three lines of Page 116 in op. cit. applies word by word to conclude that
\begin{equation}
    \label{eqn_Rubin_eqn_5_16_3}
    {\rm order}(\res_q^s(\kappa_{q,M}),H^1_{\rm s}(\QQ_q,W_M))\geq {\rm ord}_{\frak{m}}M-{\rm ind}_{\cO}({\bf c},\mathcal{F})-a-2b
\end{equation}
where $\res_q^s: H^1(G_{\QQ,S},\ast)\lra H^1_s(\QQ_q,\ast)$ is the singular projection; and that
\begin{equation}
    \label{eqn_Rubin_eqn_5_16_4}
    {\rm order}(\res_q(\eta),H^1_{\rm f}(\QQ,W_M^*(1)))\geq {\rm order}(\eta, H^1(G_{\QQ,S},W_M^*(1)))-a\,.
\end{equation}
Moreover, we have as explained in \cite[Page 116]{rubin00}
\begin{equation}
    \label{eqn_Rubin_eqn_116_5}
    {\rm length}_{\cO}\left(H^1_{\rm s}(\QQ,W_M)\right)\leq {\rm ord}_{\frak{m}}M+b\,.
\end{equation}
Applying \cite[Theorem 2.3.4]{mr02} with $\mathcal{G}_1=\mathcal{F}$, $\mathcal{G}_2=\mathcal{F}^q$ and $T=W_M$, we infer that
\begin{equation}
    \label{eqn_MR04_Global_Duality}
    {\rm order}(\res_q(\eta),H^1_{\rm f}(\QQ,W_M^*(1))) \leq {\rm length}_{\cO}\left({\rm coker}\left(H^1_{\mathcal{F}^q}(\QQ,W_M)\stackrel{\res_q^s}{\lra}H^1_{\rm s}(\QQ_q,W_M)\right)\right)\,.
\end{equation}
Moreover, since $\kappa_{q,M}\in H^1_{\mathcal{F}^q}(\QQ,W_M)$, we have
\begin{align}
    \notag {\rm length}_{\cO}\left({\rm coker}\left(H^1_{\mathcal{F}^q}(\QQ,W_M)\stackrel{\res_q^s}{\lra}H^1_{\rm s}(\QQ_q,W_M)\right)\right)\leq {\rm length}_{\cO}&\left(H^1_{\rm s}(\QQ_q,W_M)\right)\\
    \label{eqn_Rubin_eqn_116_6} &-{\rm order}(\res_q^s(\kappa_{q,M}),H^1_{\rm s}(\QQ_q,W_M))\,.
\end{align}
Combining \eqref{eqn_Rubin_eqn_5_16_3}, \eqref{eqn_Rubin_eqn_116_5} \eqref{eqn_MR04_Global_Duality} and \eqref{eqn_Rubin_eqn_116_6}, we conclude that
\begin{equation}
    \label{eqn_Rubin_eqn_116_7}
    {\rm order}(\res_q(\eta),H^1_{\rm f}(\QQ,W_M^*(1)))\leq {\rm ind}_{\cO}({\bf c},\mathcal{F})+a+3b\,.
\end{equation}
Now \eqref{eqn_Rubin_eqn_116_7} with \eqref{eqn_Rubin_eqn_5_16_4} shows that 
\begin{equation}
    \label{eqn_Rubin_eqn_117_1}
   {\rm order}(\eta, H^1(G_{\QQ,S},W_M^*(1)))\leq {\rm ind}_{\cO}({\bf c},\mathcal{F})+2a+3b\,.
  \end{equation}
  This holds true for every $M$ and $\eta \in H^1_{\mathcal{F}^*}(\QQ,W_M^*(1))$. Since $H^1_{\mathcal{F}^*}(\QQ,W^*(1))=\varinjlim_M H^1_{\mathcal{F}^*}(\QQ,W_M^*(1))$, it follows from  \eqref{eqn_Rubin_eqn_117_1} that
$$\frak{m}^{2a+3b+{\rm ind}_{\cO}({\bf c},\mathcal{F})}H^1_{\mathcal{F}^*}(\QQ,W_M^*(1))=0\,.$$
This concludes the proof of Theorem~\ref{thm_appendix_ES_main}(i).


\item[ii)] In this part of the proof, we will largely rely on the argument of Rubin in \cite[\S V.2]{rubin00}. We explain here what portions of the argument we need to modify for our purposes.

In the easier case when $a=0$, the argument we present below applies with slight alteration. We therefore assume henceforth $a\neq 0$. Let $I=\frak{m}^s$ be as in Lemma~\ref{lemma_Rubin_5_2_5} (we will take $s$ as large as necessary). We will apply Lemma~\ref{lemma_Rubin_5_2_3}  on setting $C$ to be the image of $H^1_{\mathcal{F}^*}(\QQ,W_{\frak{m}^{2a}}^*(1))$ under the injective map
$$H^1_{\mathcal{F}^*}(\QQ,W_{\frak{m}^{2a}}^*(1))\lra H^1_{\mathcal{F}^*}(\QQ,W_I^*(1)).$$
It follows from Lemma~\ref{lemma_aux_to_choose_k_optimally} that 
\begin{equation}\label{eqn_bound_on_C_Rubin_5_2_3}
|C|\leq \chi(\overline{X})^{2a}.
\end{equation}
To save ink, let us put $k=|C|$.  Let us choose the set $\frak{P}$ of rational primes as in Lemma~\ref{lemma_Rubin_5_2_3} and define $n$ to be the product of the elements of $\frak{P}$. We then have
$$\frak{m}^aH^1_{\mathcal{F}_n^*}(\QQ,W_{\frak{m}^{2a}}^*(1))=\{0\}\,.$$
Moreover, since the image of the injective map
$$[\frak{m}^a]: H^1_{\mathcal{F}_n^*}(\QQ,W_{\frak{m}^{a}}^*(1))\hookrightarrow  H^1_{\mathcal{F}_n^*}(\QQ,W_{\frak{m}^{2a}}^*(1))$$
is contained in $\frak{m}^aH^1_{\mathcal{F}_n^*}(\QQ,W_{\frak{m}^{2a}}^*(1))=0$, it follows that $H^1_{\mathcal{F}_n^*}(\QQ,W_{\frak{m}^{a}}^*(1))=0$. Since we also have 
$$H^1_{\mathcal{F}_n^*}(\QQ,W_{\frak{m}^{a}}^*(1))\stackrel{\sim}{\lra}H^1_{\mathcal{F}_n^*}(\QQ,W_{I}^*(1))[\frak{m}^{a}]$$
by \cite[Lemma 3.5.3]{mr02}, it follows that \begin{equation}
\label{eqn_all_important_vanishing}
    H^1_{\mathcal{F}_n^*}(\QQ,W_{I}^*(1))=\{0\}\,.
\end{equation}

Using \cite[Theorem 2.3.4]{mr02} with $T=W_I$, $\mathcal{G}_1=\mathcal{F}$ and $\mathcal{G}_2=\mathcal{F}^n$ along with  \eqref{eqn_all_important_vanishing}, we conclude that
\begin{equation}
     \label{eqn_114_1_1}
    {\rm length}_{\cO}\left(H^1_{\mathcal{F}^*}(\QQ,W_{I}^*(1))\right)={\rm length}_{\cO}\left({\rm coker}(\res_n^s)\right)\,.
\end{equation}
   
On combining \eqref{eqn_114_1_1} with the conclusion of Lemma~\ref{lemma_Rubin_5_2_5}, we have
\begin{align}
   \notag     {\rm length}_{\cO}\left(H^1_{\mathcal{F}^*}(\QQ,W_{I}^*(1))\right)&\leq {\rm ind}_{\cO}({\bf c},\mathcal{F})+\frac{k^2+3k}{2}a+ {(k^2+4k)}b
    \\ 
     \label{eqn_114_2} &\leq {\rm ind}_{\cO}({\bf c},\mathcal{F})+(k^2+4k)(a+b)\,.
\end{align}
Note that the right hand side of the inequality \eqref{eqn_114_2} is independent of the choice of $I$. Since \eqref{eqn_114_2} holds for all $I$ and $H^1_{\mathcal{F}^*}(\QQ,W^*(1))=\varinjlim_I H^1_{\mathcal{F}^*}(\QQ,W_{I}^*(1))$,  the proof of Theorem~\ref{thm_appendix_ES_main}(ii) follows once we set
$$t=t(a,b,\chi(\overline{X}))=\left(\chi(\overline{X})^{4a}+4\chi(\overline{X})^{2a}\right)(a+b)\,.$$


\item[iii)] Proof of (ii) shows that $t=t(a,0,\chi(\overline{X}))$ can be chosen in a way that it will only depend on $a$ and $\chi(\overline{X})$. The proof of (iii) follows from (ii) in this case.
\item[iv)] Proof of (ii) shows that one can choose $t=t(0,0,\chi(\overline{X}))=0$ in this case and (iv) follows from (ii).
\end{proof}
\subsection{Examples}
\label{appendix_subsec_examples}
Suppose $f$ is a cuspidal non-CM eigenform as in the main body of our article and let $\xi$ be a Hecke character associated to the imaginary quadratic field $K$, say of infinity type $(\ell_1,\ell_2)$. We recall the theta series $\theta_{\xi}:=\theta(\xi^{-1}|\cdot|^{\ell_1})$. Recall also that $X_{f,\xi}^\circ:=R_f^*\otimes R_{\xi}^*\cong R_f^*\otimes {\rm Ind}_{K/\QQ}\xi_p$, which is a free $\cO$-module of rank $4$ endowed with a continuous action of $G_{\QQ,\Sigma}$, and $X_{f,\xi}:=X_{f,\xi}^\circ\otimes\QQ_p$. In \S\ref{appendix_subsec_examples}, we shall discuss the validity of the hypothesis \ref{item_HypXa} with $X=X_{f,\xi}$.

Throughout \S\ref{appendix_subsec_examples}, we have set  $X=X_{f,\xi}$ and $X_\circ=X_{f,\xi}^\circ$.
\begin{proposition}\label{prop_appendix_example} 
 \item[i)] Suppose $\overline{\rho}_f\vert_{G_K}$ is absolutely irreducible and $\xi\neq \xi^c$. Then  \ref{item_HypXa}(i) is valid.
 \item[ii)] Suppose $\overline{\rho}_f\vert_{G_K}$ is absolutely irreducible and ${\xi}\not\equiv {\xi}^c\mod \frak{m}$. Then \ref{item_HypXcirc}(i) holds true.
\item[iii)] \ref{item_HypXa}(iii) always holds.
\item[iv)] Suppose $\overline{\rho}_f\vert_{G_K}$ is absolutely irreducible and ${\xi}\not\equiv {\xi}^c \mod \frak{m}^{a+1}$. Then \ref{item_HypXa}(iv) holds true.
\item[v)] If \ref{item_tau} holds true, then so does \ref{item_HypXa}(iii).
\end{proposition}

\begin{proof}
\item[i)] Suppose $U$ is a non-trivial $G_{\QQ}$-stable submodule of $X_{f,\xi}$, so that 
$$0<\dim \Hom_{G_{\QQ}}(U,X_{f,\xi})=\dim \Hom_{G_{\QQ}}(U,{\rm Ind}_{K/\QQ}(V_f^*\otimes\xi_p))\,.$$
It follows from Frobenius reciprocity that
$$\dim \Hom_{G_{K}}(U\vert_{G_K}, V_f^*\otimes\xi_p)>0\,.$$
Since $V_f^*\otimes\xi_p$ is irreducible by assumption, it follows that $U\vert_{G_K}\supset V_f^*\otimes\xi_p$. Since $U$ and $V_f^*$ are $G_\QQ$-representations, they are both invariant under complex conjugation. Thence, we must also have $U\vert_{G_K}\supset V_f^*\otimes\xi_p^c$. Since $\xi_p$ and $\xi_p^c$ are distinct and $V_{f}^*\vert_{G_K}$ is irreducible, it follows that $V_f^*\otimes\xi_p$ and $V_f^*\otimes\xi_p^c$ have no common constituents. Thence, 
$U\vert_{G_K}\supset (V_f^*\otimes\xi_p) \oplus (V_f^*\otimes\xi_p^c)$ has dimension $4$ and $U=X_{f,\xi}$. This concludes the proof that $X_{f,\xi}$ is irreducible. 

Let us write $\mathds{1}$ for the free $\cO/\frak{m}$-vector space of dimension one on which $G_\QQ$ acts trivially. We contend to prove that 
$$\Hom_{G_{\QQ}}(\mathds{1},\overline{X}_{f,\xi})=0=\Hom_{G_{\QQ}}(\mathds{1},\overline{X}_{f,\xi}^*(1))\,.$$
Using Frobenius reciprocity again, $\Hom_{G_{\QQ}}(\mathds{1},\overline{X}_{f,\xi})=\Hom_{G_{K}}(\mathds{1}\vert_{G_K},\overline{R}_{f}^*\otimes \xi_p)$ and the latter vector space is trivial since we assumed $\overline{\rho}_f\vert_{G_K}$ is absolutely irreducible. This shows $\Hom_{G_{\QQ}}(\mathds{1},\overline{X}_{f,\xi})=0$; the proof for $\overline{X}_{f,\xi}^*(1)$ is identical.
\item[ii)] The argument in the first paragraph of the proof of (i) applies verbatim under the running hypothesis.
\item[iii)] This follows from the fact that we have $\varepsilon_f\varepsilon_{\theta_\xi}\neq 1$ (which is an immediate consequence of the fact that $N_f$ and the conductor of $\xi$ are coprime to the discriminant of $K$); see the discussion in \cite[\S11.1]{KLZ2} and the final paragraph of \cite[Remark 11.1.3]{KLZ2}.
\item[iv)] Let us set $W_{f,\xi}:=X_{f,\xi}^\circ\otimes\QQ_p/\ZZ_p$ and suppose $U\subset (\tau-1)W_{f,\xi}$ is a $G_\QQ$-stable submodule. It follows from (i) that $U$ has finite cardinality (since any divisible $G_K$-stable submodule of $(\tau-1)W_{f,\xi}$ is the image of a $G_K$-stable subspace of $(\tau-1)X_{f,\xi}$, which is zero). 

For any $s\in \ZZ^+$, let us put $W_{s}:=W_{f,\xi}[\frak{m}^s]$ and $W_{f,s}:=V_{f}^*/R_f^*[\frak{m}^s]$ to ease notation. Let us also write $\xi_s$ for the character $\xi_p\mod \frak{m}^s$ so that $W_{s}:={\rm Ind}_{K/\QQ}(W_{f,s}\otimes \xi_s)$ is a free $\cO/{\frak{m}^s}$-module of rank 4. 

Let us pick $s\geq a+1$ large enough so that $U\subset (\tau-1)W_s$ and write $U_1$ (respectively, $U_2$) for the projection of $U\vert_{G_K}$ to $W_{f,s}\otimes \xi_s$ (respectively, to $W_{f,s}\otimes \xi_s^c$) under the canonical $G_K$-equivariant splitting
$$W_s\vert_{G_K}=(W_{f,s}\otimes \xi_s)\oplus (W_{f,s}\otimes \xi_s^c)\,.$$
We claim that $U_1\subset W_{f,a}\otimes \xi_a$. Suppose on the contrary that $U_1$ contains a vector $w\in  W_{f,a+1}\otimes \xi_{a+1}$ which is not in $ W_{f,a}\otimes \xi_a$. Since $\overline{\rho}_f\vert_{G_K}$ is absolutely irreducible by assumption, it follows that the $G_K$-orbit of $w$ generates $W_{f,a+1}\otimes \xi_{a+1}$. This shows that $U[\frak{m}^{a+1}]\supset W_{f,a+1}\otimes \xi_{a+1}$. Moreover, since $\xi_p \not\equiv \xi_p^c \mod \frak{m}^{a+1}$ and both representations $U[\frak{m}^{a+1}]$ and $W_{f,a+1}$ are stable under complex conjugation, it follows that 
$$U[\frak{m}^{a+1}]\vert_{G_K}=(W_{f,a+1}\otimes \xi_{a+1})\oplus (W_{f,a+1}\otimes \xi_{a+1}^c).$$
This shows that 
\begin{equation}
    \label{eqn_propA21_1}
    {\rm length}_{\cO}\,U[\frak{m}^{a+1}]=4(a+1).
\end{equation}
On the other hand, since $U \subset (\tau-1)W_s$, we have an inclusion $U[\frak{m}^{a+1}]\subset (\tau-1)W_s\cap W_{a+1}$. We will now prove that 
\begin{equation}
    \label{eqn_propA21_2}
    {\rm length}_{\cO}\,(\tau-1)W_s\cap W_{a+1}\leq 3(a+1)
\end{equation}
whenever $s\geq a+1$. This will contradict \eqref{eqn_propA21_1} and complete the proof that $U$ is annihilated by $\frak{m}^a$.

In order to prove \eqref{eqn_propA21_2}, consider the commutative diagram 
$$\xymatrix{0\ar[r]& W_{a+1}\ar[r]\ar[d]_{\tau-1}&W_s\ar[r]^(.4){[p^a]}\ar[d]_{\tau-1}&W_{s-a-1}\ar[r]\ar[d]^{\tau-1}& 0\\
0\ar[r]&W_{a+1}\ar[r]&W_s\ar[r]^(.4){[p^a]}&W_{s-a-1}\ar[r]& 0
}$$
with exact rows. Snake lemma gives rise to the long exact sequence
\begin{align}
\label{eqn_propA21_snake}
    0\lra W_{a+1}^{\tau=1}\lra W_s^{\tau=1} \lra W_{s-a-1}^{\tau=1}\lra W_{a+1}/(\tau-1)W_{a+1} \stackrel{\alpha}{\lra} W_s/(\tau-1)W_s
\end{align}
and $\ker(\alpha)=\left( W_{a+1}\cap (\tau-1)W_s\right)/(\tau-1)W_{a+1}$. In particular,
\begin{equation}
    \label{eqn_propA21_3}
    {\rm length}_{\cO}\,(\tau-1)W_s\cap W_{a+1}={\rm length}_{\cO}\,\ker(\alpha)+{\rm length}_{\cO}\,(\tau-1)W_{a+1}\,.
\end{equation}
Moreover, \eqref{eqn_propA21_snake} shows that ${\rm coker}(\beta)\stackrel{\sim}{\lra}\ker(\alpha)$ and also that 
$${\rm length}_{\cO}\,\ker(\alpha)+{\rm length}_{\cO}\,W_{s}^{\tau=1}={\rm length}_{\cO}\,W_{a+1}^{\tau=1}+{\rm length}_{\cO}\,W_{s-a-1}^{\tau=1}\,.$$
Combining this equality with \eqref{eqn_propA21_3}, we conclude that
\begin{equation}
    \label{eqn_propA21_4}
    {\rm length}_{\cO}\,(\tau-1)W_s\cap W_{a+1}={\rm length}_{\cO}\,W_{a+1}^{\tau=1}+{\rm length}_{\cO}\,(\tau-1)W_{a+1}+{\rm length}_{\cO}\,W_{s-a-1}^{\tau=1}-{\rm length}_{\cO}\,W_{s}^{\tau=1}\,.
\end{equation}
For any positive integer $r$, the exactness of the sequence
$$0\lra W_r^{\tau=1}\lra W_r\stackrel{\tau-1}{\lra}W_r \lra W_r/(\tau-1)W_r$$
shows that ${\rm length}_{\cO}\, W_r^{\tau=1}={\rm length}_{\cO}\, W_r/(\tau-1)W_r$, hence also that
$${\rm length}_{\cO}\, W_r^{\tau=1}+{\rm length}_{\cO}\,  (\tau-1)W_r={\rm length}_{\cO}\, W_r=4r.$$
These facts combined with \eqref{eqn_propA21_4} shows 
\begin{equation}
        \label{eqn_propA21_5}
{\rm length}_{\cO}\,(\tau-1)W_s\cap W_{a+1}=4(a+1)+{\rm length}_{\cO}\,W_{s-a-1}/(\tau-1)W_{s-a-1}-{\rm length}_{\cO}\,W_{s}/(\tau-1)W_s\,.
\end{equation}
Since we have
$$W_r/(\tau-1)W_r=X_{f,\xi}^\circ/(\tau-1,\frak{m}^r)X_{f,\xi}^\circ\cong \cO/\frak{m}^r$$
we infer that
$${\rm length}_{\cO}\, W_r/(\tau-1)W_r=r\,.$$
Noting that $s-a-1\geq 0$ by choice, this combined with \eqref{eqn_propA21_5} yields
$${\rm length}_{\cO}\,(\tau-1)W_s\cap W_{a+1}=4(a+1)+(s-a-1)-s=3(a+1),$$
completing the proof of \eqref{eqn_propA21_2}.

The portion that concerns the largest proper $G_\QQ$-stable submodule of $(\tau-1)W_{f,\xi}^*(1)$ is proved in an identical manner.
\end{proof}


\section{A corrigendum to \cite{BLForum}}\label{appendix:corrige}

There is a small imprecision in the formula presented in \cite[Theorem~2.1]{BLForum}. In this appendix, we present the correct (and more general) formula. What follows  should replace \S2.2 of  op. cit. In particular, the numberings and notation all correspond to op. cit. Recall that $\Sigma^{(1)}$ denotes the set of Hecke characters of $\AA_K^\times$ with infinity type $(r,s)$ satisfying $1-k/2\le r,s\le k/2-1$, where $k$ is the weight of a fixed modular form $f$.

\setcounter{section}{2}\renewcommand\thesection{\arabic{section}}

\setcounter{subsection}{1}

\subsection{$p$-adic $L$-function of Hida and Perrin-Riou}\label{S:padicL}

Let $\Bf_1$ and $\Bf_2$ be two Hida families of tame levels $N_1$ and $N_2$ respectively. Suppose that $N$ is an integer divisible by both $N_1$ and $N_2$ with the same  prime factors as $N_1N_2$. There exists a 3-variable $p$-adic $L$-function $L_p(\Bf_1,\Bf_2,s)$, where $s$ is the cyclotomic variable, c.f. \cite[\S7.7]{KLZ2}.\footnote{The normalization we have chosen here is the one in  \cite{LLZ1} and differs from that in \cite{KLZ2} by a power of $N$. See the discussion in Remark~2.7.5(1) in op. cit.} More precisely, let  $f_1$ and $f_2$ be weight $k$ (resp., weight $l$) specializations of $\Bf_1$ and $\Bf_2$ with $l<k$. Suppose that $f_1$ and $f_2$ are $p$-old and they are the $p$-ordinary stabilizations of the newforms $f_1^\circ$ and $f_2^\circ$, respectively.  Note that if we specialize  $(\Bf_1,\Bf_2)$ at $(f_1,f_2)$, then $L_p(\Bf_1,\Bf_2,s)$ becomes the $p$-adic $L$-function associated to the Rankin--Selberg product of $f_1^\circ$ and $f_2^\circ$ given by \cite[Theorem~2.7.4]{KLZ2}, multiplied by $N^{2s-k-l+2}$. 
When we evaluate the latter at $s=j\in \ZZ$, we have
\[
L_p(f_1,f_2,j)=\frac{\left\langle f_1^{c},e_{\rm ord}\left(f_2^{[p]}\times\cE_{1/N}(j-l,k-1-j)\right)\right\rangle_{N,p}}{\langle f_1,f_1\rangle_{N,p}},
\]
where $f_1^{c}$ is the conjugate form given as in \cite[\S2.2]{loeffler18}, $f_2^{[p]}$ is the $p$-depletion of $f_2$, $\langle\sim,\sim\rangle_{N,p}$ denotes the Petersson inner product  at level $\Gamma_1(N)\cap \Gamma_0(p)$ and $\cE_{\alpha}(\phi_1,\phi_2)$ denotes the $p$-depleted Eisenstein series 
\begin{align}
&\sum_{n\ge 1,p\nmid n}\left(\sum_{0<d|n}\phi_1(d)\phi_2(n/d)\left[e^{2\pi id/N}-\phi_1\phi_2(-1)e^{-2\pi id/N}\right]\right)q^n\notag\\
=&\sum_{n\ge 1,p\nmid n}\left(\sum_{d|n}\sign(d)\phi_1(d)\phi_2(n/d) e^{2\pi id/N}\right)q^n\label{eq:eisenstein}
\end{align}
whenever $\phi_1$ and $\phi_2$ are two characters on $\Zp^\times$ and $\alpha\in\frac{1}{N}\ZZ/\ZZ$, as given by \cite[Definition~5.3.1]{LLZ1}. Note that we write our characters additively here (so an integer $j$ that appears as an argument of a $p$-adic $L$-function stands for its evaluation under the character $\chi_\cyc^j$).  We recall from \cite[Theorem~2.7.4]{KLZ2} (see also Remark~2.7.5(1) in op. cit.) that this $p$-adic $L$-function has the following interpolation formula. Let  $\alpha_i$ and $\beta_i$ be the roots to the Hecke polynomial of $f_i^\circ$ at $p$, with $\alpha_i$ being the unit root. 
If $j$ is an integer such that $l\le j\le k-1$ and $\chi$ is a finite character on $\Gamma_\cyc$ of conductor $p^n$, then
\begin{equation}\label{eq:interpolationHida}
L_p(f_1,f_2,j+\chi)=\frac{\cE(f_1,f_2,j+\chi)}{\cE(f_1)\cE^*(f_1)}\times\frac{i^{k-l}N^{2j-k-l+2}\Gamma(j)\Gamma(j-l+1)L(f_1^\circ,f_2^\circ,\chi^{-1},j)}{2^{2j+k-l}\pi^{2j+1-l}\langle f_1^\circ,f_1^\circ\rangle_{N_{1}}},
\end{equation}
where $\cE(f_1)=1-\beta_1/p\alpha_1$, $\cE^*(f_1)=1-\beta_1/\alpha_1$ and
\[
\cE(f_1,f_2,j+\chi)=\begin{cases}
\left(1-\frac{p^{j-1}}{\alpha_1\alpha_2}\right)\left(1-\frac{p^{j-1}}{\alpha_1\beta_2}\right)\left(1-\frac{\beta_1\alpha_2}{p^j}\right)\left(1-\frac{\beta_1\beta_2}{p^j}\right)&\text{if $\chi$ is trivial,}\\
\tau(\chi)^2\cdot \left(\frac{p^{2j-2}}{\alpha_1^2\alpha_2\beta_2}\right)^n&\text{if $\chi$ is of conductor $p^n>1$.}
\end{cases}
\]
Recall that $\ff$ is a modulus of $K$ with $(p,\ff)=1$. We write $H_{\ff p^\infty}$ for the ray class group of $K$ of conductor $\ff p^\infty$. As in \cite[\S6.2]{LLZ2}, we define
\[
\Theta=\sum_{\fa:(\fa,p)=1}[\fa]q^{N_{K/\QQ}(\fa)}\in\Lambda(H_{\ff p^\infty})[[ q]],
\]
where $\fa$ runs over ideals of $K$ and $[\fa]$ denotes the corresponding element of $H_{\ff p^\infty}$ via Artin reciprocity map. Given any character $\omega$ of $H_{\ff p^\infty}$, $\Theta(\omega)$ is then the $p$-depleted theta series attached to $\omega$. Note that its level divides $N_{K/\QQ}(\ff)\cdot{\rm disc}(K/\QQ)\cdot p^\infty$. On replacing $N$ by the lowest common multiple of the level of $f$ and $N_{K/\QQ}(\ff)\cdot{\rm disc}(K/\QQ)$ if necessary, we define a 2-variable $p$-adic $L$-function 
\[
L_p(f/K,\Sigma^{(1)})\in\Lambda_L(H_{\ff p^\infty}):=\mathfrak{o}[[H_{\ff p^\infty}]]\otimes L,
\]
which assigns a character $\omega$ of $H_{\ff p^\infty}$ the value
\[
L_p(f/K,\Sigma^{(1)})(\omega)=\frac{\left\langle( f^\lambda)^c,e_\ord(\Theta(\omega)\times\cE_{1/N}(k/2-1-\omega_\QQ,k/2-1))\right\rangle_{N,p}}{\langle  {f^\lambda}, {f^\lambda} \rangle_{N,p}} ,
\]
where $ f^\lambda$ is the ordinary $p$-stabilization of $f$ and $\omega_\QQ$ is the character given by the composition of $\omega$ with $\Zp^\times\hookrightarrow (\cO_K{\otimes}\Zp)^
\times\rightarrow H_{\ff p^\infty}$. 
 Suppose that the theta series of $\omega$, which we denote by $g$, is the ordinary $p$-stabilization of a CM modular form. If $f^\lambda$ and  $g$ vary inside  Hida families  $\Bf_1$ and $\Bf_2$ respectively, then we recover the Hida $p$-adic $L$-function $L_p(\Bf_1,\Bf_2,s)$. More specifically,  we have the formula
 \[
 L_p(f^\alpha,g,k/2)=L_p(f/K,\Sigma^{(1)})(\omega).
 \]

\begin{theorem}
\label{thm:sigma1interpolationformula}
Let $\psi$  be a Hecke character of infinity type $(a,b)$ so that $\psi\chi_\cyc^{-k/2}\in \Sigma^{(1)}$ with conductor dividing $\ff p^\infty$. Let  $\psi_p$ be the $p$-adic avatar of $\psi$. If the conductor of $\psi$ is coprime to $p$, then
  \[ L_p(f/K, \Sigma^{(1)})(\psi_{p}\cdot \chi_\cyc^{-k/2}) = \frac{\cE(f,\psi,0)}{\cE(f)\cE^*(f)}\times\frac{i^{k-b+a-1}N^{b+a-k+1}\Gamma(a)\Gamma(b)}{2^{a+b+k-1}\pi^{a+b}}\times \frac{L(f/K, \psi, 0)}{ \langle f, f\rangle_{N}}  ,\]
  where  $\cE(f)=1-\lambda'/p\lambda$,  $\cE^*(f)=1-\lambda'/\lambda$,  $\lambda$  and $\lambda'$ are the roots of the Hecke polynomial of $f$ at $p$, with the former being the unit root and
  \[
  \cE(f,\psi,j)=\prod_{\fq\in\{\fp,\fp^c\}}\left(1-\frac{p^{j-1}}{\lambda\psi(\fq)}\right)\left(1-\frac{\lambda'\psi(\fq)}{p^{j}}\right).
  \]
  If the $p$-primary part of the conductor of $\psi$ is given by $\fp^m(\fp^c)^n$ with $m+n\ge 1$, then the value of the $p$-adic $L$-function at $\psi_p\cdot \chi_\cyc^{-k/2}$ for $1\le j\le k-1$ is given by
  \[\frac{p^{a(m+n)}\tau(\psi_p)}{\lambda^{m+n}\cE(f)\cE^*(f)}\times\frac{i^{k-b+a-1}N^{a+b-k+1}\Gamma(a)\Gamma(b)}{2^{a+b+k-1}\pi^{a+b}}\times\frac{ L(f/K,\overline{\psi},0) }{\langle f,f\rangle_N},\]
where $\tau(\psi_p)$ is defined by
$\Theta(\psi_p\chi_\cyc^{-b})|_{b-a+1}\begin{pmatrix}
&-1\\p^{m+n}&
\end{pmatrix}=\tau(\psi_p)\overline{\Theta(\psi_p\chi_\cyc^{-b})}$,\footnote{If $\gamma=\begin{pmatrix}
a&b\\c&d
\end{pmatrix}\in \GL_2^+(\QQ)$, the normalization of the slash operator on a weight $\ell$ modular form $h$ is given by
\[
(h|_\ell \gamma)(z)=\det(\gamma)^{\ell-1}(cz+d)^{-\ell}h(\gamma z).
\]}

When $\psi$ has infinity type $(0,0)$ and $j$ is an integer with $1\le j\le k-1$, we have
\[
L_p(f/K, \Sigma^{(1)})(\psi_{p}\cdot \chi_\cyc^{j-k/2})=  \frac{\cE(f,\psi,j)}{\cE(f)\cE^*(f)}\times\frac{i^{k-1}N^{2j-k+1}\Gamma(j)^2}{2^{2j+k-1}\pi^{2j}}\times \frac{L(f/K, \psi, j)}{ \langle f, f\rangle_{N}}  
\]
if the conductor of $\psi$ is coprime to $p$. If the $p$-primary part of the conductor of $\psi$ is given by $\fp^m(\fp^c)^n$ with $m+n\ge 1$, then it is equal to
 \[\frac{p^{j(m+n)}\tau(\psi_p)}{\lambda^{m+n}\cE(f)\cE^*(f)}\times\frac{i^{k-1}N^{2j-k+1}\Gamma(j)^2}{2^{2j+k-1}\pi^{{2j}}}\times\frac{ L(f/K,\overline{\psi},j) }{\langle f,f\rangle_N}.\]
\end{theorem}

\begin{proof}
We have
\begin{align*}
&\Theta(\psi_p\cdot \chi_\cyc^{-k/2})\times\cE_{1/N}(k/2-1-(\psi_p\chi_\cyc^{-k/2})_\QQ,k/2-1)\\
=&\Theta(\psi_p\chi_\cyc^{-b}\cdot \chi_\cyc^{b-k/2})\times\cE_{1/N}(k/2-1-(-a-b+k),k/2-1)\\
=&\dd^{k/2-b} \Theta(\psi_p\chi_\cyc^{-b})\times \cE_{1/N}(a+b-k/2-1,k/2-1),
\end{align*}
where $\dd=q\frac{d}{dq}$. By \cite[Lemma~6.5(iv)]{hida88} and \cite[Lemma~5.3]{hida85}, we can rewrite the Petersson product in the numerator of the definition of $L_p(f/K, \Sigma^{(1)})(\psi_{p}\cdot \chi_\cyc^{-k/2}) $ as
\begin{align*}
&\left\langle(f^\lambda)^c, e_\ord\circ\Xi \left(\Theta(\psi_p\chi_\cyc^{-b})\times\dd^{k/2-b}\cE_{1/N}(a+b-k/2-1,k/2-1)\right)\right\rangle_{N,p}\\
=&\left\langle(f^\lambda)^c, e_\ord\circ\Xi \left(\Theta(\psi_p\chi_\cyc^{-b})\times\cE_{1/N}(a-1,k-b-1)\right)\right\rangle_{N,p},
\end{align*}
where $\Xi$ denotes the holomorphic projection.
Since $\Theta(\psi_p\chi_\cyc^{-b})$ is the $p$-depletion of a CM modular form (say $g$) of weight $b-a+1$, we have
\[
L_p(f/K, \Sigma^{(1)})(\psi_{p}\cdot \chi_\cyc^{-k/2})=L_p(f,g,b).
\]
The first formula thus follows \eqref{eq:interpolationHida}.

We now suppose that $\Theta(\psi)$ is of level $Np^{s+1}$, and primitive at $p^{s+1}$ with $s\ge0$ (so that $s+1=m+n$). 
We write $W_{Np^*}$ for the Atkin-Lehner involution of level $Np^*$ under the normalization we have chosen.
Following the calculations in \cite[pages 224-225]{hidabook}, we have
\begin{align*}
&\left\langle ( f^\lambda)^c,e_\ord\circ\Xi(\Theta(\psi_p\cdot \chi_\cyc^{-k/2})\times\cE_{1/N}(k/2-1-(\psi_p\chi_\cyc^{-k/2})_\QQ,k/2-1))\right\rangle_{N,p}\\
=&
\frac{p^{(k-1)s}}{\lambda^{s}}\left\langle W_{Np}(f^\lambda)(p^sz), e_\ord\circ\Xi \left(\Theta(\psi_p\cdot \chi_\cyc^{-k/2})\times\cE_{1/N}(k/2-1-(\psi_p\chi_\cyc^{-k/2})_\QQ,k/2-1)\right)\right\rangle_{N,p^{s+1}}.
\end{align*}
As in the previous case, we can rewrite the Petersson product above as
\begin{align*}
\frac{p^{(k-1)s}}{\lambda^{s}}\left\langle W_{Np}(f^\lambda)(p^sz), e_\ord\circ\Xi \left(\Theta(\psi_p\chi_\cyc^{-b})\times\cE_{1/N}(a-1,k-b-1)\right)\right\rangle_{N,p^{s+1}}.
\end{align*}
As in \cite[Appendix A, Step 2]{loeffler18},  this is equal to
\[
\frac{p^{(k-1)s}}{\lambda^{s}}\times\left\langle W_{Np}(f^\lambda)(p^sz), e_\ord\circ\Xi \left(\Theta(\psi_p\chi_\cyc^{-b})\times\tilde E\right)\right\rangle_{N,p^{s+1}},\]
where $\tilde{E}$ is the Eisenstein series given by
\[
\sum_{n\ge 1}q^n\sum_{d|n,p\nmid\frac{n}{d}}d^{b-1}(n/d)^{-b}\left(e^{2\pi id/N}+(-1)^{k-1}e^{-2\pi i d/N}\right).
\]

By duality (as in \cite[Appendix A, Step 3]{loeffler18}), we may rewrite the quantity above as
\begin{align*}
&\frac{p^{s+1}}{\lambda^{s}}\times\left\langle W_{N}(f^\lambda), e_\ord\circ\Xi \left(W_{p^{s+1}}(\Theta(\psi_p\chi_\cyc^{-b}))\times W_{p^{s+1}}(\tilde E)\right)\right\rangle\\
=&\frac{cp^{s+1}}{\lambda^{s}}\times{p^{(s+1)(k-2-b)}}\left\langle \overline{f^\lambda}, e_\ord\circ\Xi \left(W_{p^{s+1}}(\Theta(\psi_p\chi_\cyc^{-b}))\times  E_{1/Np^{s+1}})\right)\right\rangle\\
=&\frac{cp^{(s+1)(k-1-b)}}{\lambda^{s}}\left\langle \overline{f^\lambda}, e_\ord\circ\Xi \left(W_{p^{s+1}}(\Theta(\psi_p\chi_\cyc^{-b}))\times  E_{1/Np^{s+1}})\right)\right\rangle,
\end{align*}
where $c$ is the Atkin-Lehner pseudo-eigenvalue of $f^\lambda$ and $E_{1/Np^{s+1}}$ is as defined in \cite[\S4-5]{LLZ1}. 
By a theorem of Rankin-Selberg and Shimura (c.f. \cite[Theorem~7.1]{kato04}), this is equal to 
\[
\frac{cp^{(s+1)(k-1-b)}}{\lambda^{s}}\times\frac{i^{k-b+a-1}\Gamma(a)\Gamma(b)}{2^{a+b+k-1}\pi^{a+b}(Np^{s+1})^{k-a-b-1}}\times D_{Np}\left(f^\lambda,W_{p^{s+1}} \left(\Theta(\psi_p\chi_\cyc^{-b})\right),b\right),
\]
where $D_{Np}$ is the usual Rankin-Selberg convolution with Euler factors dividing $Np$ removed.
We can simplify this expression as
\begin{align*}
&\frac{cN^{a+b-k+1}p^{a(s+1)}}{\lambda^s}\times\frac{i^{k-b+a-1}\Gamma(a)\Gamma(b)}{2^{a+b+k-1}\pi^{{a+b}}}\times D_{Np}\left(f^\lambda,W_{p^{s+1}} \left(\Theta(\psi_p\chi_\cyc^{-b})\right),j\right)\\
=&\frac{cN^{a+b-k+1}p^{a(s+1)}}{\lambda^s}\times\frac{i^{k-b+a-1}\Gamma(a)\Gamma(b)\tau(\psi_p)}{2^{a+b+k-1}\pi^{{a+b}}}\times D_{Np}\left(f,\overline{\Theta(\psi_p\chi_\cyc^{-b})},b\right).
\end{align*}
Finally, recall that $\langle f^\lambda,f^\lambda\rangle_{N,p}=c\lambda\cE(f)\cE^*(f)\langle f,f\rangle_N$ and that $D_{Np}\left(f,\overline{\Theta(\psi_p\chi_\cyc^{-b})},b\right)=L(f/K,\overline{\psi},0)$ (see \cite[\S3.4]{nekovar95}), hence the result.
\end{proof}
\begin{remark}
The absolute value of $\tau(\psi_p)$ is $p^{-(m+n)/2}$ under our normalization of the slash operator.
\end{remark}

\begin{corollary}\label{cor:SU}
Let $\LSU_{f/K}$ be the $2$-variable $p$-adic $L$-function of Skinner-Urban defined in \cite[Theorem~12.7]{skinnerurbanmainconj} and $\psi$ as in Theorem~\ref{thm:sigma1interpolationformula} with infinity type $(0,0)$. Then, the quotient
\[\frac{\LSU_{f/K}(\psi_p\cdot\chi_\cyc^{k-2})}{L_p(f/K, \Sigma^{(1)})(\psi_{p}\cdot \chi_\cyc^{k/2-1})}
\]
is a $p$-adic unit if  \textup{\textbf{\ref{item_HIm}}}, \textup{\textbf{\ref{item_HDist}}} and \textup{\textbf{\ref{item_HSS}}}  hold.
\end{corollary}
\begin{proof}
The comparison of the two choices of periods is given in the proof of \cite[Theorem~12.7]{skinnerurbanmainconj}. The result follows from the respective interpolation formulae and the fact that the the absolute value of $\tau(\psi_p)$ is $p^{-(m+n)/2}$, whereas that of the Gauss sum of $\psi_p$ is $p^{(m+n)/2}$.
\end{proof}

\begin{corollary}
Let $\LN_{f/K}$ be the $2$-variable $p$-adic $L$-function of Nekovar defined in \cite[\S I.5.10]{nekovar95} and $\psi$ as in Theorem~\ref{thm:sigma1interpolationformula}. Then, the quotient
\[\frac{\LN_{f/K}(\psi_p)}{L_p(f/K, \Sigma^{(1)})(\psi_{p})}
\]
is a $p$-adic unit if  \textup{\textbf{\ref{item_HIm}}}, \textup{\textbf{\ref{item_HDist}}} and \textup{\textbf{\ref{item_HSS}}}  hold.
\end{corollary}
\begin{proof}
This again follows from the respective interpolation formulae. The only non $p$-unit in the interpolation factors for $\LN_{f/K}(\psi_p)$ is $p^{(m+n)(k-1)/2}$. This has the same $p$-adic valuation as $\tau(\psi_p)p^{(m+n)k/2}$ that appears in Theorem~\ref{thm:sigma1interpolationformula} with $j=k/2$.
\end{proof}
\begin{remark}
It is important to note that the $p$-adic units 
$$\frac{\LN_{f/K}(\psi_p)}{L_p(f/K, \Sigma^{(1)})(\psi_{p})} \hbox{\,\,\,and\,\,\,\,} \frac{\LSU_{f/K}(\psi_p\cdot\chi_\cyc^{k/2-1})}{\LN_{f/K}(\psi_p)}$$
do not interpolate as $($the $p$-ordinary stabilization of$)$ $f$ runs through a Hida family.  In other words, Nekovar's $p$-adic $L$-function $\LN_{f/K}$ does not interpolate to a $3$-variable $p$-adic $L$-function.
\end{remark}

\bibliographystyle{amsalpha}
\bibliography{references}

\end{document}